\documentclass[12pt]{amsart} 

\makeatletter
\newcommand\mathcircled[1]{%
	\mathpalette\@mathcircled{#1}%
}
\newcommand\@mathcircled[2]{%
	\tikz[baseline=(math.base)] \node[draw,ellipse,inner sep=1pt] (math) {$\m@th#1#2$};%
}

\usepackage{multicol}
\usepackage{latexsym}
\usepackage{psfrag}
\usepackage{amsmath}
\usepackage{amssymb}
\usepackage{epsfig}
\usepackage{amsfonts}
\usepackage{amscd}
\usepackage{mathrsfs}
\usepackage{graphicx}
\usepackage{enumerate}
\usepackage[autostyle=false, style=english]{csquotes}
\MakeOuterQuote{"}
\usepackage{ragged2e}
\usepackage[all]{xy}
\usepackage{mathtools}

\usepackage{bbold}

\usepackage{mathbbol}

\newlength\ubwidth

\usepackage[dvipsnames]{xcolor}

\usepackage{placeins}
\usepackage{tikz}
\usetikzlibrary{shapes, positioning}
\usetikzlibrary{matrix,shapes.geometric,calc,backgrounds}
\usetikzlibrary{arrows.meta} 
\usetikzlibrary{shapes.geometric, arrows, positioning, decorations.pathreplacing} 

\newtheorem{theorem}{Theorem}

\newtheorem{prop}[theorem]{Proposition}

\newtheorem{defn}[theorem]{Definition}

\usepackage{amsmath,amssymb,amsthm}
\usepackage[dvipsnames]{xcolor}
\usepackage[colorlinks=true,citecolor=RoyalBlue,linkcolor=red,breaklinks=true]{hyperref}

\usepackage{graphicx}
\usepackage{mathdots} 
\usepackage[margin=2.5cm]{geometry}
\usepackage{ytableau}

\usepackage{cancel}

\usepackage{blkarray}
\usepackage{booktabs}

\def\doublestroke#1{\pdfliteral{1 Tr .3 w}#1\pdfliteral{0 Tr 0 w}}

\begin{document}

\title[Generating CMPP partitions for $k=1$]
{Combinatorial construction of known positive series for 
partition classes defined by Capparelli, Meurman, Primc, and Primc 
in the $k$=1 Case}

\author[Kur\c{s}ung\"{o}z]{Ka\u{g}an Kur\c{s}ung\"{o}z}
\address{Ka\u{g}an Kur\c{s}ung\"{o}z, Faculty of Engineering and Natural Sciences, 
    Sabanc{\i} University, Tuzla, Istanbul 34956, Turkey}
\email{kursungoz@sabanciuniv.edu}


      
\date{May 2026}

\dedicatory{Dedicated to Krishnaswami Alladi for his 70th birthday.}

\begin{abstract}
Recently, Capparelli, Meurman, A. Primc and M. Primc 
introduced a class of colored partitions 
which has since been called CMPP partitions. 
This generalized earlier work by M. Primc and \v{S}iki\'{c}, and by Trup\v{c}evi\'{c}.
One main reason why CMPP partitions are significant 
is the authors’ conjecture that the generating functions 
are infinite products in all cases. 
CMPP partitions are true extensions of the partition classes 
in the Rogers-Ramanujan-Gordon identities 
which are defined by difference conditions. 
As such, a natural question is to look for generating functions 
similar to the series side of Andrews-Gordon identities. 
Russell found such bivariate series for one case.  
These evidently positive series overlap with the 
positive series found earlier by Griffin, Ono and Warnaar in the edge cases.  
Russell used symbolic computation in the proofs. 
We will combinatorially interpret 
Russell’s bivariate series 
extending one case of the series due to Griffin, Ono and Warnaar
in a base partition and moves setting, 
and supply some missing cases, as well.
\end{abstract}

\maketitle

An (ordinary) integer partition $\lambda$ is an expression of a 
non-negative integer by an unordered sum of positive integers~\cite{A_thebluebook}.  
\begin{align*}
 n = \lambda_1 + \lambda_2 + \cdots + \lambda_m
\end{align*}
The parts $\lambda_i$'s may be written 
in non-increasing or non-decreasing order, 
and a limited number of zeros may be allowed 
to appear as parts.  
The number $n$ being partitioned is called the \emph{weight}, 
denoted by $\left\vert \lambda \right\vert$
and the number of parts $m$ is called the \emph{length} of the partition.  
For example, 4 has the following five partitions. 
\begin{align*}
  4, \quad
  1+3, \quad 
  2+2, \quad 
  1+1+2, \quad 
  1+1+1+1.
\end{align*}
Another notation for partitions is the frequency notation.  
A given partition is assigned a sequence 
$(f_1, f_2, \ldots )$ where each $f_i$ is 
the number of occurrences of the part $i$ in the partition.  
For example, $1+3$ has $f_1 = f_3 = 1$ 
and $f_i = 0$ for $i=2$ or $i \geq 4$.  
The weight of the partition in this notation is 
$\sum_{i \geq 1} i f_i$.  
\begin{align*}
  1+3 \quad \leftrightarrow \quad (1,0,1,0,0,0,\ldots)
\end{align*}
Multiple copies of parts may be considered.  
Then, parts with the same magnitude 
can be distinguished by assigning them colors, or indices.  

The abbreviated name CMPP partitions is coined 
after the paper by Capparelli, Meurman, A. Primc and M. Primc~\cite{CMPP_original}.  
The four authors elaborated on ideas by M. Primc and \v{S}iki\'{c}~\cite{PS16, PS19}, 
and Trup\v{c}evi\'{c}~\cite{T18}.  
It is convenient to describe CMPP partitions on a visualization~\cite{CMPP_original}.  

Take $\ell$ copies of the positive integers
arranged as below.  
\begin{align*}
  \begin{array}{cccccccc}
    1 & & 3 & & 5 & & 7 & \\
    & 2 & & 4 & & 6 & & 8 \\[5pt]
    & & & \vdots & & & & \\[5pt]
    1 & & 3 & & 5 & & 7 & \\
    & 2 & & 4 & & 6 & & 8 \\
    1 & & 3 & & 5 & & 7 & \\
    & 2 & & 4 & & 6 & & 8
  \end{array} \cdots
\end{align*}
Different apperances of the same number are considered as different parts.
One can associate colors or secondary indices with
the different copies of $\mathbb{Z}_+$.

Then, replace each part by its frequency, and incorporate 
non-negative integers $k_0$, $k_1$, \ldots, $k_\ell$ 
as \emph{initial conditions}.  
\begin{align}
\label{CMPPdiag}
  \begin{array}{cccccccccc}
    k_\ell & & f_1 & & f_3 & & f_5 & & f_7 & \\
    & \cdot & & f_2 & & f_4 & & f_6 & & f_8 \\[5pt]
    & & & & & \vdots & & & & \\[5pt]
    k_2 & & f_1 & & f_3 & & f_5 & & f_7 & \\
    & \cdot & & f_2 & & f_4 & & f_6 & & f_8 \\
    k_1 & & f_1 & & f_3 & & f_5 & & f_7 & \\
    & k_0 & & f_2 & & f_4 & & f_6 & & f_8
  \end{array} \cdots
\end{align}
The initial conditions do not change the weight of the partition.  
An important attribute of CMPP partitions is the integer 
\begin{align*}
 k = k_0 + k_1 + \cdots + k_\ell.  
\end{align*}

Next, consider sums along \emph{downward paths} 
on \eqref{CMPPdiag}.  
An example with $\ell = 3$ is displayed below.  
We take the frequencies vertically aligned with $k_0$ as zeros.  
\begin{center}
  \includegraphics[scale=0.15]{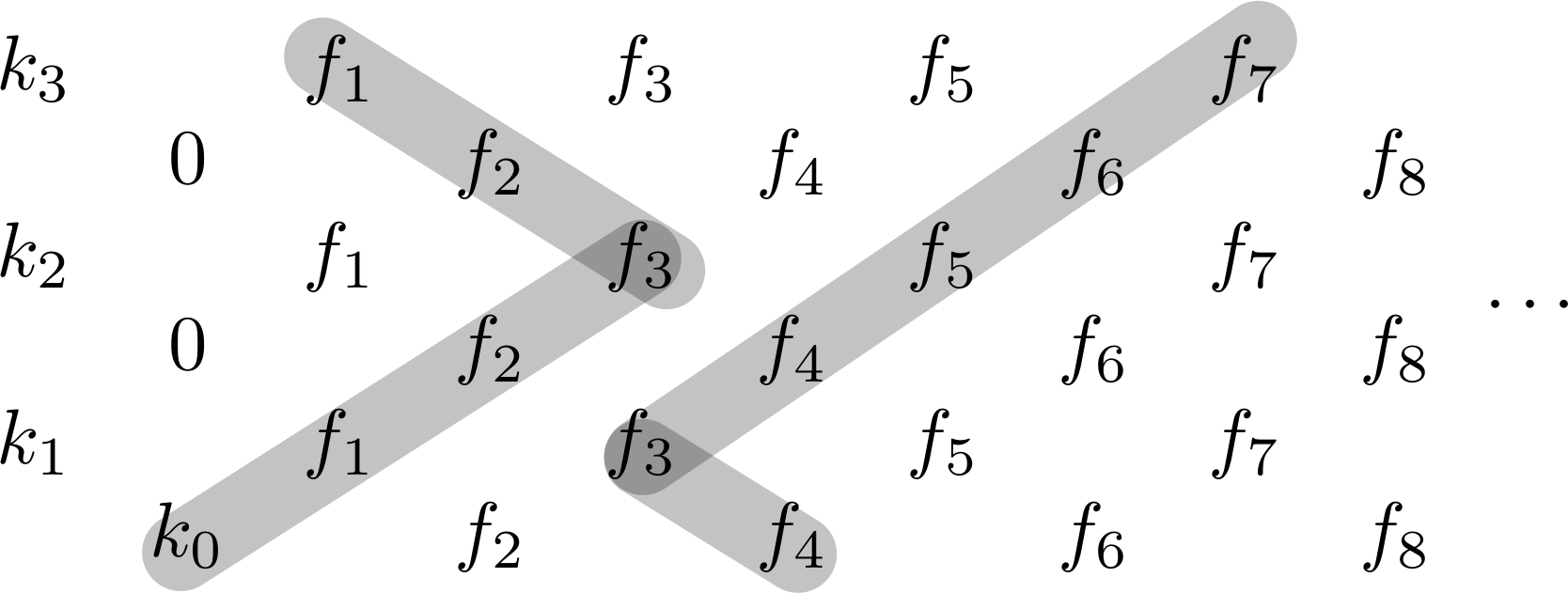}
\end{center}

For chosen and fixed non-negative $k_0$, $k_1$, \ldots $k_\ell$, 
with positive $k = k_0 + k_1 + \cdots + k_\ell$; 
a CMPP partition, 
namely an array of frequencies, 
is \emph{admissible} if each downward sum is at most $k$.  
We will denote admissible CMPP partitions by 
$\displaystyle \doublestroke{\lambda}$, 
$\displaystyle \doublestroke{\mu}$, etc.  

CMPP partitions~\cite{CMPP_original} are discovered in the context of 
affine Lie algebras and their representations~\cite{LW}. 
From an integer partition perspective, 
they are significant primarily due to their alleged generating functions.  
Capparelli, Meurman, A. Primc and M.Primc conjectured that the generating functions 
of CMPP partitions are infinite products.  
In fact, the aforementioned authors have several conjectures 
involving infinite products in the context of 
affine Lie algebras and their representations.  
Some of these infinite products are described in~\cite{CMPP_original}, 
and explicitly stated in~\cite{KRTW}.  
Please see~\cite{Russell_starting_point} for a more extensive literature review, 
and~\cite{CMPP_original, DK, JKS, PS16, PS19, PT25, Warnaar25} for further reading, 
proofs in some cases of, connections with, and extensions of these conjectures.  
The author claims no mastery what so ever 
in the theory of affine Lie algebras and their representations, 
and refers the reader to the contemporary experts in the field 
in addition to reading the above listed citations.

Recently, Russell established 
evidently positive multiple series generating functions 
for CMPP partitions for $k = 1$~\cite{Russell_starting_point}.  
The edge cases of these series were already known 
due to Griffin, Ono, and Warnaar~\cite{GOW}.  
In~\cite{GOW}, evidently positive series are given 
which were later found to be connected to CMPP partitions~\cite{CMPP_original}.  
In particular, the series in~\cite{GOW} cover the cases of
$[k, 0, \ldots, 0]-$ and $[0, \ldots, 0, k]-$admissible CMPP partitions 
for arbitrary positive $k$ with the following notation.  
For chosen and fixed $i = $ 1, 2, \ldots, $\ell$, 
set $k_i = 1$ and the other $k_\cdot$'s $ = 0$.  
Call the number of $[k_0, k_1, \ldots, k_\ell]-$admissible CMPP partitions 
of $n$ with $j$ parts $F(i,j,n)$.  
Then, set 
\begin{align*}
  P_i(z) := P_i(z; q) := \sum_{ j,n \geq 0 } F(i,j,n) \; q^n \; z^j.  
\end{align*}

We recall the standard notation.  
The $q$-Pochhammer symbol for parameter $a$ and base $q$ is defined as 
\begin{align*}
 (a; q)_n = \prod_{j = 1}^n \left( 1 - a q^{j-1} \right), 
 \quad 
 (a; q)_\infty = \lim_{n \to \infty} (a; q_n), 
\end{align*}
and the infinite product converges absolutely 
for $\left\vert q \right\vert < 1$~\cite{GR}.  

\begin{theorem}[\cite{Russell_starting_point} and partly~\cite{GOW}]
\label{thmRussellMain}
  For arbitrary but fixed $i=$ 0, 1, 2, \ldots, $\ell$, 
  set $k_i = 1$ and the rest of the $k_\cdot$'s $=0$.  
  Then, 
  \begin{align}
  \label{eqRussellMain} 
    P_i(z) = \sum_{ n_1, n_2, \ldots, n_\ell \geq 0} 
      \frac{ q^{ N_1^2 + N_2^2 + \cdots 
          + N_\ell^2 + L_{\ell, i}(N_1, N_2, \ldots, N_\ell) } \; z^{N_1} }
        { (q; q)_{n_1} (q; q)_{n_2} \ldots (q; q)_{n_\ell} }, 
  \end{align}
  where $N_s = n_s + n_{s+1} + \cdots + n_\ell$ 
  for $s = $ 1, 2, \ldots, $\ell$, 
  and the linear forms $L_{\ell, i}(N_1, N_2, \ldots, N_\ell)$ follow the pattern below.  
  {\allowdisplaybreaks \begin{align*}
    L_{\ell, 0}(N_1, N_2, \ldots, N_\ell) & = N_1 + N_2 + \cdots + N_\ell \\ 
    L_{\ell, \ell}(N_1, N_2, \ldots, N_\ell) & = N_2 + N_3 + \cdots + N_\ell \\ 
    L_{\ell, 1}(N_1, N_2, \ldots, N_\ell) & = N_3 + N_4 + \cdots + N_\ell \\ 
    L_{\ell, \ell-1}(N_1, N_2, \ldots, N_\ell) & = N_4 + N_5 + \cdots + N_\ell \\ 
    & \vdots
  \end{align*}}
\end{theorem}

The lineup of the indices of the initial conditions 
is different from Russell's~\cite{Russell_starting_point}, 
and the reason for it will be apparent in the proof of Theorem \ref{thmRussellMain} 
in Section \ref{secProofs}.  

Russell first found $q$-contiguous equations satisfied by the series 
\eqref{eqRussellMain}, 
then proved Theorem \ref{thmRussellMain} by showing that the multiple series 
on the right hand side of \eqref{eqRussellMain} satisfies those functional equations.  

CMPP partitions are significant from a partition theoretic perspective 
also because they are natural extensions of Rogers-Ramanujan-Gordon identities 
in at least two ways.  
At this point, we need to make a connection to 
Rogers-Ramanujan-Gordon identities~\cite{A_RRG, RRG}, 
and Andrews-Gordon identities~\cite{AG}.  

\begin{theorem}[Rogers-Ramanujan-Gordon Identities~\cite{RRG}]
\label{thmRRG}
  Choose and fix positive integers $\ell$ and $a$ 
  such that $a \leq \ell+1$.  
  Let $A_{\ell+1,a}(n)$ be the number of (ordinary) partitions of $n$ 
  in which no part is congruent to 0 or $\pm a$ $(\mathrm{mod} (2\ell+3))$.  
  Let $B_{\ell+1,a}(n)$ be the number of (ordinary) partitions of $n$ 
  in which $f_j + f_{j+1} \leq \ell$ for all positive $j$ and $f_1 < a$.  
  Then, 
  \begin{align*}
    A_{\ell+1,a}(n) = B_{\ell+1,a}(n)
  \end{align*}
  for any non-negative integer $n$.  
\end{theorem}

CMPP partitions are Rogers-Ramanujan-Gordon partitions for $\ell = 1$, 
as noted in~\cite{KRTW}.  

\begin{theorem}[Andrews-Gordon Identities~\cite{AG}]
\label{thmAG}
  Choose and fix positive integers $\ell$ and $a$ 
  such that $1 \leq a \leq \ell+1$.  
  Let $b_{\ell+1,a}(m, n)$ be the number of partitions 
  enumerated by $B_{\ell+1,a}(n)$ 
  in which the number of parts is exactly $m$.  
  Then, 
  \begin{align}
  \label{eqAG}
    \sum_{ m, n \geq 0 } b_{\ell+1,a}(m, n) \; z^m \; q^n
    = \sum_{ n_1, n_2, \ldots, n_{\ell} \geq 0} 
      \frac{ q^{ N_1^2 + N_2^2 + \cdots 
          + N_\ell^2 + N_a + N_{a+1} + \cdots + N_{\ell} } 
            \; z^{N_1 + N_2 + \cdots + N_\ell} }
        { (q; q)_{n_1} (q; q)_{n_2} \ldots (q; q)_{n_\ell} }, 
  \end{align}
  where $N_s = n_s + n_{s+1} + \cdots + n_\ell$ 
  for $s = $ 1, 2, \ldots, $\ell$.  
\end{theorem}

Andrews also proved Theorem \ref{thmAG} by showing that 
the multiple series on the right hand side of \eqref{eqAG} satisfy 
certain $q$-contiguous equations.  
Russell's series and Andrews' series are the same for $z = 1$.  
But, the functional equations they satisfy are different.  
Moreover, Andrews' proof of Theorem \ref{thmAG} is 
traceable by hand, while Russell's proof requires 
symbolic computational assistance.  

Russell found evidently positive multiple series 
generating functions for some CMPP partition variants, 
as well~\cite{Russell_starting_point}.  
These variants also appear in~\cite{CMPP_original}.  
We have some missing cases, 
and present the generating functions in two different theorems.  
The first variant is obtained by deleting 
the top row of odd numbers in \eqref{CMPPdiag}. 

\begin{align}
\label{CMPPdiagAlt1}
  \begin{array}{cccccccccc}
    k_\ell & & \times & & \times & & \times & & \times & \\
    & \cdot & & f_2 & & f_4 & & f_6 & & f_8 \\[5pt]
    & & & & & \vdots & & & & \\[5pt]
    k_2 & & f_1 & & f_3 & & f_5 & & f_7 & \\
    & \cdot & & f_2 & & f_4 & & f_6 & & f_8 \\
    k_1 & & f_1 & & f_3 & & f_5 & & f_7 & \\
    & k_0 & & f_2 & & f_4 & & f_6 & & f_8
  \end{array} \cdots
\end{align}

We call the number of $[k_0, k_1, \ldots, k_\ell]-$admissible CMPP partitions 
on diagram \eqref{CMPPdiagAlt1}
of $n$ with $j$ parts $F^{\ast}(i,j,n)$.  
We set 
\begin{align*}
  P^{\ast}_i(z) := P^{\ast}_i(z; q) := \sum_{ j,n \geq 0 } F^{\ast}(i,j,n) \; q^n \; z^j.  
\end{align*}

\begin{theorem}
\label{thmBSLT1}[partly~\cite{Russell_starting_point}]
  For arbitrary but fixed $i=$ 0, 1, 2, \ldots, $\ell$, 
  set $k_i = 1$ and the rest of the $k_\cdot$'s $=0$.  
  Then, 
  \begin{align}
  \label{eqRussellAlt1} 
    P^{\ast}_i(z) = \sum_{ n_1, n_2, \ldots, n_\ell \geq 0} 
      \frac{ q^{ N_1^2 + N_2^2 + \cdots 
          + N_\ell^2 + L^{\ast}_{\ell, i}(N_1, N_2, \ldots, N_\ell) } \; z^{N_1} }
        { (q; q)_{n_1} (q; q)_{n_2} \ldots (q; q)_{n_{\ell-1}} (q^2; q^2)_{n_\ell} }, 
  \end{align}
  where $N_s = n_s + n_{s+1} + \cdots + n_\ell$ 
  for $s = $ 1, 2, \ldots, $\ell$, 
  and the linear forms $L^{\ast}_{\ell, i}(N_1, N_2, \ldots, N_\ell)$ 
  follow the pattern below.  
  {\allowdisplaybreaks \begin{align*}
    L^{\ast}_{\ell, 0}(N_1, N_2, \ldots, N_\ell) 
      & = N_1 + N_2 + \cdots + N_\ell \\ 
    L^{\ast}_{\ell, \ell}(N_1, N_2, \ldots, N_\ell) 
      & = N_2 + N_3 + \cdots + N_\ell \; + \; N_1 \\ 
    L^{\ast}_{\ell, 1}(N_1, N_2, \ldots, N_\ell) 
      & = N_3 + N_4 + \cdots + N_\ell \\ 
    L^{\ast}_{\ell, \ell-1}(N_1, N_2, \ldots, N_\ell) 
      & = N_4 + N_5 + \cdots + N_\ell \; + \; N_2 \\ 
    & \vdots 
  \end{align*}}
\end{theorem}

The other variant is obtained by taking 
a horizontal reflection of \eqref{CMPPdiag} 
\begin{align}
\label{CMPPdiagAlt2Prep}
  \begin{array}{cccccccccc}
    & k_0 & & f_2 & & f_4 & & f_6 & & f_8 \\
    k_1 & & f_1 & & f_3 & & f_5 & & f_7 & \\
    & \cdot & & f_2 & & f_4 & & f_6 & & f_8 \\
    k_2 & & f_1 & & f_3 & & f_5 & & f_7 & \\
    & & & & & \vdots & & & & \\
    & \cdot & & f_2 & & f_4 & & f_6 & & f_8 \\
    k_\ell & & f_1 & & f_3 & & f_5 & & f_7 & 
  \end{array} \cdots
\end{align}
and deleting the top row of even numbers.  
\begin{align}
\label{CMPPdiagAlt2}
  \begin{array}{cccccccccc}
    & k_0 & & \times & & \times & & \times & & \times \\
    k_1 & & f_1 & & f_3 & & f_5 & & f_7 & \\
    & \cdot & & f_2 & & f_4 & & f_6 & & f_8 \\
    k_2 & & f_1 & & f_3 & & f_5 & & f_7 & \\
    & & & & & \vdots & & & & \\
    & \cdot & & f_2 & & f_4 & & f_6 & & f_8 \\
    k_\ell & & f_1 & & f_3 & & f_5 & & f_7 & 
  \end{array} \cdots
\end{align}

We set the notation again.  
We call the number of $[k_0, k_1, \ldots, k_\ell]-$admissible CMPP partitions 
on diagram \eqref{CMPPdiagAlt2}
of $n$ with $j$ parts $F^{\ast\ast}(i,j,n)$.  
We set 
\begin{align*}
  P^{\ast\ast}_i(z) := P^{\ast\ast}_i(z; q) 
  := \sum_{ j,n \geq 0 } F^{\ast\ast}(i,j,n) \; q^n \; z^j.  
\end{align*}

\begin{theorem}
\label{thmBSLT2}[partly~\cite{Russell_starting_point}]
  For arbitrary but fixed $i=$ 0, 1, 2, \ldots, $\ell$, 
  set $k_i = 1$ and the rest of the $k_\cdot$'s $=0$.  
  Then, 
  \begin{align}
  \label{eqRussellAlt2} 
    P^{\ast\ast}_i(z) = \sum_{ n_1, n_2, \ldots, n_\ell \geq 0} 
      \frac{ q^{ N_1^2 + N_2^2 + \cdots 
          + N_\ell^2 + L^{\ast\ast}_{\ell, i}(N_1, N_2, \ldots, N_\ell) } \; z^{N_1} }
        { (q; q)_{n_1} (q; q)_{n_2} \ldots (q; q)_{n_{\ell-1}} (q^2; q^2)_{n_\ell} }, 
  \end{align}
  where $N_s = n_s + n_{s+1} + \cdots + n_\ell$ 
  for $s = $ 1, 2, \ldots, $\ell$, 
  and the linear forms $L^{\ast\ast}_{\ell, i}(N_1, N_2, \ldots, N_\ell)$ 
  follow the pattern below.  
  {\allowdisplaybreaks \begin{align*}
    L^{\ast\ast}_{\ell, 0}(N_1, N_2, \ldots, N_\ell) 
      & = N_1 + N_2 + \cdots + N_\ell \; + \; N_1 \\ 
    L^{\ast\ast}_{\ell, \ell}(N_1, N_2, \ldots, N_\ell) 
      & = N_2 + N_3 + \cdots + N_\ell \\ 
    L^{\ast\ast}_{\ell, 1}(N_1, N_2, \ldots, N_\ell) 
      & = N_3 + N_4 + \cdots + N_\ell \; + \; N_2 \\ 
    L^{\ast\ast}_{\ell, \ell-1}(N_1, N_2, \ldots, N_\ell) 
      & = N_4 + N_5 + \cdots + N_\ell \\ 
    & \vdots
  \end{align*}}
\end{theorem} 

The even moduli generalization of Theorem \ref{thmRRG} 
is found by Bressoud~\cite{Bressoud79}.  
The multiple series in both Theorems \ref{thmBSLT1} and \ref{thmBSLT2} 
are either the series in Bressoud's~\cite{Bressoud80}, 
the series in~\cite{KLRS}, 
or their lookalikes.  

\begin{theorem}[\cite{Bressoud80}]
  Let $b^{\ast}_{\ell+1,a}(m, n)$ be the number of partitions 
  enumerated by $b_{\ell+1,a}(m, n)$ with the additional 
  condition that if $f_j + f_{j+1} = \ell$, 
  then $j f_j + (j+1)f_{j+1} \equiv a-1 (\mathrm{mod} \; 2 )$.  
  Then, 
  \begin{align*}
    \sum_{ m, n \geq 0 } b^{\ast}_{\ell+1,a}(m, n) \; z^m \; q^n
    = \sum_{ n_1, n_2, \ldots, n_{\ell} \geq 0} 
      \frac{ q^{ N_1^2 + N_2^2 + \cdots 
          + N_\ell^2 + N_a + N_{a+1} + \cdots + N_{\ell} } 
            \; z^{N_1 + N_2 + \cdots + N_\ell} }
        { (q; q)_{n_1} (q; q)_{n_2} \ldots (q; q)_{n_{\ell-1}} (q^2; q^2)_{n_\ell} }, 
  \end{align*}
  where $N_s = n_s + n_{s+1} + \cdots + n_\ell$ 
  for $s = $ 1, 2, \ldots, $\ell$.  
\end{theorem}

The goal of this paper is to give a combinatorial construction of 
the multiple series \eqref{eqRussellMain}, \eqref{eqRussellAlt1}, and \eqref{eqRussellAlt2}; 
hence to give an alternative proof of Russell's Theorems, 
i.e. Theorems \ref{thmRussellMain}, \ref{thmBSLT1}, and \ref{thmBSLT2}; 
and supply similar results in some missing cases.  
As noted above, the edge cases of Theorem \ref{thmRussellMain} 
were known before~\cite{GOW}.  
The necessary additional terminology and notation is 
introduced in Section \ref{secPrelim}, 
and the construction is done gradually in Sections \ref{secMoves} and \ref{secProofs}.  
We conclude with some more connections and potential directions 
for future research in Section \ref{secFutureWork}.

\section{Preliminaries}
\label{secPrelim}

On a CMPP diagram \eqref{CMPPdiag}, 
for a chosen and fixed part, 
consider the parts greater than this designated part, 
which are lined up in the diagonal direction containing the designated part.  
We call this the \emph{right arm} of the designated part.  
\emph{Left arm}, \emph{right leg}, \emph{left leg} are defined similarly.  
\begin{center}
  \includegraphics[scale=0.1]{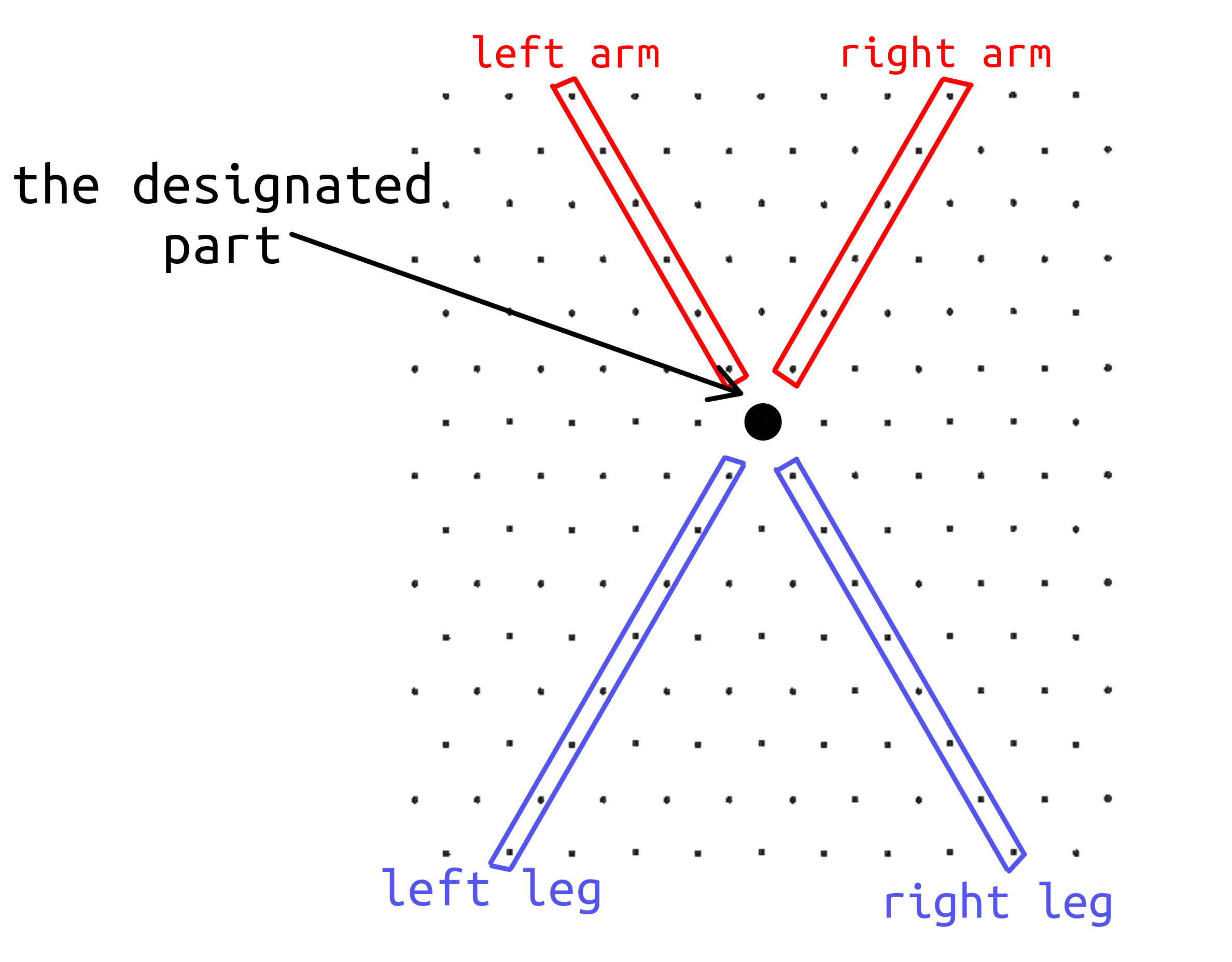}
\end{center}
Observe that these limbs may be empty or degenerate, 
and left arms or left legs may end with an initial condition.  
The initial conditions also will be assigned right arms and right legs.  

We also need to describe the positions of parts 
relative to the designated part, as in the following figure.  
\begin{center}
  \includegraphics[scale=0.1]{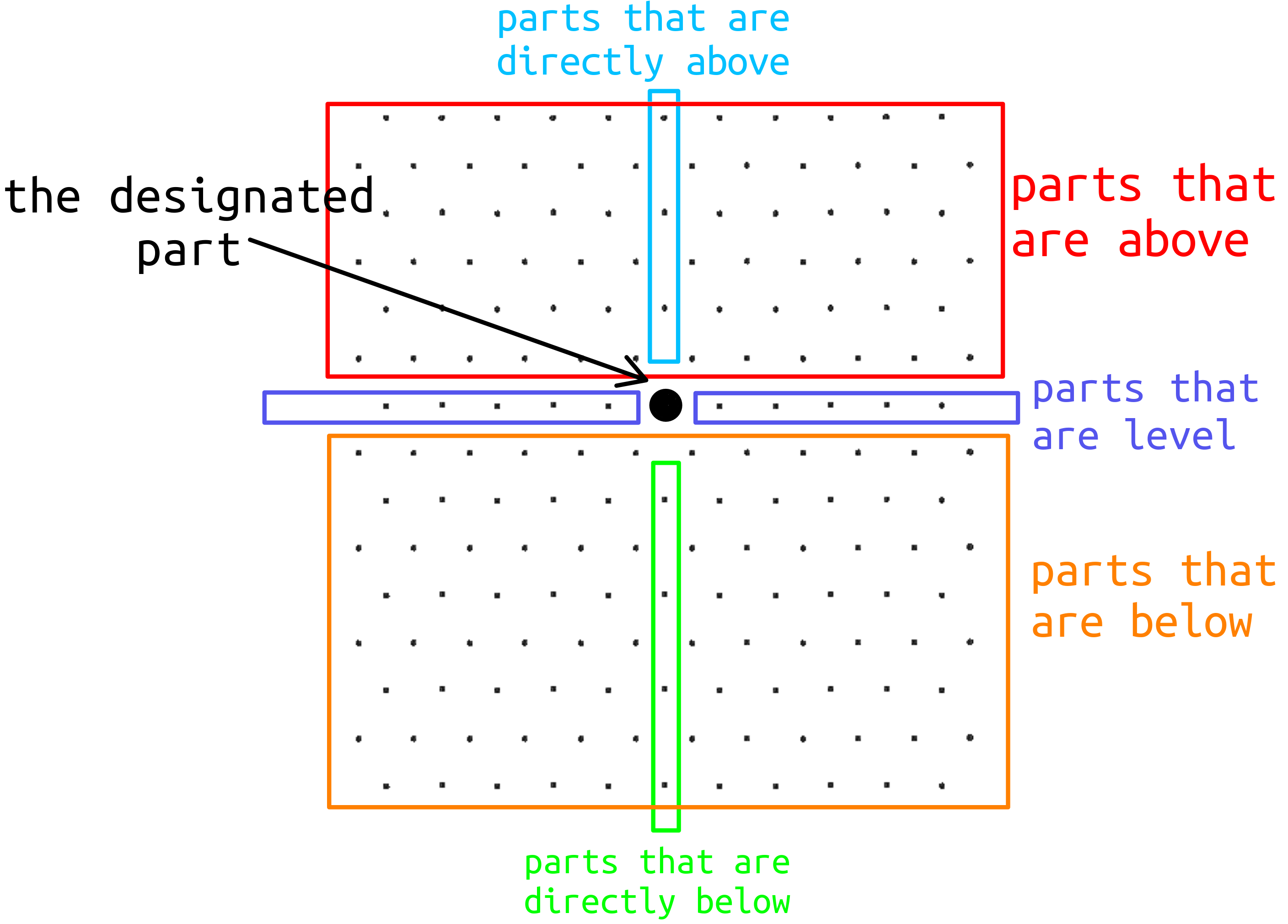}
\end{center}

\begin{defn}
\label{defAbsHeight}  
  The \emph{absolute height} of a designated part 
  is the number of parts directly below it.  
\end{defn}
Notice that the absolute height is an integer between 0 and $\ell-1$.  
The absolute heights of the appearing parts are shown in red in the following example. 
\begin{center}
  \includegraphics[scale=0.2]{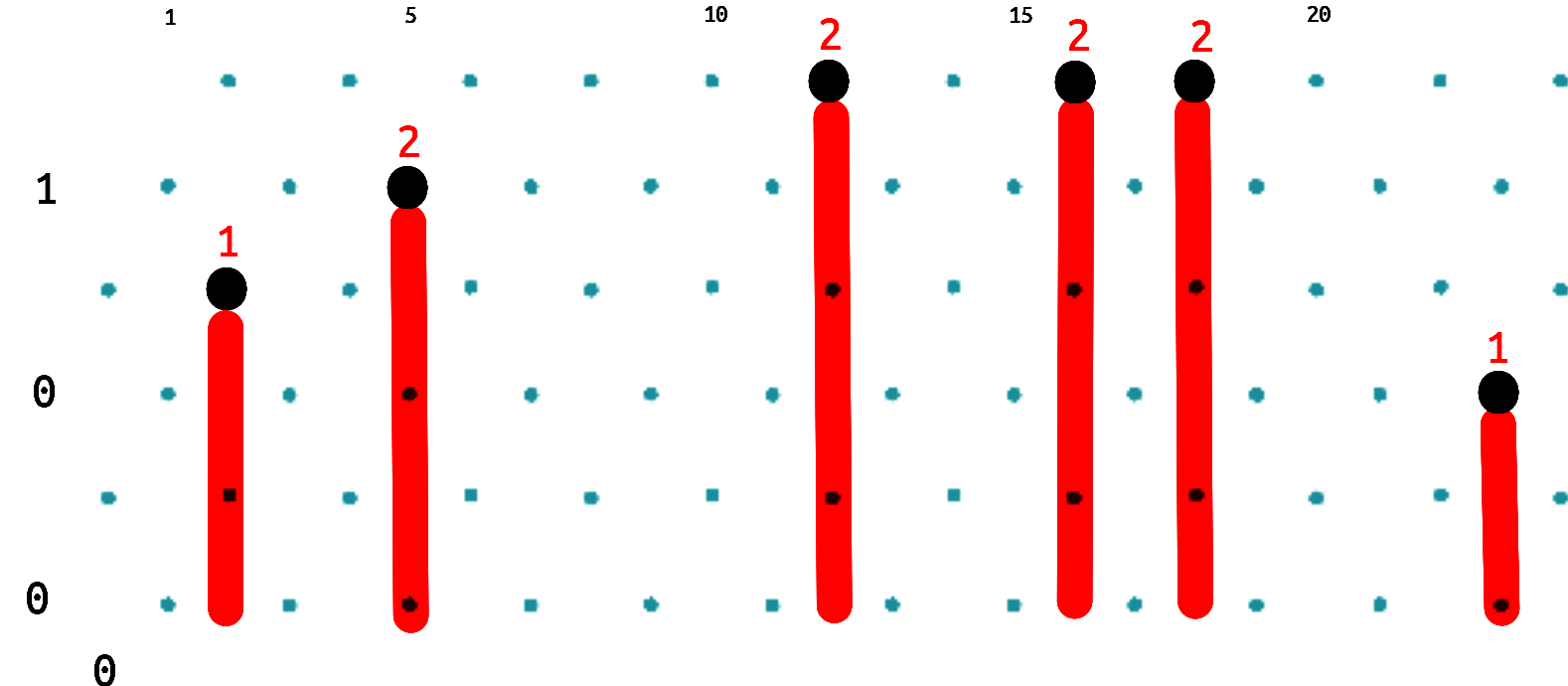}
\end{center}
Absolute heights can also be used to distinguish copies of the same integer.  
For example, the CMPP partition in the figure above contains parts 
\begin{align*}
  2_1, \quad 5_2, \quad 12_2, \quad 16_2, \quad 18_2, \quad 23_1.  
\end{align*}

\subsection{Algorithm for Determining the Relative Heights}
\label{subsecRelHghtDeterm}

Given $\displaystyle{\doublestroke{\lambda}}$, 
an admissible CMPP partition with $k = 1$, 

{\bf (H)} For $h = 0, 1, \ldots, \ell-1$, in this order, 

{\bf (P)} For the parts with no relative heights assigned, 
i.e. the parts still appearing in the possibly folded diagram 
(the folding operation is described in (F))
from smallest to largest, 
do the following: 

{\bf (DL)} Draw the right leg of the preceding part, 
which could be a positive initial condition $k_j = 1$ for $j > 0$;  
and the left leg of the succeding part.  

{\bf (S)} If the number of parts directly below the part at hand 
and above both drawn legs is strictly larger than the current $h$, 
then either skip to the succeeding part if there is one, 
or skip to the next $h$ and start over with the smallest part 
if there are no succeeding parts.  

{\bf (F)} If the number of parts directly below the part $n$ at hand 
and above both drawn legs is equal to the current $h$, 
then set the relative height of $n$ as $h$.  
Imagine two vertical lines, one passing through parts equal to $n-(h+1)$, 
and the other through parts equal to $n+(h+1)$.  
Fold the diagram identifying these vertical lines, 
and concealing parts that are in $[n+(h+1), n-(h+1))]$.  
At this point, $n$ is temporarily out of the picture, 
and its preceding and succeeding parts are now neighbors.  
Continue with the succeding part, if there is any, 
or increment the $h$ by 1 and start over with the smallest part if there are none.  

When all parts are assigned a relative height, stop.  

Let's work on an example.  
Consider the CMPP partition with 
$\ell = 3$, $k_0$ $= k_1$ $= k_2$ $= 0$, $k_3 = 1$, 
and parts $2_1$, $5_2$, $12_2$, $16_2$, $18_2$, $23_1$.  
The indices of parts indicate absolute heights.  
\begin{center}
  \includegraphics[scale=0.2]{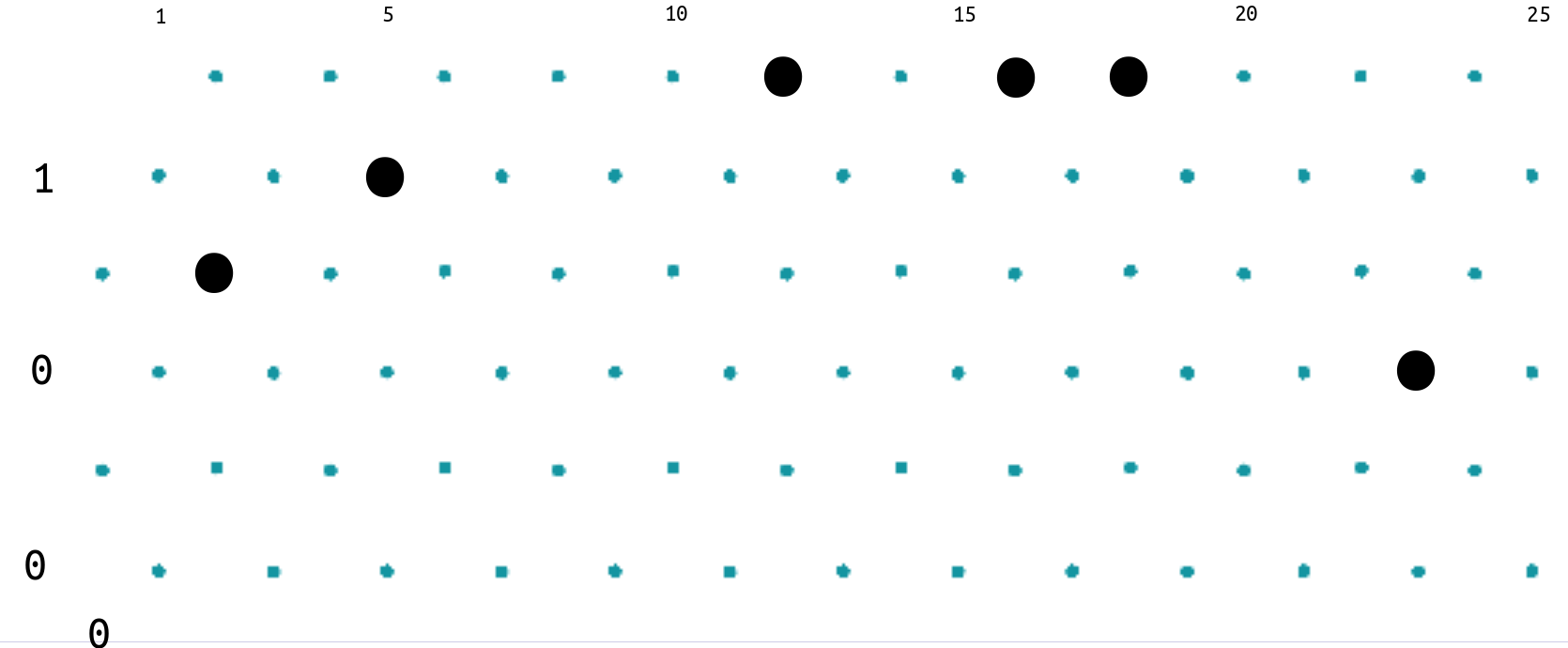}
\end{center}
We remove the part sizes momentarily, 
and start traversing the parts from the left.  
The zeroth round is for determining relative height zero, 
and the part in consideration is shown in red.  
\begin{center}
  \includegraphics[scale=0.1]{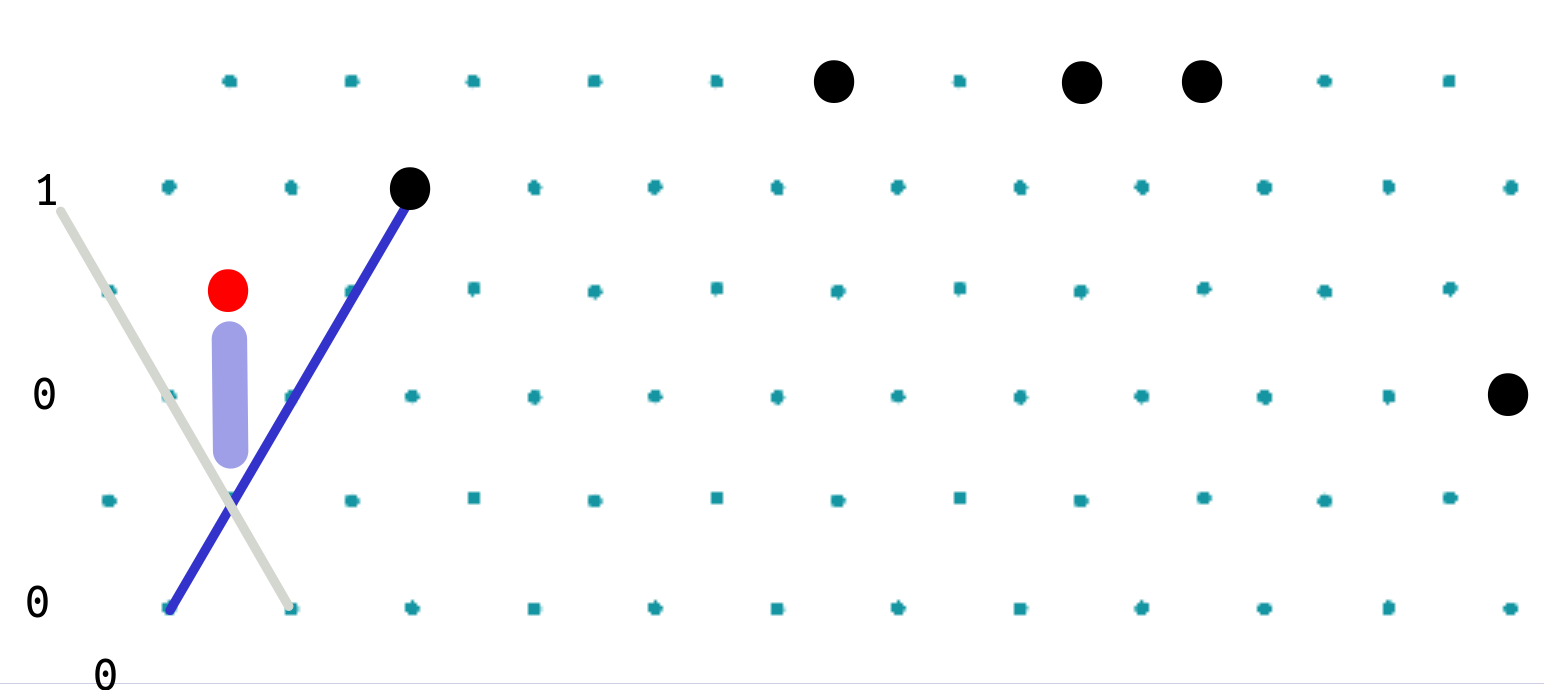} 
  \raisebox{1cm}{$\rightarrow$}
  \includegraphics[scale=0.1]{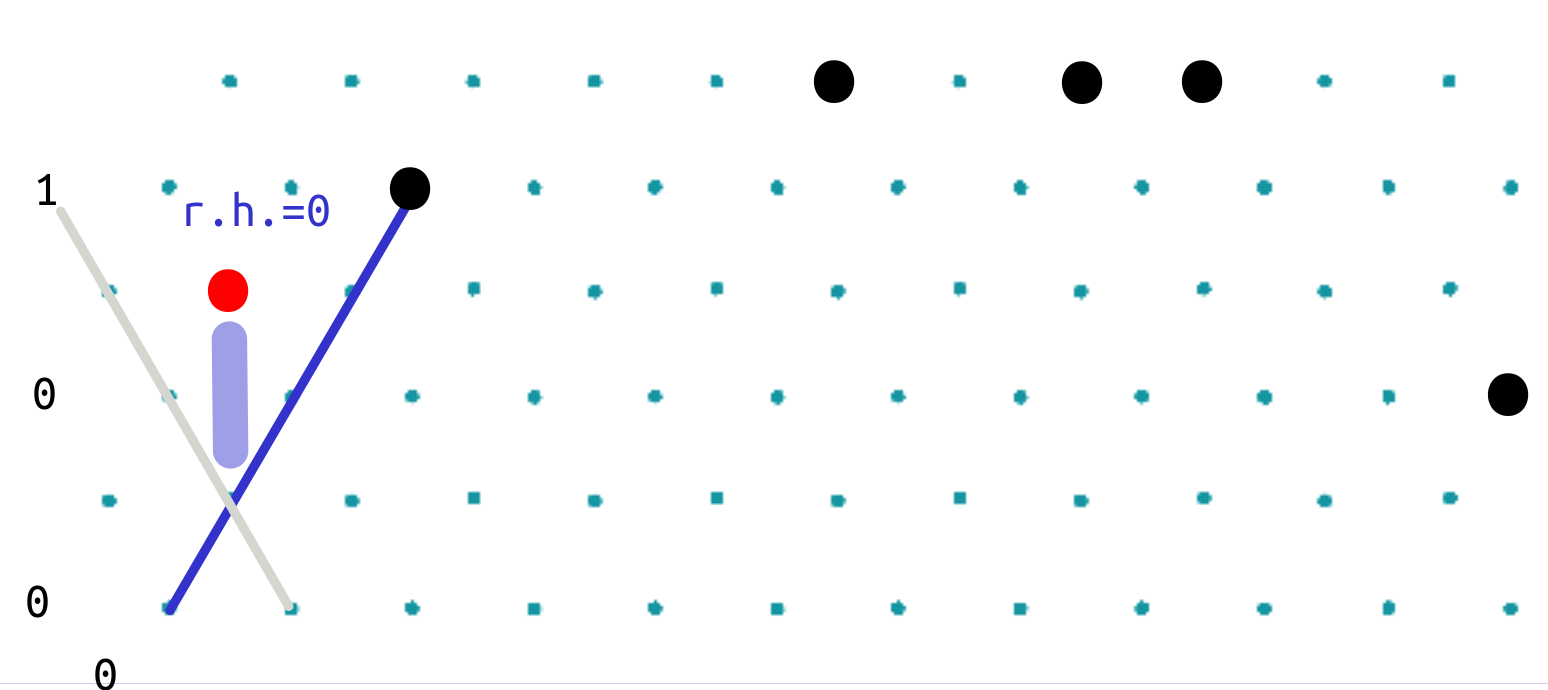} 
\end{center}
We hit a part with relative height zero.  
So we assign it its relative height, 
then fold and identify.  
\begin{center}
  \includegraphics[scale=0.1]{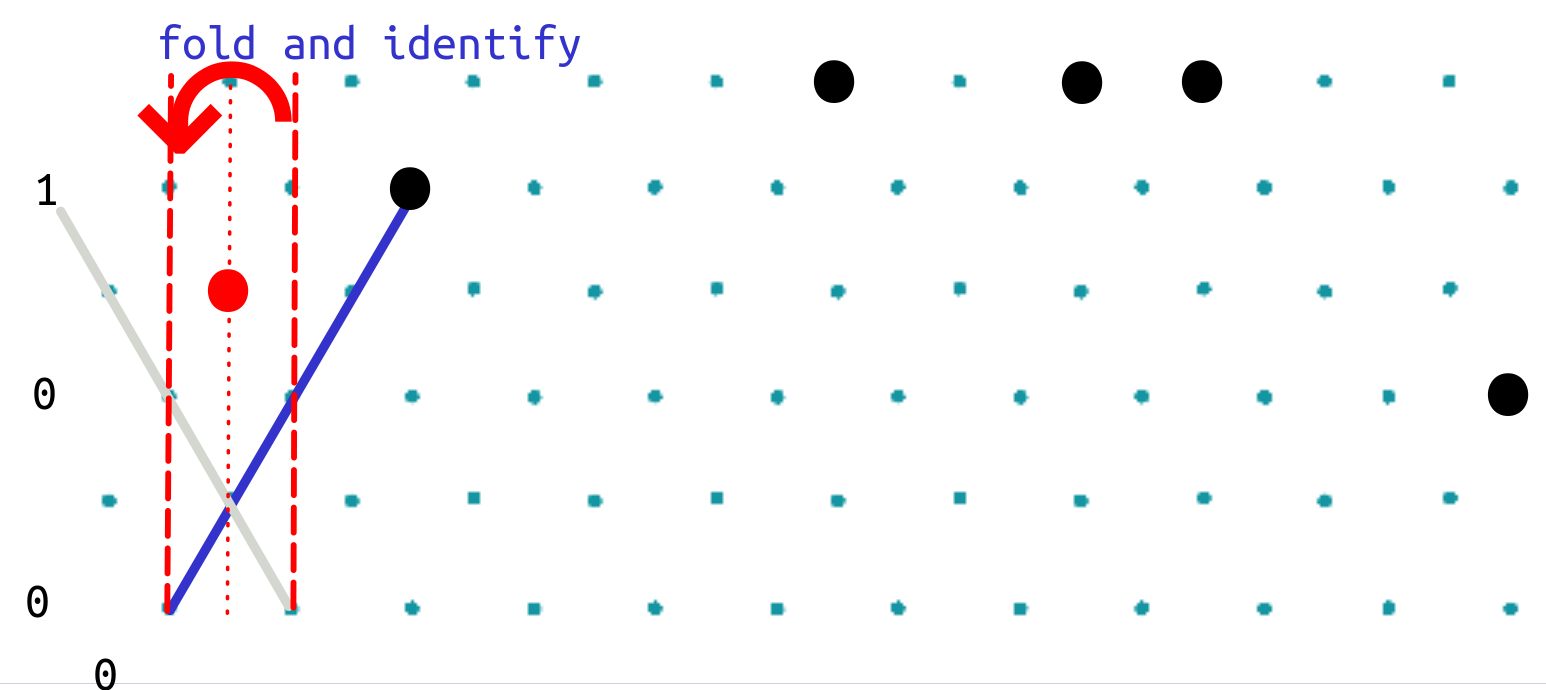} 
  \raisebox{1cm}{$\rightarrow$}
  \includegraphics[scale=0.1]{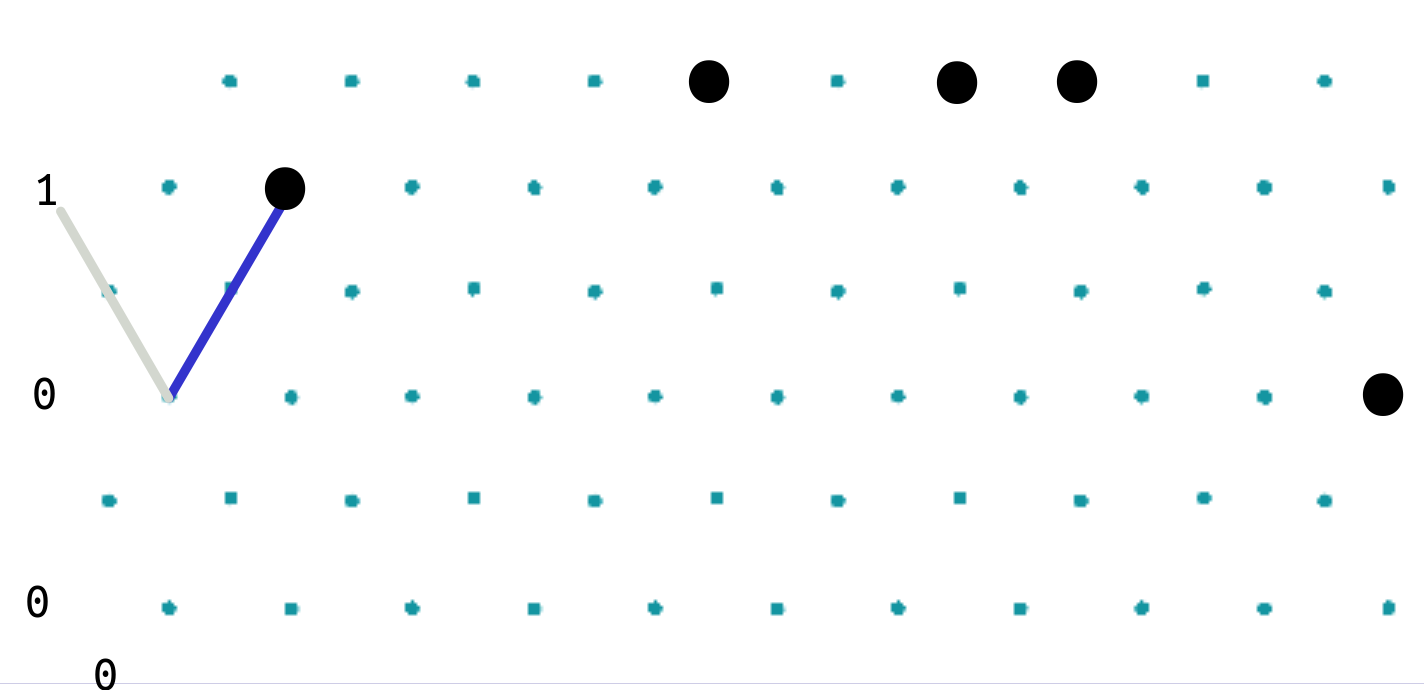} 
\end{center}
We continue with the next part.  
\begin{center}
  \includegraphics[scale=0.1]{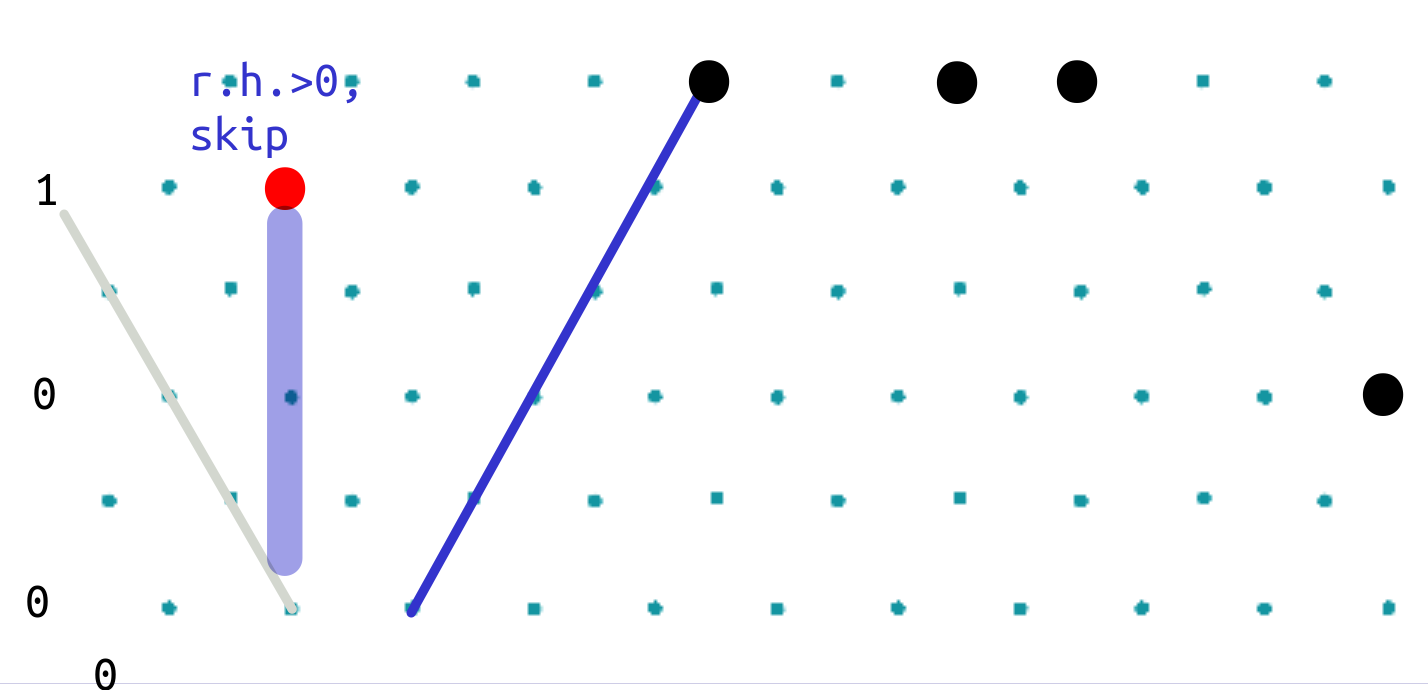} 
  \raisebox{1cm}{$\rightarrow$}
  \includegraphics[scale=0.1]{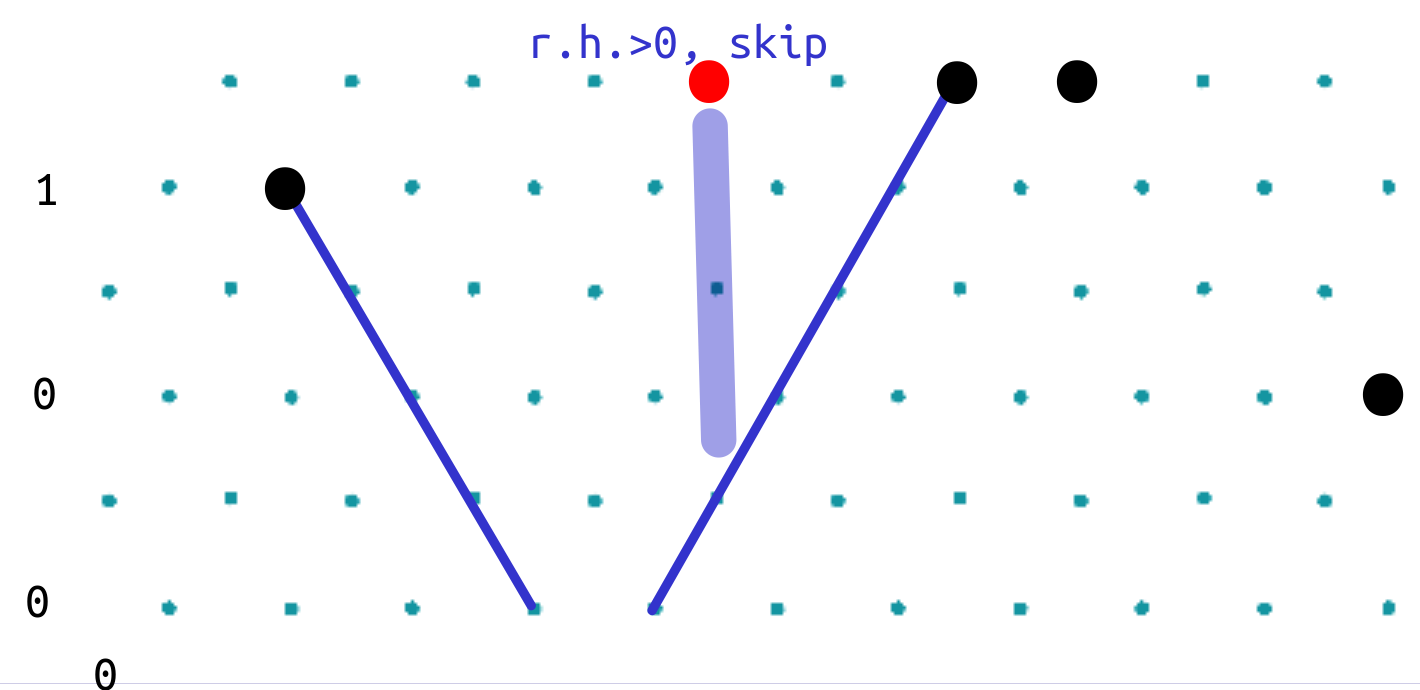} 
\end{center}
These two parts do not have relative height zero, 
so we skipped them.  
The next one does, so we assign it relative height zero, 
then fold and identify.  
\begin{center}
  \includegraphics[scale=0.1]{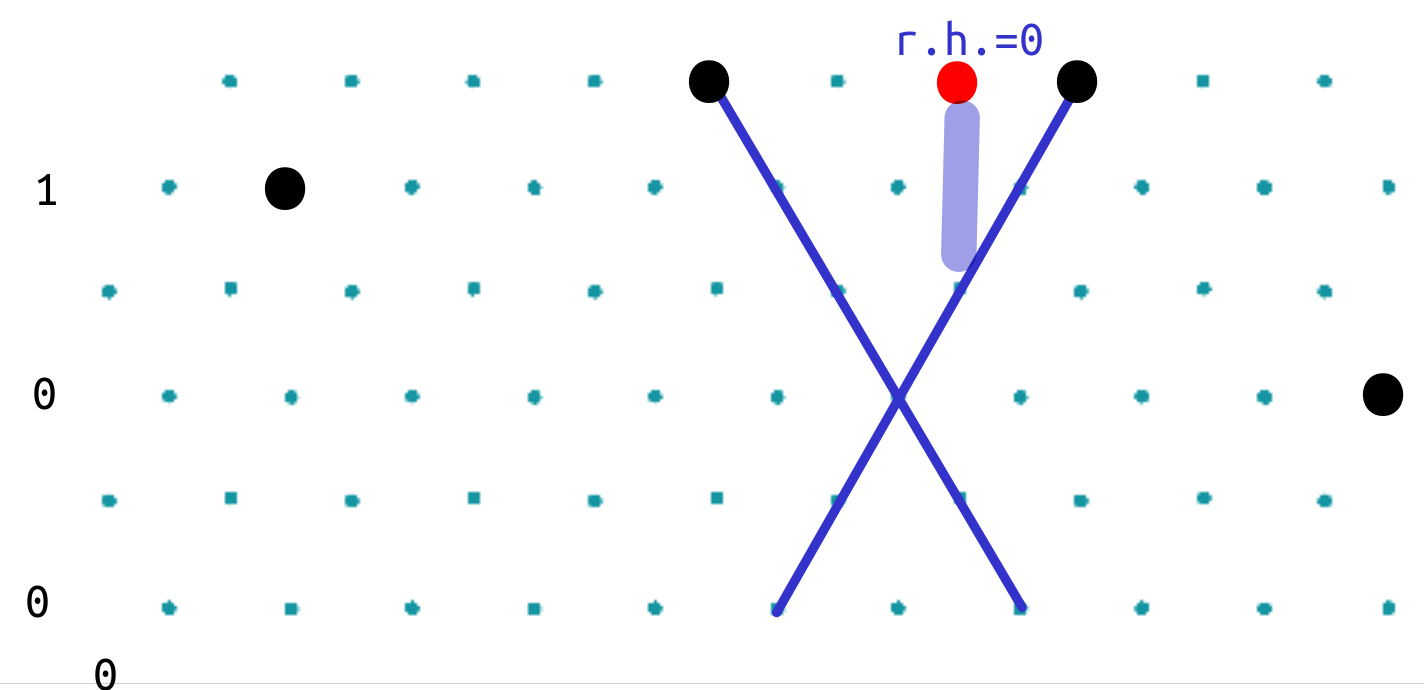} 
  \raisebox{1cm}{$\rightarrow$}
  \includegraphics[scale=0.1]{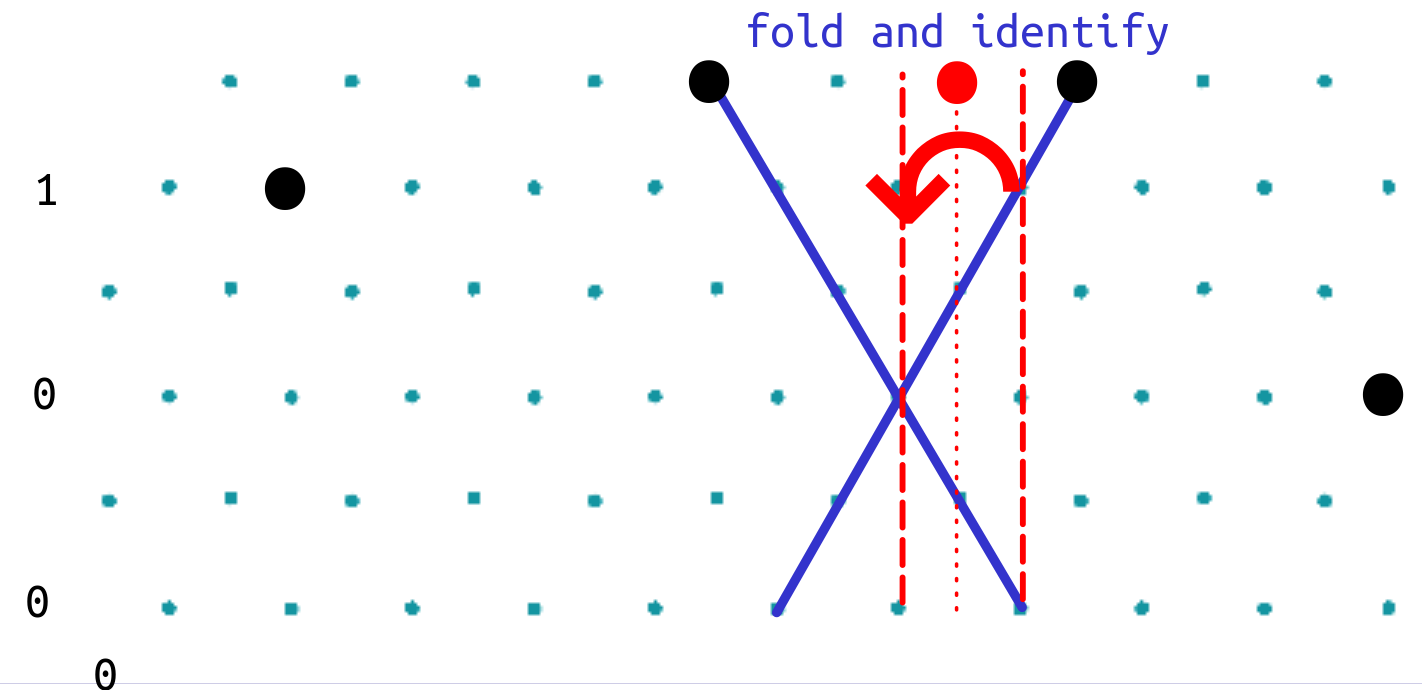} 
  \raisebox{1cm}{$\rightarrow$}
  \includegraphics[scale=0.1]{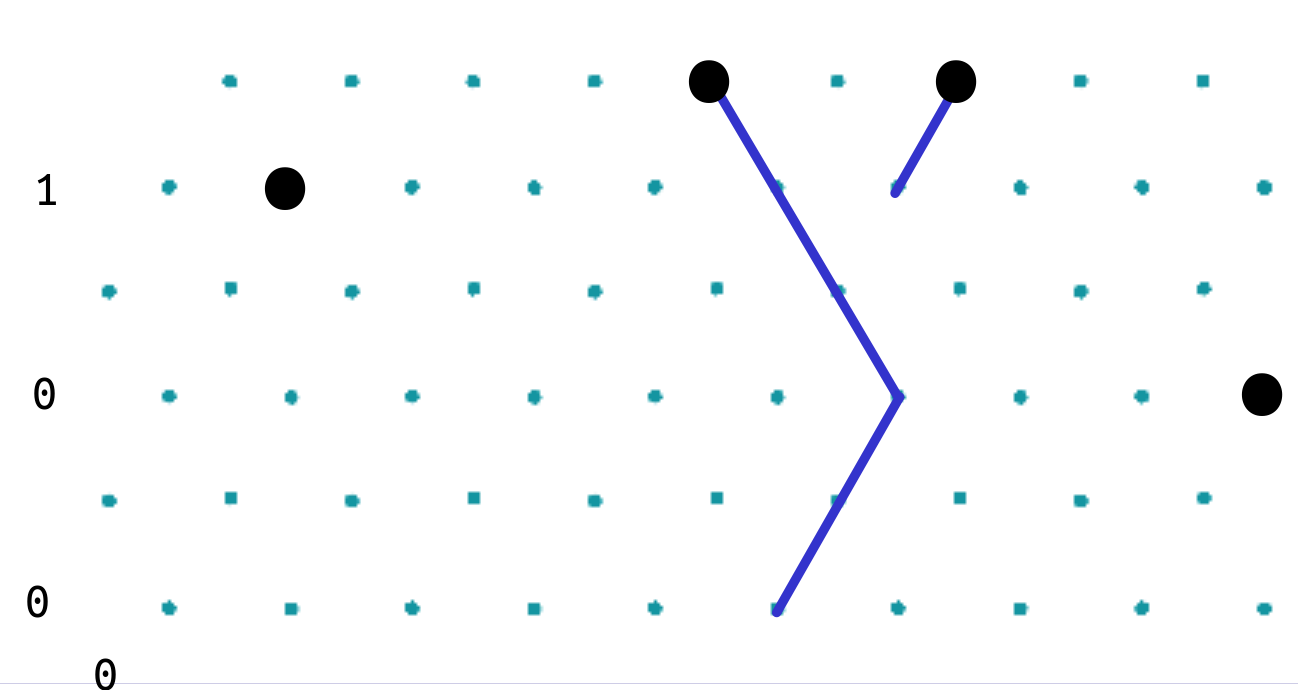} 
\end{center}
Among the remaining two parts, 
the latter one has relative height zero.  
So, we fold and identify around it.  
\begin{center}
  \includegraphics[scale=0.1]{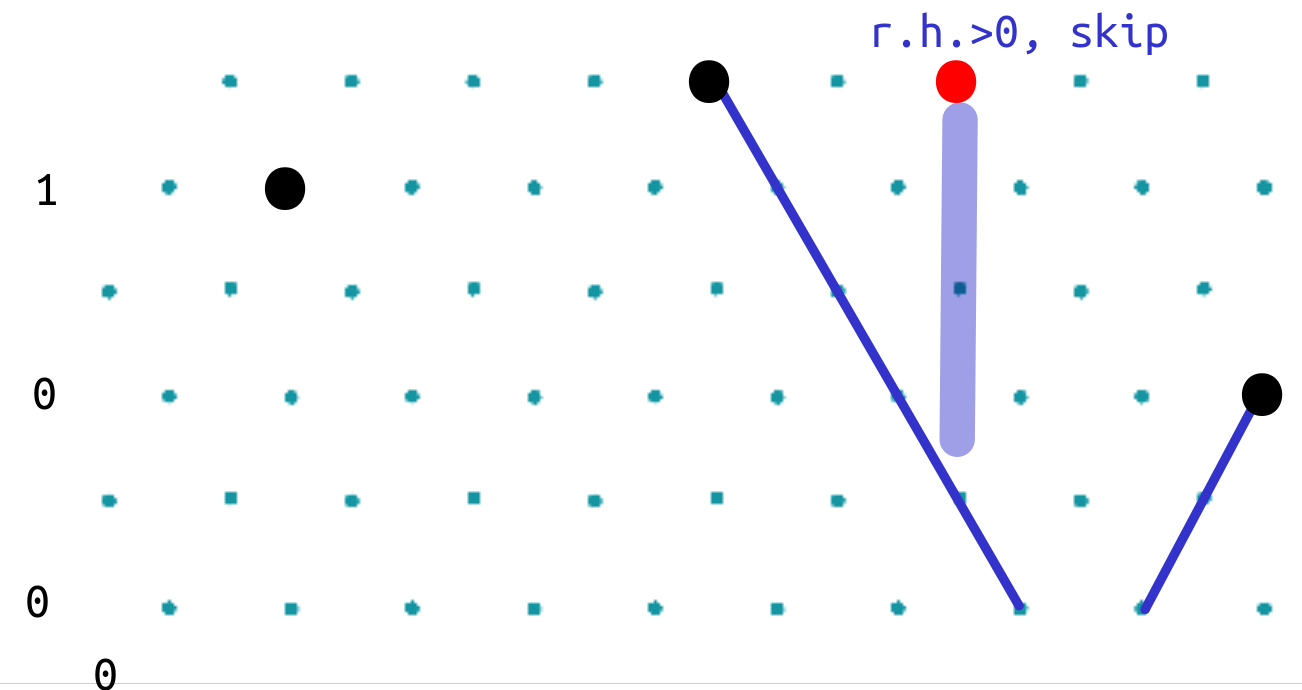} 
  \raisebox{1cm}{$\rightarrow$}
  \includegraphics[scale=0.1]{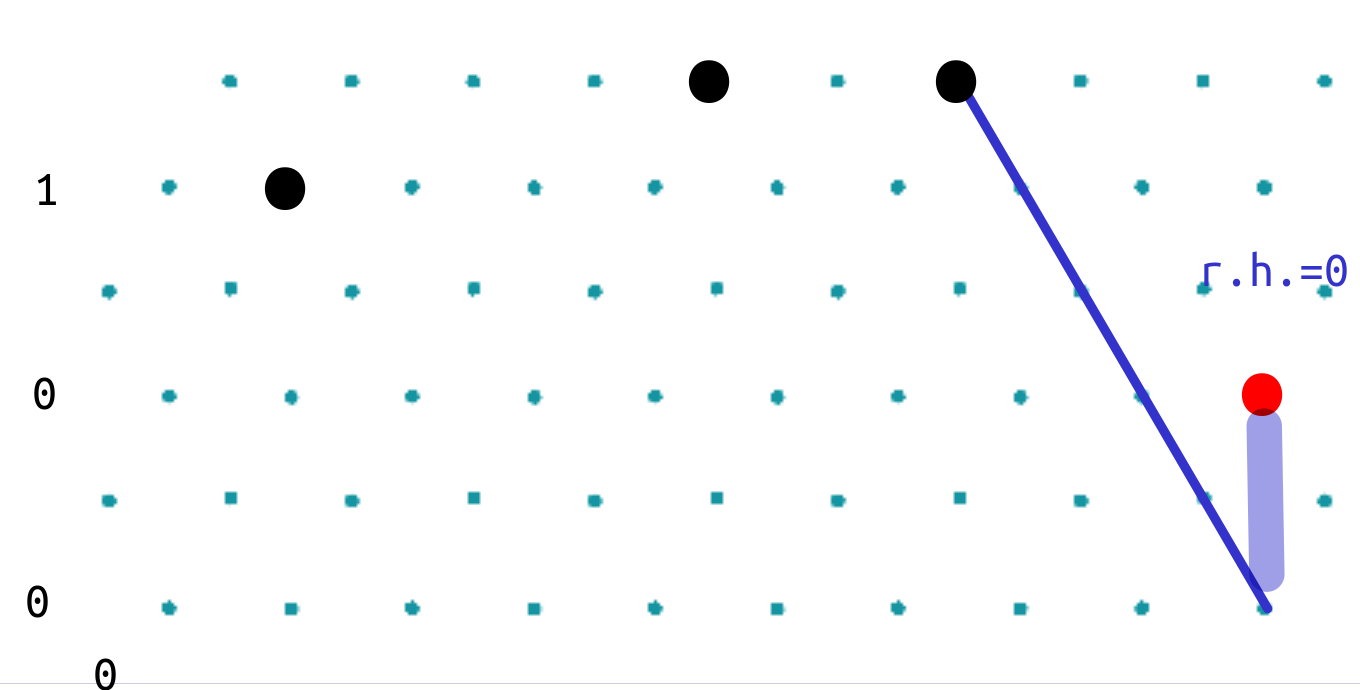} \\ 
  \vspace{5mm}
  \raisebox{1cm}{$\rightarrow$}
  \includegraphics[scale=0.1]{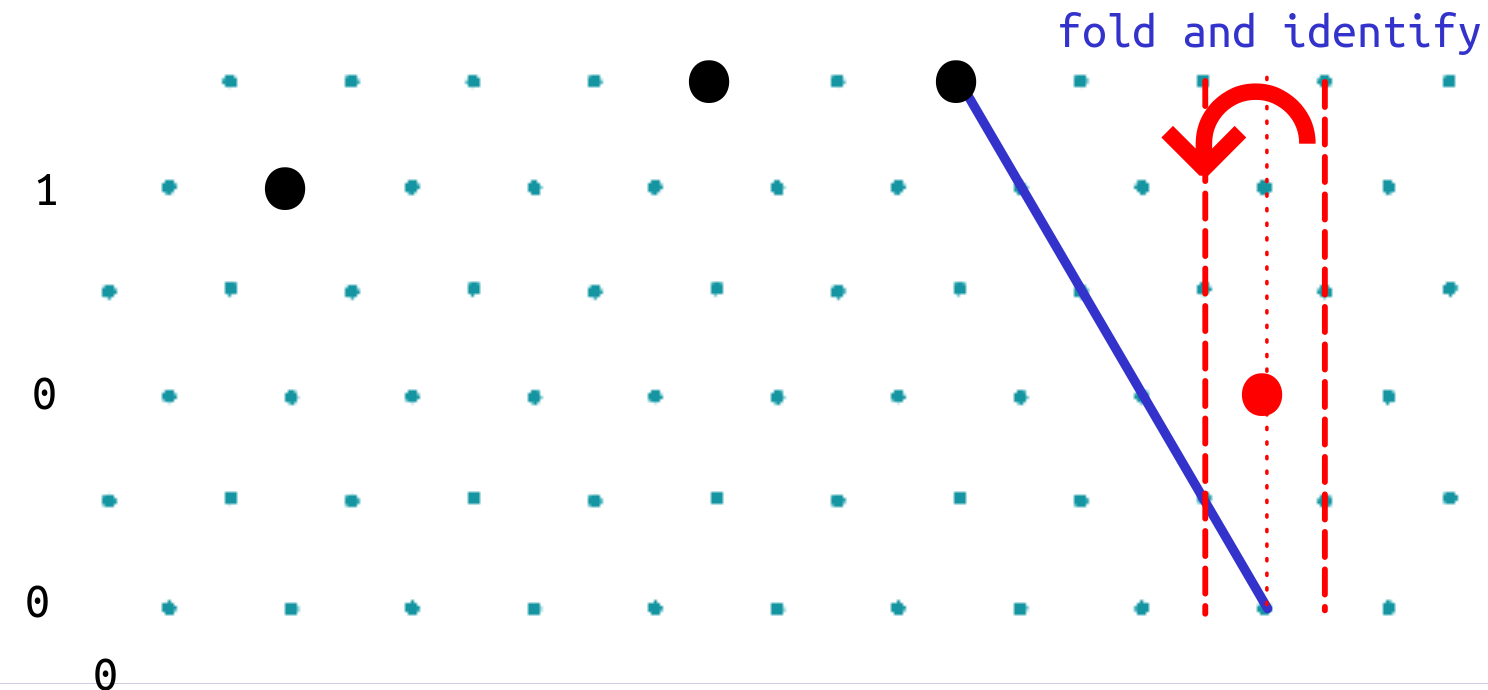} 
  \raisebox{1cm}{$\rightarrow$}
  \includegraphics[scale=0.1]{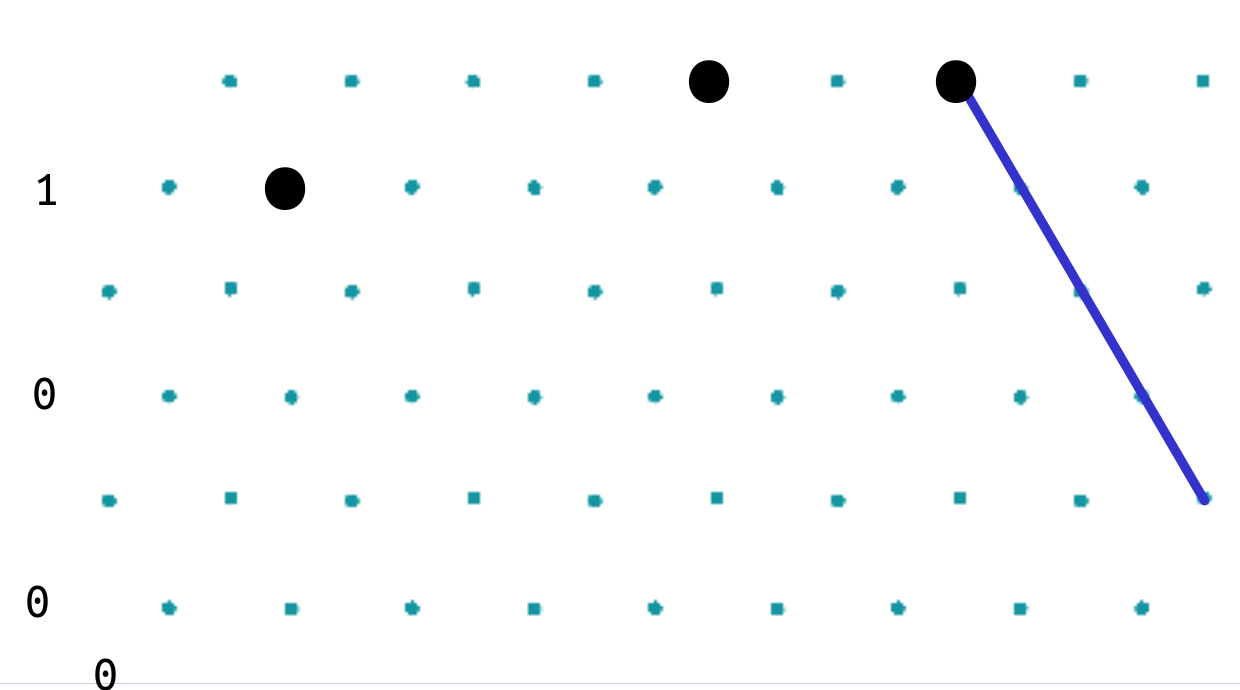} 
\end{center}
The next round, namely the first round, is for determining relative height one.  
We traverse the remaining parts from left to right again.  
The first part has relative height one, 
so we fold and identify around it.  
\begin{center}
  \includegraphics[scale=0.1]{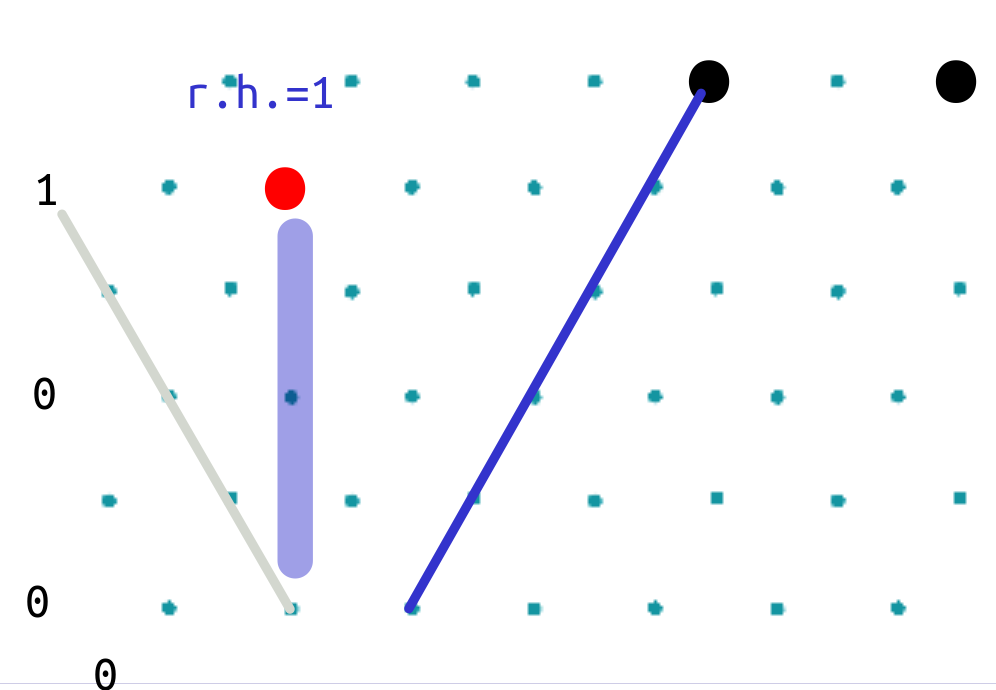} 
  \raisebox{1cm}{$\rightarrow$}
  \includegraphics[scale=0.1]{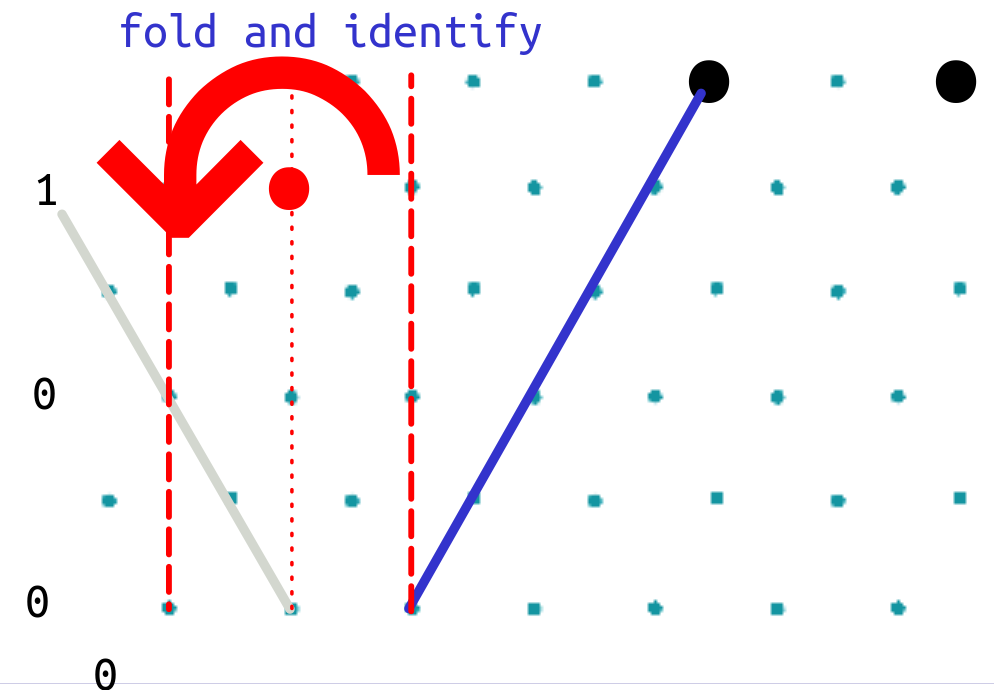} 
  \raisebox{1cm}{$\rightarrow$}
  \includegraphics[scale=0.1]{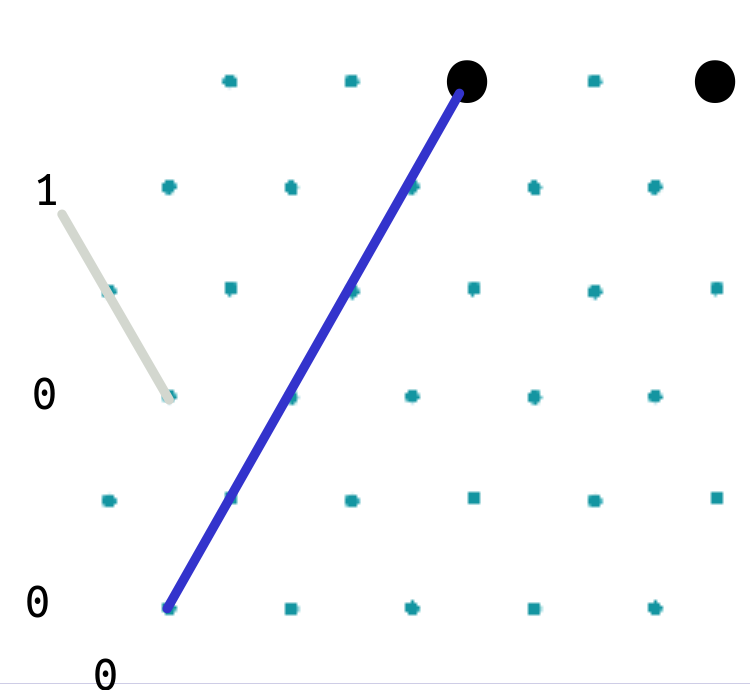} 
\end{center}
Notice that we identify parts that are further apart than 
those in the previous round.  
The reasons is that we identify parts 
that are
\begin{align*}
  (\textrm{the part which has just been assigned relative height } h) 
  \; \pm \; (h+1).  
\end{align*}
$h$ has been increased from zero to one in this round.  
For the remaining two parts, 
only the former has relative height one, 
so we fold and identify around it.  
The latter does not have relative height one, 
so we skip it.  
\begin{center}
  \includegraphics[scale=0.1]{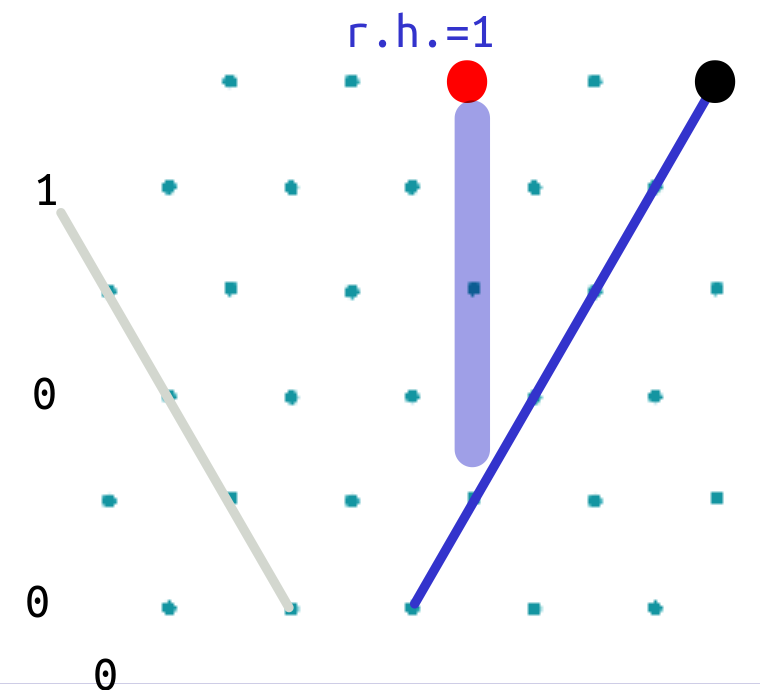} 
  \raisebox{1cm}{$\rightarrow$}
  \includegraphics[scale=0.1]{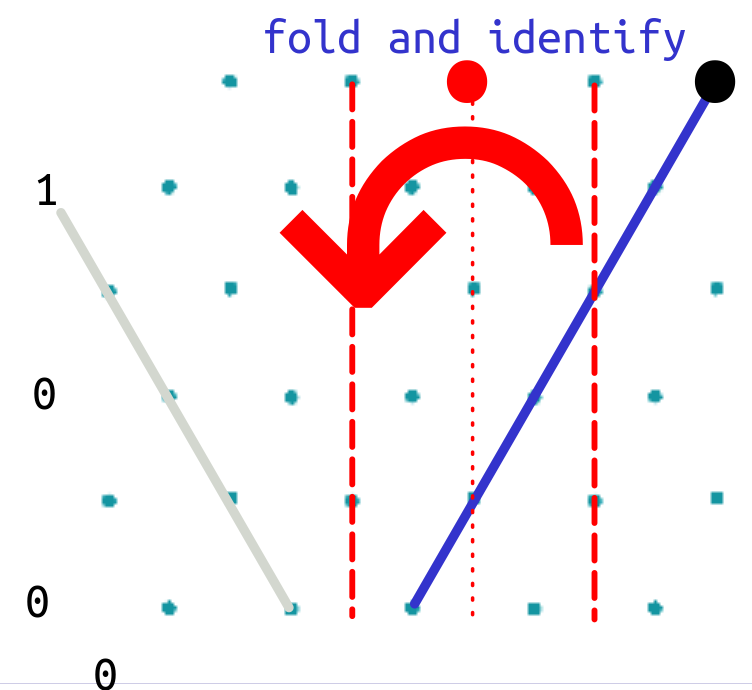} 
  \raisebox{1cm}{$\rightarrow$}
  \includegraphics[scale=0.1]{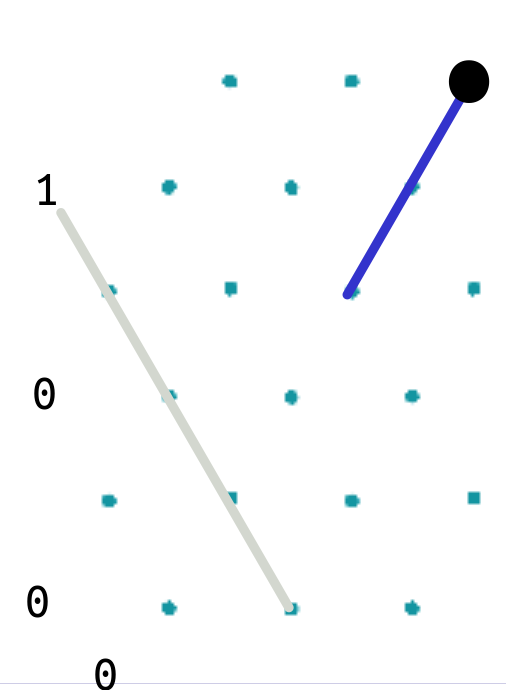} 
  \raisebox{1cm}{$\rightarrow$}
  \includegraphics[scale=0.1]{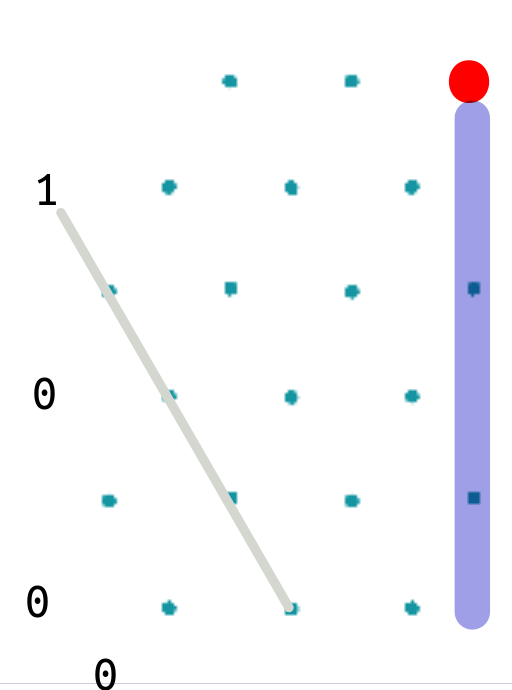} 
\end{center}
The second and the last round is for 
relative height two.  
There is only one part, 
and it has relative height two.  
We assign it its relative height, 
and there is no need to fold and identify anymore.  
\begin{center}
  \includegraphics[scale=0.1]{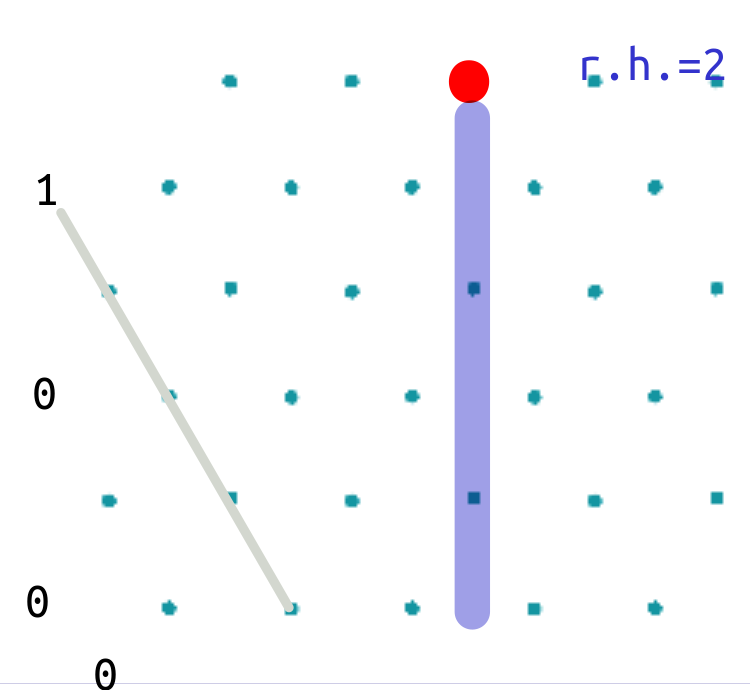} 
\end{center}
Finally, we unfold everything, 
and indicate relative heights of parts 
along with their absolute heights.  
We also make the part sizes visible again.  
\begin{center}
  \includegraphics[scale=0.2]{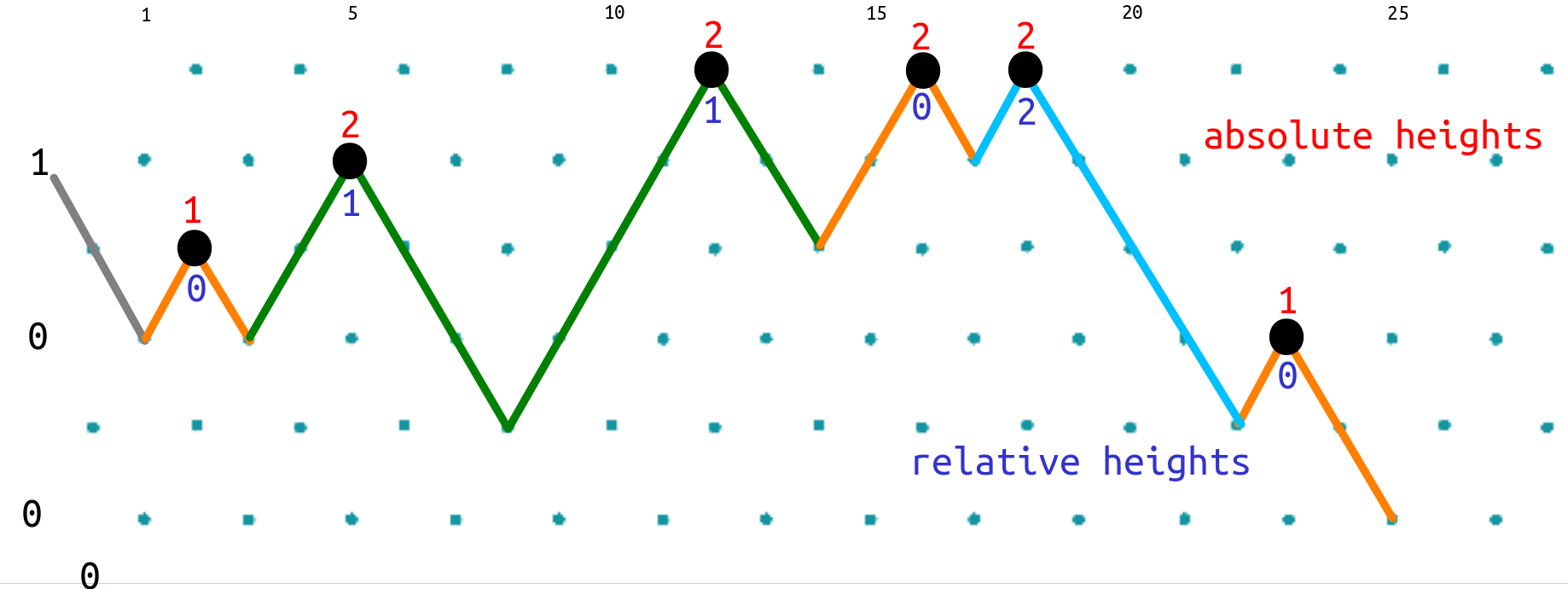} 
\end{center}

We denote relative heights of parts in parentheses as subscripts.  
For example, the last displayed CMPP partition is 
\begin{align*}
  2_{(0)} + 5_{(1)} + 12_{(1)} + 16_{(0)} + 18_{(2)} + 23_{(0)}.  
\end{align*}

\subsection{Observations on the Assignment of the Relative Heights}
\label{subsecRelHghtAlgoRemarks}

There are many details to be pointed out and many remarks to be made
on the algorithm given in Subsection \ref{subsecRelHghtDeterm}.  

When both legs of each part are drawn, 
the obtained wedges may overlap, 
but none of them may be contained in the other 
because of the admissibility condition of CMPP partitions with $k = 1$.  
For visual convenience, we draw the upper envelope of the resulting shape.  
This will make pictures look clearer 
when we deal with the forward and backward moves later.  
Below are three examples for relative positions of two parts.  
In the first two, the legs of the wedges overlap, 
in the third, one wedge is completely contained in the other.  
The first may appear in an admissible CMPP partition with $k = 1$, 
while the second and third may not.  
\begin{center}
  \includegraphics[scale=0.1]{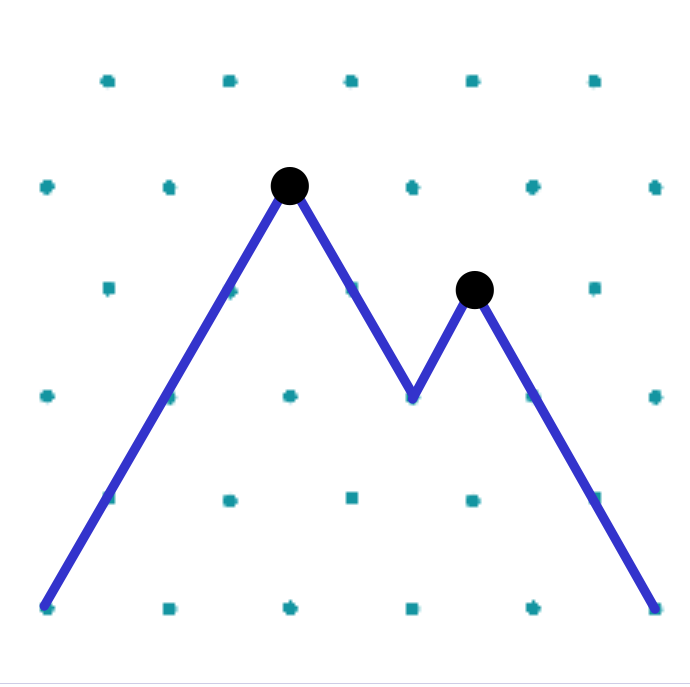}
  \hspace{1cm}
  \includegraphics[scale=0.1]{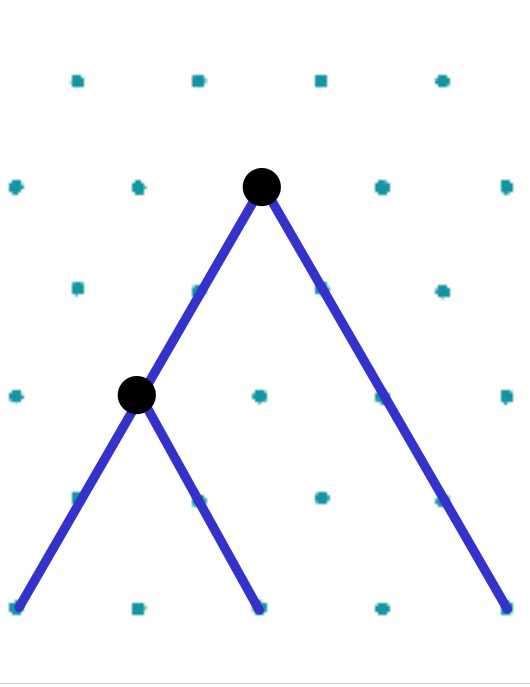}
  \hspace{1cm}
  \includegraphics[scale=0.1]{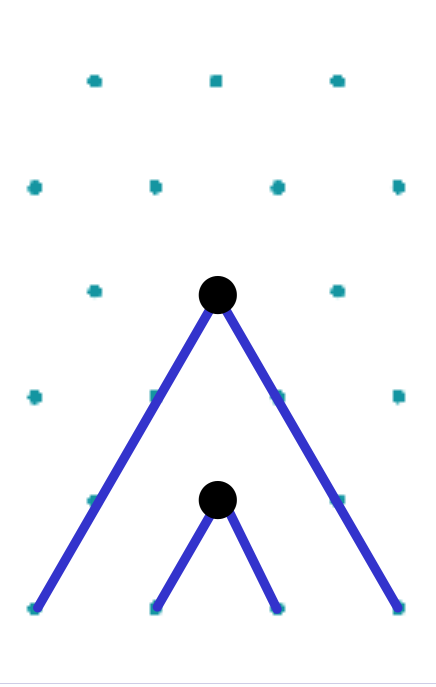}
\end{center}

It is clear that the relative height is at most equal to the absolute height 
for each part.  
Therefore, relative heights can be at most $\ell-1$.  

\begin{prop}
\label{propAfterRoundH}
  For an arbitrary but fixed $h=$0, 1,  \ldots, $\ell-1$, 
  after the steps (P), (DL), and (S) or (F) 
  are run for each appearing part in the intermediate diagram, 
  the leftover parts at the end cannot have relative height $h$ or smaller.  
\end{prop}

It follows that the leftover parts after \emph{round} $h$ 
cannot have absolute height $h$ or smaller, either.  
It also follows that after the termination of the algorithm, 
all parts will be assigned a relative height, 
and these will necessarily be unique.  

\begin{proof}

  We will work with a chosen and fixed $h$ = 0, 1, \ldots, $\ell-1$.  
  Before and after a folding operation, we will show two things: 
  
  {\bf (i)} There can be no parts whose 
  relative heights become smaller after a folding operation.  
  
  {\bf (ii)} A folding operation conceals one and only one part.  
  This is necessarily the part the relative height of which is 
  declared $h$.  
  
  Both of these claims can be shown with a visual aid.  
  \begin{figure}
    \includegraphics[scale=0.2]{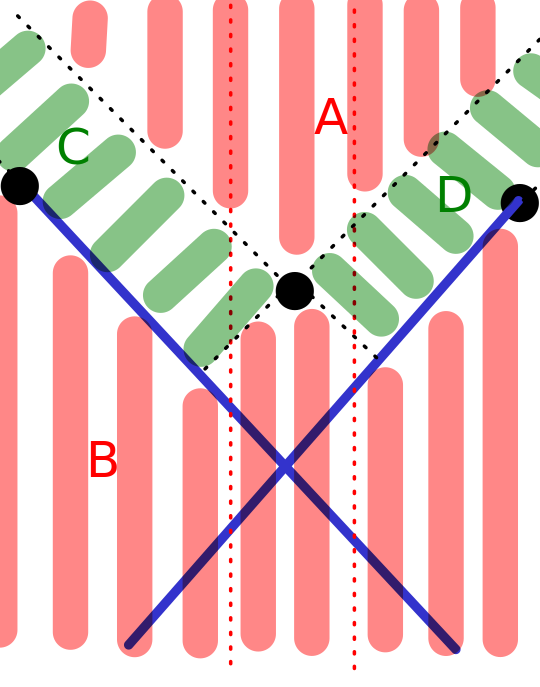} \\
    \caption{Before a folding operation on a designated part}
    \label{figFoldingOpnBefore}
  \end{figure}
  Figure \ref{figFoldingOpnBefore} depicts a folding operation aroud 
  the designated part $n$.  
  There may or may not be a preceding ($m$) or a succeeding part ($M$), 
  in which case the discussion is simplified.  
  
  We see that there can be no parts in regions A or B, 
  because otherwise the CMPP conditions for $k = 1$ would not hold.  
  There cannot be any parts in regions C or D, 
  because otherwise the relative height of the designated part 
  would have been smaller.  
  If we think in terms of induction, 
  a smaller relative height is not possible.  
  \begin{figure}
    \includegraphics[scale=0.2]{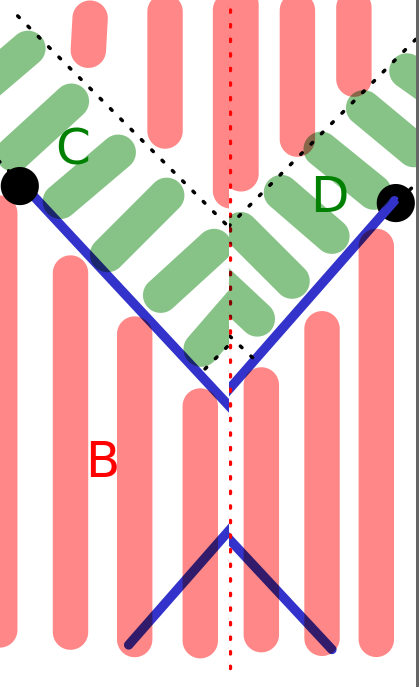} \\
    \caption{After a folding operation on a designated part}
    \label{figFoldingOpnAfter}
  \end{figure}
  Figure \ref{figFoldingOpnAfter} shows the locality of the hidden part after the folding.  
  The right leg of the preceding part of $M$ may only have moved downward, 
  because there are no parts in region C.  
  Similary, the left leg of the succeeding part of $m$ 
  may have only moved downward because there are no parts in region D.  
  Therefore, after a folding operation during round $h$, 
  no part that remains visible on the diagram 
  may have an implied relative height smaller than $h$.  
  
  For the middle part in other succesive triples of parts, 
  the positions of the right leg of the preceding part 
  and the left leg of the succeeding part remain unchanged.  
  This follows from the fact that there are no parts in regions A, C or D.  
\end{proof}

\section{Definition and Some Properties of Moves}
\label{secMoves}

After the determination of relative heights, 
we can describe moves, 
and prove some of their properties.  

\begin{defn}
\label{defFwdMove}
For a \emph{forward move} on the designated part $n$ with relative height $h$, 
consider the two possibilities in the following picture.  
\begin{center}
  \includegraphics[scale=0.15]{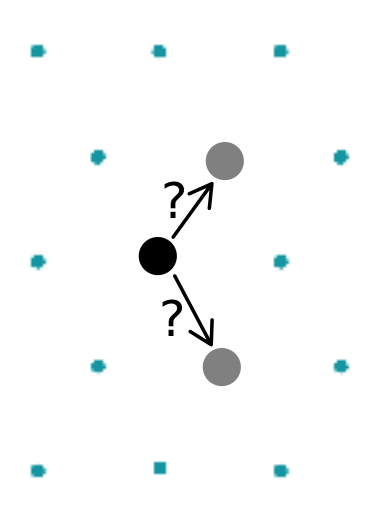}
\end{center}
Except for one case, 
we perform the move which satisfies the four properties below: 

(K1) The move preserves the CMPP condition for $k=1$.  

(M0) The move preserves the relative height of $n$.  

(MP) If there is a preceding part, 
the move preserves the relative height of the preceding part.  

(MS) If there is a succeeding part, 
the move preserves the relative height of the succeeding part.  

If at least one of the four conditions are violated for 
either possibility, the move is not allowed.  

The exceptional case is when the suceeding part is $n+2h+2$, 
it is level with $n$, 
and with a strictly greater relative height $h'$.  
\begin{center}
  \includegraphics[scale=0.2]{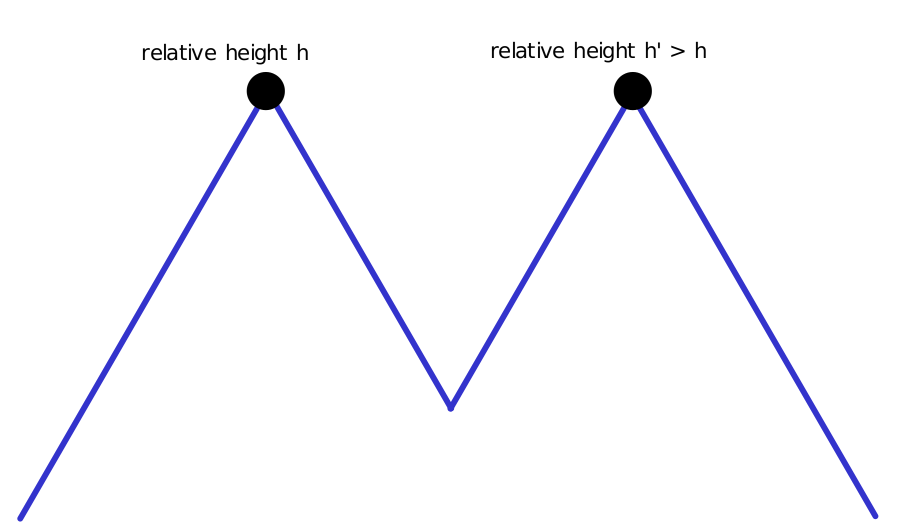}
\end{center}
In this case, 
we switch the relative heights of $n$ and $n+2h+2$, 
try the move on $n+2h+2$.  
If the move on $n+2h+2$ succeeds, the relative height swap is kept.  
If not, both $n$ and $n+2h+2$ retain their original 
relative heights $h$ and $h'$, respectively.  
\end{defn}

A \emph{backward move} is defined likewise.  
The exceptional case is 
when part $n$ with relative height $h$ 
is preceded by a level $n-2h-2$ with strictly greater relative height $h'$.  
In this case, we immediately switch relative heights.  
This is also necessitated by the determination of relative heights. 
In particular, the relative height of $n-2h-2$ is now $h$, 
and the relative height of $n$ is now $h'$, 
and the part which is trying to be moved is $n-2h-2$.  
This configuration is kept even if a backward move is not possible on $n-2h-2$.  

Of course we need to show that when a move is possible, 
it is unambiguously defined.  
We will show when a move is possible, 
exactly one of the possibilities 
keep the relative height of the designated part $n$, 
and the other either changes it by $\pm 1$, 
or the resulting partition is not admissible anymore.  
We examine several cases.  

If the relative heights and the absolute heights are equal for $n$ 
and the potential $(n+1)$'s, 
then the forward move is uniquely possible, 
because exactly one of the $(n+1)$'s in the picture in Definition \ref{defFwdMove} 
will have the same absolute height as $n$, and the other will have 
absolute height that differs by $\pm 1$.  


In case the right leg of the preceding part 
or the left leg of the suceeding part contains an $n$ in them, 
there are two exclusive and complementary cases.  
One of them is that the $n$ on the right leg of the preceding part 
has strictly greater absolute height, 
or the left leg of the suceeding part does not contain an $n$.  
\begin{center}
  \includegraphics[scale=0.15]{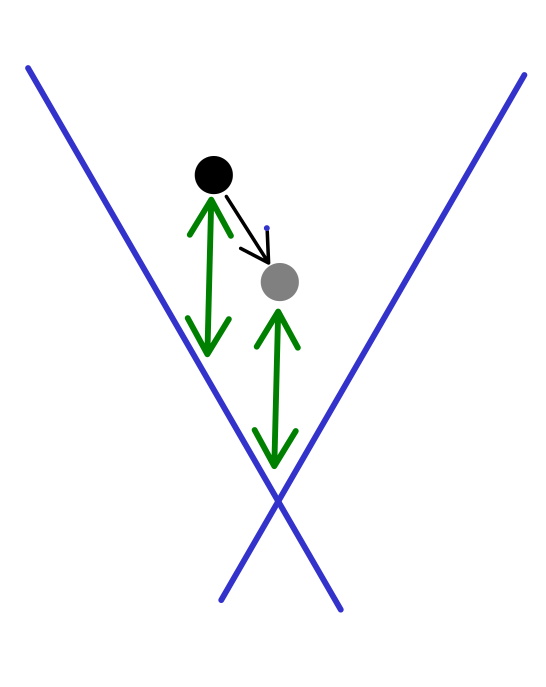}
\end{center}
In this case, only the $n+1$ with the smaller absolute height 
has the same relative height as $n$, 
and the other has relative height one more than that of $n$, 
i.e. the move is uniquely defined.  

The other case is when the left leg of the succeeding part 
contains an $n$ with equal or greater absolute height 
of the $n$ in the right leg of the preceding part, 
or the left leg of the preceding part has no $n$'s.  
In this case, only the $n+1$ in the picture in Definition \ref{defFwdMove} 
with the greater absolute height will have the same relative height as 
the $n$ we are trying to move forward, 
and the other will have exactly one less.  
\begin{center}
  \includegraphics[scale=0.15]{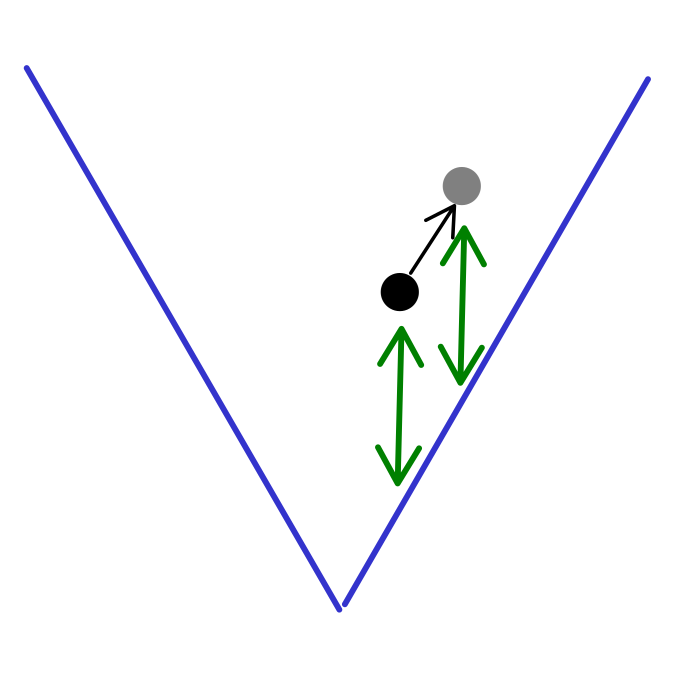}
\end{center}
In the above discussion, 
the preceding part of $n$ may refer to 
a preceding part of $n$ which became the immediately preceding part of $n$
after a few folding operations.  
It is likewise for the succeeding part of $n$.  
We now know that a possible move is unambiguously defined.  
About the possibility of moves, 
we will prove a claim that is enough for our purposes 
instead of the fully general case.  

\begin{prop}
\label{propPossibleMoves} 
  Let $\doublestroke{\lambda}$ be a CMPP partition for $k=1$, 
  and $h = $ 0, 1, \ldots, $\ell-1$ be arbitrary but fixed.  
  Let $m < n$ be parts of $\doublestroke{\lambda}$ with the same relative height $h$ 
  such that there are no parts $r$ 
  with relative height $h$ or less between $m$ and $n$.  
  \begin{enumerate}[{\bf (i)}]
   \item If the forward movement of $m$ does not alter 
    the relative heights of the parts preceding $m$, 
    then $m$ can make a forward move unless $n = m+2h+2$, 
    in which case $m$ and $n$ are level.  
   \item If the backward movement of $n$ does not alter 
    the relative heights of the parts succeeding $n$, 
    then $n$ can make a backward move unless $m = n-2h-2$, 
    in which case $m$ and $n$ are level.  
  \end{enumerate}

\end{prop}

\begin{proof}
  When $n = m+2h+2$, $m$ and $n$ has to be level 
  so as to not force the lower one to have smaller relative height.  
  \begin{center}
    \includegraphics[scale=0.15]{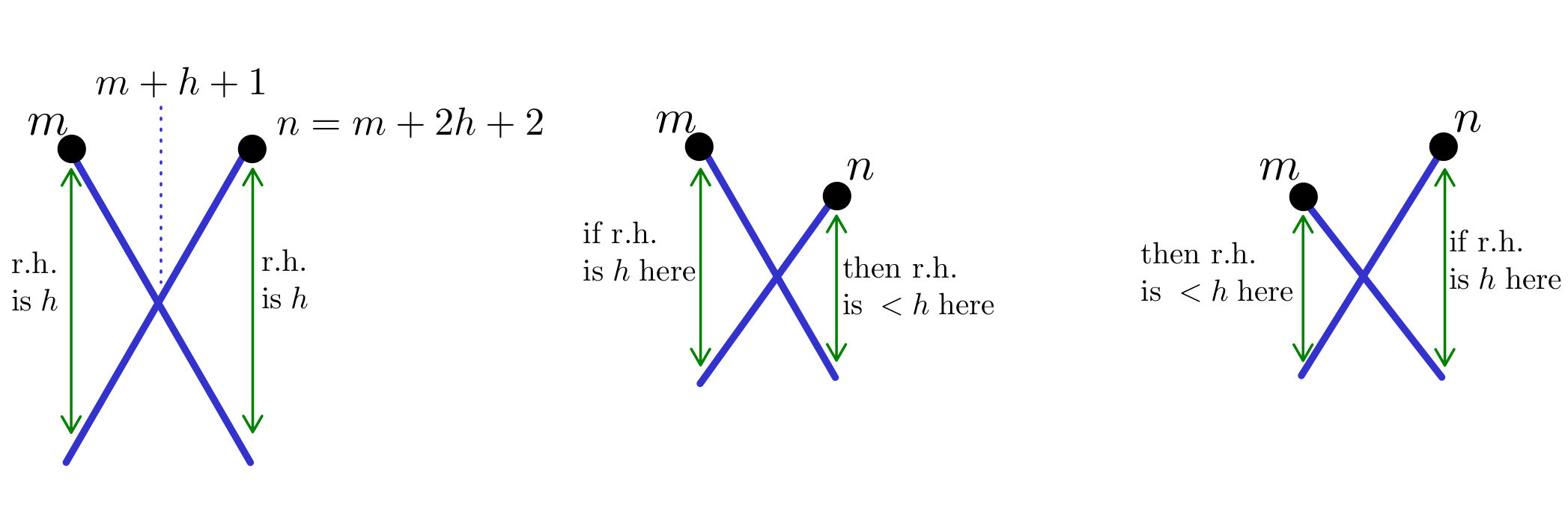}
  \end{center}
  When $n = m+2h+2$, $m$ and $n$ are level, and $m$ is moved forward, 
  either $m$ or $n$ will be forced to have relative height $h-1$, 
  or an inadmissible CMPP partition will arise for $h = 0$.  
  This can be seen in the second and the third pictures above, 
  if one changes $m$ to $m+1$ there.  
  
  These prove the impossibilities.  
  For the affirmative parts, we examine several exclusive cases.  
  
  Suppose there are no other parts between $m$ and $n$, and $n > m+2h+2$.  
  If $m$ and $n$ are level, 
  then the relative heights of neither is determined by the relevant leg of the other.  
  So, both options for moving $m$ forward is safe as far as 
  the relative height of $n$ is concerned.  
  We can make the unique forward move on $m$ 
  which chooses the $(m+1)$ with relative height $h$.  
  By the hypothesis, the move does not alter the relative heights of the preceding parts.  
  Similarly, we can make a backward move on $n$.  
  
  
  If $n$ is above $m$, 
  then the relative height of $m$ may be determined by the left leg of $n$, 
  but the relative height of $n$ cannot be determined by the right leg of $m$.  
  So, we can move $m$ forward in the direction of its left arm.  
  If $n$ is above $m$ but the relative height of $m$ is not determined 
  by the left leg of $n$, then we move $m$ forward as in the previous paragraph.  
  

  If $n$ is below $m$, 
  then the relative height of $n$ may be determined by the right leg of $m$, 
  but the relative height of $m$ cannot be determined by the left leg of $n$.  
  So, we can move $m$ forward in the direction of its left leg.  
  If $n$ is below $m$ but the relative height of $n$ is not determined 
  by the right leg of $m$, then we move forward as in the second previous paragraph.  
  

  We treat the backward move on $n$ similarly.  
  
  If there are any parts between $m$ and $n$, 
  let $r$ be the succeeding part of $m$, 
  and $s$ be the preceding part of $n$.  
  Both $r$ and $s$ have relative heights greater than $h$ by the hypothesis.  
  In fact, $r$ and $s$ may be the same part.  
  
  The forward move on $m$ is possible 
  by the same arguments as above.  
  Moreover; there are simplifications, 
  because the relative height of $r$ 
  may not be determined by the right leg of $m$.  
  We fold around $m$ before we set the relative height of $r$.  
  Likewise, a backward move on $n$ is possible.  
  
  The case we need to examine more carefully is the case 
  in which $r = m+2h+2$ and $m$ and $r$ are necessarily level.  
  By the exceptional move, 
  we first switch the relative heights of $m$ and $r$, 
  and justify the possibility of moving $r$ forward.  
  
  Note that we cannot have $n = s+2h+2$, 
  because this forces either $s$ to have relative height $h$ or less, 
  or $n$ to have relative height strictly less than $h$.  
  \begin{center}
    \includegraphics[scale=0.2]{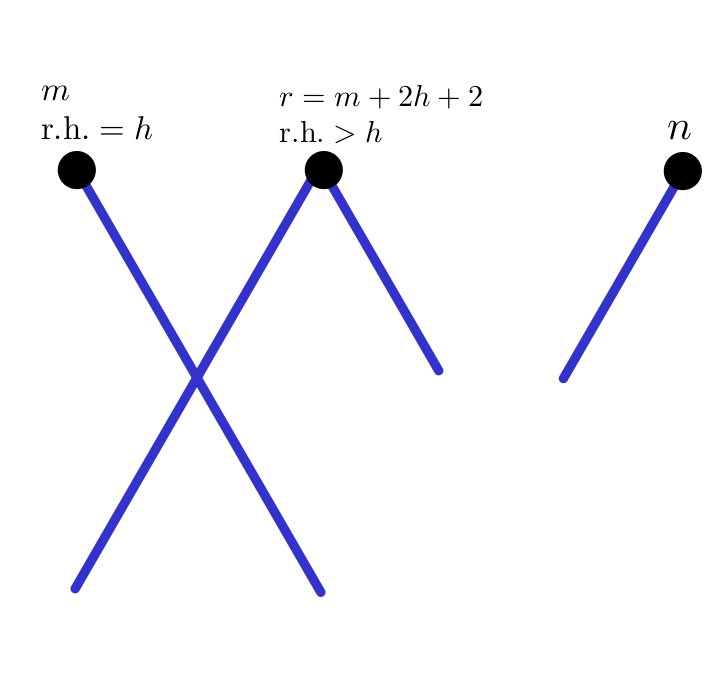}
    \raisebox{2cm}{$\longrightarrow$}
    \includegraphics[scale=0.2]{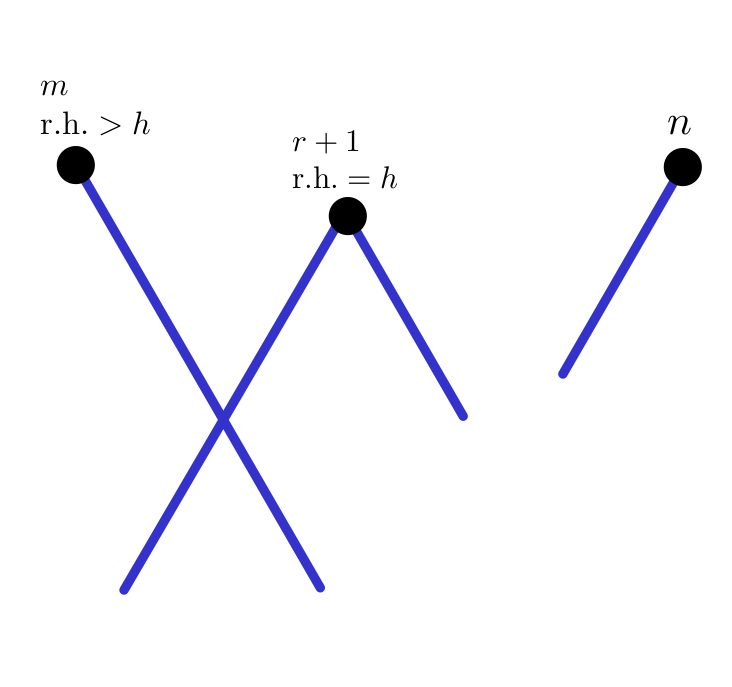}
  \end{center}
  In the former picture above, 
  the relative height of $m$ is necessarily determined by the left leg of $r$, 
  and we fold around $m$ before we fold around $r$. 
  In the latter picture, 
  the relative height of $r+1$ is determined by the right leg of $m$, 
  and we fold around $r+1$ before we fold around $m$.  
  The intermediate configurations we obtain 
  after folding around $m$ in the first picture 
  and after folding around $r+1$ in the second picture are the same.  
  This justifies the relative height switch 
  and the possibility of the forward move afterwards.  
  
  Also note that when $r = m+2h+2$ and $m$ and $r$ are level, 
  $r$ cannot be suceeded by a part $r+2h+2$ or less, 
  because this either forces $r$ to have relative height $h$ or less, 
  or the succeeding part to have relative height strictly less than $h$.  
  Both of there are contrary to the hypotheses.  
  
\end{proof}

One can observe that when a part $n$ with relative height $h$
succeeds a part $s$ with relative height greater than $h$, 
$n$ can move backward until it becomes $s+2h+2$ and 
at which time it necessarily becomes level with $s$, 
and $s$ and $s+2h+2$ will switch relative heights.  
In fact, this is independent of the existence of a part $m<s$ 
with relative height $h$.  
Therefore; a part with relative height $h$ can move backwards 
and become smaller than parts with larger relative heights.  
This is the key in constructing the base partitions further below, 
and establish their uniqueness.  
To see that the conditions of the claim is satisfied 
when all parts succeeding $n$ have relative heights $h$ or greater, 
remember that the relative height of $n$ is assigned before all parts succeeding it, 
so the right leg of $n$ cannot be used in determining the relative height 
of any part succeeding it.  

We now construct the base partitions.  

\begin{prop}
\label{propBasePtns}
  For arbitrary but fixed $i=$ 0, 1, 2, \ldots, $\ell$, 
  set $k_i = 1$ and the rest of the $k_\cdot$'s $=0$.  
  Choose and fix non-negative integers $n_1$, $n_2$, \ldots, $n_\ell$.  
  Then, there is a unique admissible CMPP partition 
  having exactly $n_s$ parts with relative height $s-1$ 
  for $s = $ 1, 2, \ldots, $\ell$, 
  and its weight is 
  \begin{align}
  \label{eqBasePtnWeight}
    N_1^2 + N_2^2 + \cdots + N_\ell^2 
    + L_{\ell, i}(N_1, N_2, \ldots, N_\ell)
  \end{align}
  where $N_s = n_s + n_{s+1} + \cdots + n_\ell$ 
  for $s = $ 1, 2, \ldots, $\ell$, 
  and the linear forms $L_{\ell, i}(N_1, N_2, \ldots, N_\ell)$ 
  follow the pattern below.  
  {\allowdisplaybreaks \begin{align*}
    L_{\ell, 0}(N_1, N_2, \ldots, N_\ell) 
      & = N_1 + N_2 + \cdots + N_\ell \\ 
    L_{\ell, \ell}(N_1, N_2, \ldots, N_\ell) 
      & = N_2 + N_3 + \cdots + N_\ell \\ 
    L_{\ell, 1}(N_1, N_2, \ldots, N_\ell) 
      & = N_3 + N_4 + \cdots + N_\ell \\ 
    L_{\ell, \ell-1}(N_1, N_2, \ldots, N_\ell) 
      & = N_4 + N_5 + \cdots + N_\ell \\ 
    & \vdots
  \end{align*}}
\end{prop}

\begin{proof}
  We showed that in a minimal configuration, 
  the parts with smaller relative height 
  precede the parts with greater relative height, 
  and the parts having the same relative height $h$ 
  are level with difference $2h+2$ between them.  
  Observe that the initial conditions $k_j$ for $j =$ 1, 2, \ldots $\ell$ 
  have the same vertical alignment as $-1$'s, if we had any.  
  $k_0$ is different.  
  \begin{figure}
    \includegraphics[scale=0.3]{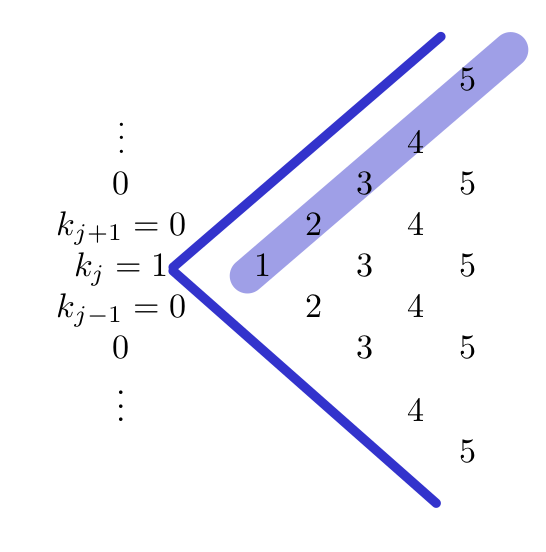} \\
    \caption{Left end of a diagram}
    \label{figLeftEnd}
  \end{figure}
  For an arbitrary but fixed $j = $1, 2, \ldots, $\ell$, 
  consider the left end of a diagram for $k_j = 1$.  
  We first want to find the smallest part with relative height $h$.  
  We cannot use any part on or above the right arm of $k_j$, 
  or any part on or below the right leg of $k_j$.  
  Since we want the smallest possible part, 
  we use the parts just below the right arm of $k_j$ 
  as marked in Figure \ref{figLeftEnd}.  
  Their relative heights are greater than the parts equal to them and below them.  
  Let us record the absolute and relative heights of these parts in the following table.  
  
  \begin{center}
  \begin{tabular}{ccc}
    part & absolute height & relative height \\ \hline 
    1 & $(j-1)$ & 0 \\
    2 & $j$ & 1 \\
    3 & $j$ & 2 \\ 
    & $\vdots$ & \\ 
    $(h+1)$ & $( j-1 + \lceil \frac{h}{2} \rceil )$ & $h$ \\ 
    & $\vdots$ & 
  \end{tabular}
  \end{center}
  
  There are two possibilities that this streak breaks.  
  Either the relative height and the absolute height become equal, 
  as shown in Figure \ref{figRightLegShorter},
  and the relevant line in the table above becomes 
  \begin{center}
  \begin{tabular}{ccc}
    part & absolute height & relative height \\ \hline 
    & $\vdots$ & \\ 
    $2j$ & $(2j-1)$ & $(2j-1)$
  \end{tabular}, 
  \end{center}
  \begin{figure}
    \includegraphics[scale=0.2]{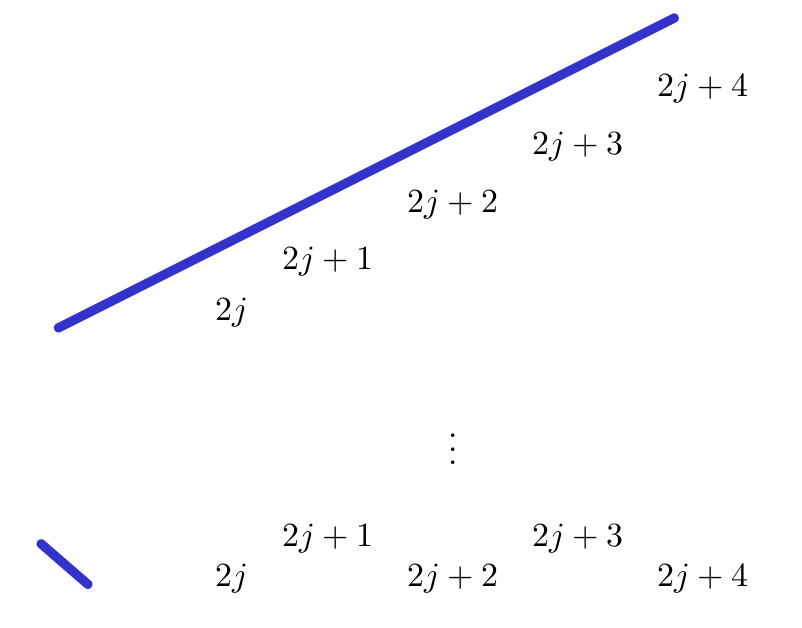} \\
    \caption{The right leg is shorter than the right arm}
    \label{figRightLegShorter}
  \end{figure}
  or the absolute height becomes $(\ell-1)$, 
  as shown in Figure \ref{figRightLegShorter},
  and the relevant line in the table becomes 
  \begin{center}
  \begin{tabular}{ccc}
    part & absolute height & relative height \\ \hline 
    & $\vdots$ & \\ 
    $2(\ell - j) + 1$ & $(\ell-1)$ & $2(\ell-j)$
  \end{tabular}.  
  \end{center}
  \begin{figure}
    \includegraphics[scale=0.2]{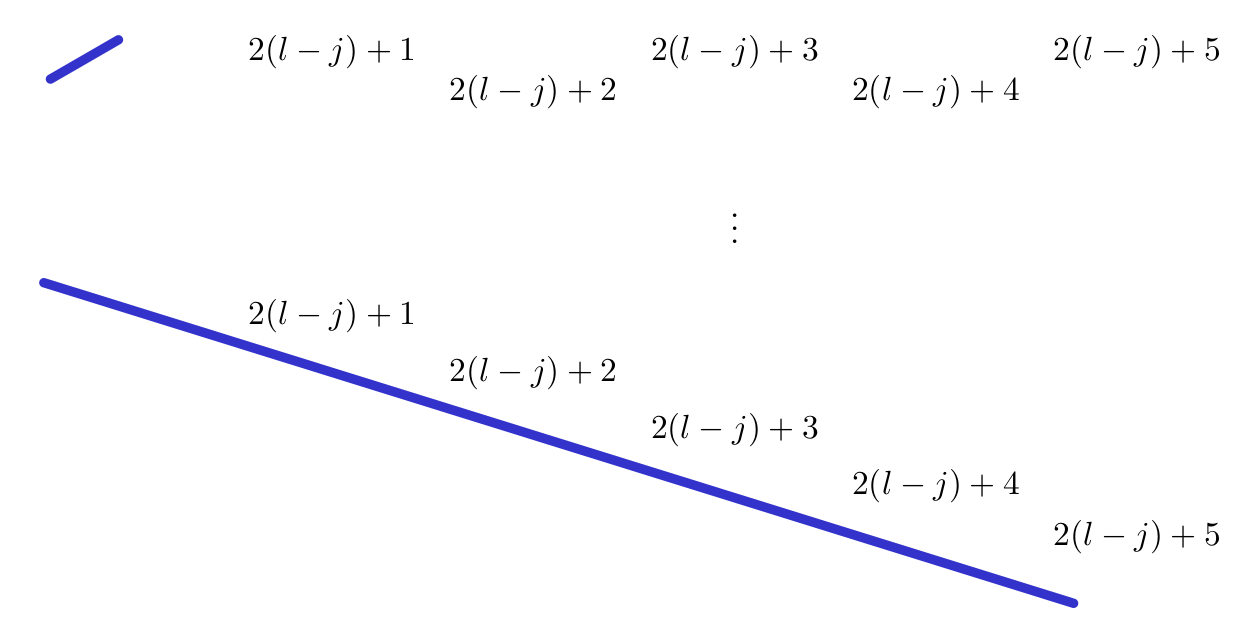} \\
    \caption{The right arm is shorter than the right leg}
    \label{figRightArmShorter}
  \end{figure}
  Due to the parity mismatch of the involved parts, 
  these cannot happen at the same time.  
  In either of these cases, 
  the relative height increases by one for every other part, 
  as opposed to every part in the original table corresponding to Figure \ref{figLeftEnd}.  
  
  In the former possibility, the table continues as 
  \begin{center}
  \begin{tabular}{ccc}
    part & absolute height & relative height \\ \hline 
    $2j$ & $(2j-1)$ & $(2j-1)$ \\ 
    $2j+2$ & $2j$ & $2j$ \\ 
    & $\vdots$ & \\ 
    $2(h-j+1)$ & $h$ & $h$ \\ 
    & $\vdots$ & \\ 
    $2(\ell - j)$ & $(\ell-1)$ & $(\ell-1)$ 
  \end{tabular}, 
  \end{center}
  and in the latter possibility as 
  \begin{center}
  \begin{tabular}{ccc}
    part & absolute height & relative height \\ \hline 
    $2(\ell-j)+1$ & $(\ell-1)$ & $2(\ell-j)$ \\ 
    $2(\ell-j)+3$ & $(\ell-1)$ & $2(\ell-j)+1$ \\ 
    & $\vdots$ & \\ 
    $2(h-\ell+j)+1$ & $(\ell-1)$ & $h$ \\ 
    & $\vdots$ & \\ 
    $(2j-1)$ & $(\ell-1)$ & $(\ell-1)$ 
  \end{tabular}.  
  \end{center}
  
  When $k_0 = 1$ and the other $k_j$'s are zero, 
  we draw the right arm of $k_0$, and the figure becomes as below.  
  \begin{figure}
    \includegraphics[scale=0.2]{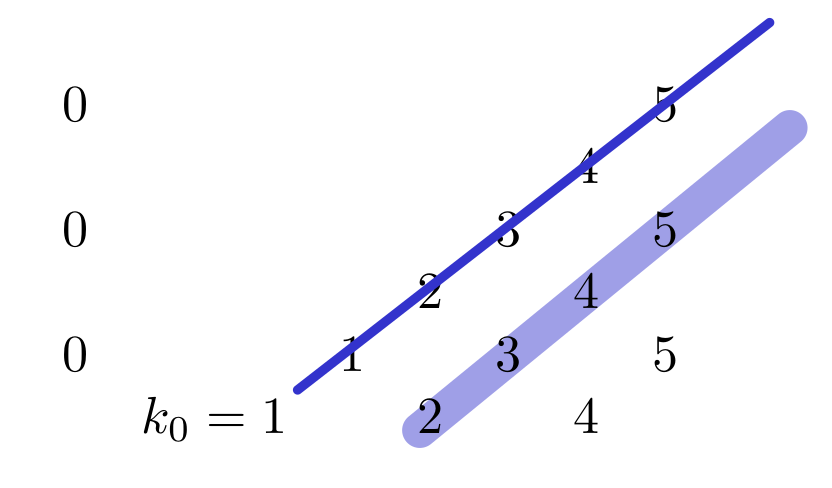} \\
    \caption{The left end of a diagram when $k_0 = 1$}
    \label{fig_k0_is_1}
  \end{figure}
  The table of the smallest parts with a given height 
  is also given below.  
  The skipped odd parts share the same relative height and absolute height 
  as the even parts preceding them.  
  \begin{center}
  \begin{tabular}{ccc}
    part & absolute height & relative height \\ \hline 
    2 & 0 & 0 \\ 
    4 & 1 & 1 \\ 
    6 & 2 & 2 \\ 
    & $\vdots$ & \\ 
    $2(h+1)$ & $h$ & $h$ \\ 
    & $\vdots$ & \\ 
    $2\ell$ & $(\ell - 1)$ & $(\ell - 1)$ 
  \end{tabular}.  
  \end{center}
  
  When a part $n$ with relative height $h$ is placed in a base partition, 
  we encounter the portion of a diagram shown in Figure \ref{fig_n_placed}.  
  \begin{figure}
    \includegraphics[scale=0.2]{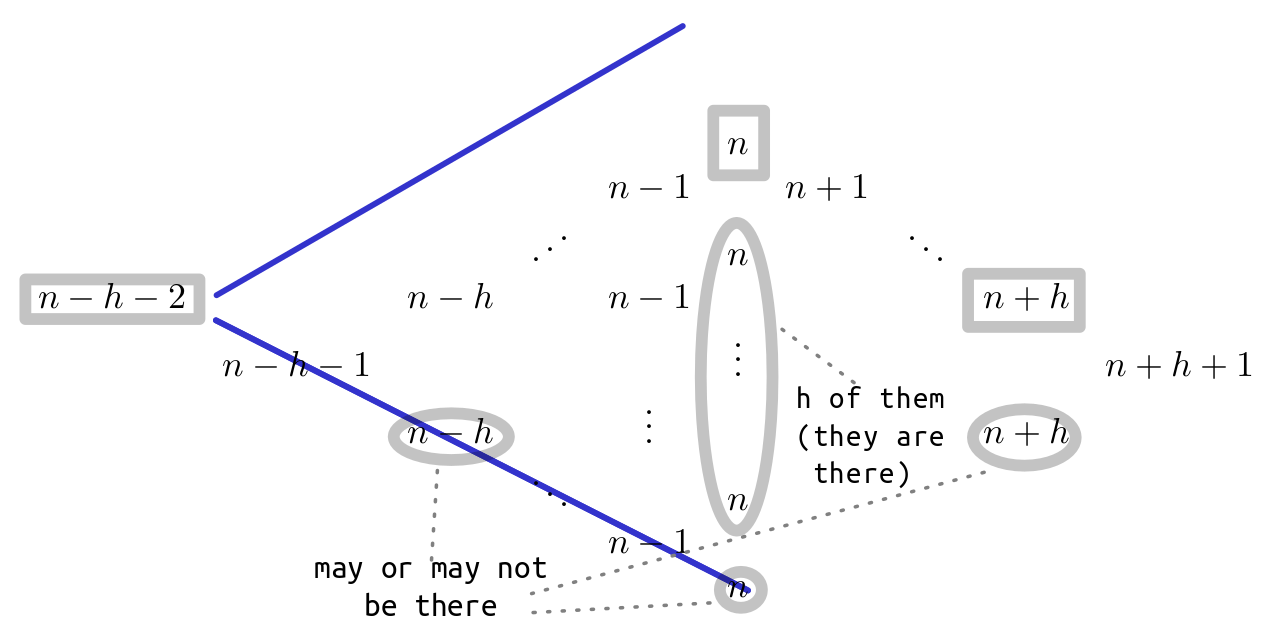} \\
    \caption{Just placed the boxed $n$}
    \label{fig_n_placed}
  \end{figure}
  $(n-h-1)$ may be zero, and hence $(n-h-2)$ may be -1.  
  In other words, $(n-h-2)$ may be 
  one of the initial conditions $k_j = 1$ for $j > 0$.  
  Because this is a tight configuration, 
  we could not have used any part on or above 
  the right arm of the boxed $n-h-2$, 
  and the relative height of $n$ 
  is determined by the right leg of the boxed $n-h-2$.  
  
  Now, the succeeding part cannot be on or above the right arm of the boxed $n+h$.  
  Because otherwise either the relative height 
  of the previously placed $n$ (the boxed $n$) is forced to be less than $h$, 
  or the resulting partition would not be admissible.  
  Similarly, the next part to be placed cannot be on 
  or below the right leg of the boxed $n+h$.  
  Therefore; when such $n$ is placed, 
  we can think of the boxed $n-h-2$ and the boxed $n+h$ 
  as initial conditions that are equal to 1.  
  The boxed $n-h-2$ and the boxed $n+h$ are level, 
  and they have the same parity.  
  So, we can think of the placement of the succeeding part 
  as starting over after increasing each part by $2h+2$.  
  
  This is essentially how we construct the base partition 
  when the placed part is still between the right arm 
  and the right leg of the initial condition.  
  That is, when both the right arm and the right leg of the initial condition 
  contain parts that are equal top the newly placed part.  
  This part of the construction is uniform for 
  the presence and the absence of parts with a fixed relative height $h$.  
  In other words, we do not need to adjust the construction according to 
  how many parts there are with a given relative height $h$.  
  Please keep in mind that per the first table above, 
  increasing a part by one in the diagonal direction 
  will increase the relative height by one in this case.  
  
  When we are past the right leg of the initial condition, 
  and hence the absolute heights and the relative heights become equal, 
  the last part $n$ with relative height $h$ can tightly be followed 
  by a level $n+2h+2$ with relative height $h$, 
  or by an $n+2h+4$ with a relative height $h+1$, etc.  
  No smaller part than the level $n+2h+2$ 
  will both have relative height $h$ and respect the relative height of 
  the preceding $n$ with relative height $h$.  
  Ditto for the $n+2h+4$ with relative height $h+1$, etc.  
  We can bring the shifted imaginary initial conditions into the discussion, 
  but it is not necessary.  
  
  
  Like in the preceding case, 
  the discussion goes the same if we skip a certain relative height $h$ 
  or use an arbitrary number of parts with that relative height.  
  Unlike in the preceding case, 
  the next placed part must be increased by two 
  in order to increase the relative height by one.  
  
  The third case is when we are past the right arm of the initial condition, 
  and place an $n$ with relative height $h$ and absolute height necessarily $(\ell - 1)$.  
  In this case, the $n+2h+2$ with absolute height $(\ell - 1)$ 
  will have relative height $h$, 
  the $n+2h+4$ with absolute height $(\ell - 1)$ 
  will have relative height $h+1$, etc. 
  No smaller part than the said parts will both have the said respective relative heights 
  and respect the relative height of the previously placed $n$ at the same time.  
  In this case, as well, 
  the next placed part must be increased by two 
  in order to increase the relative height by one.  
  Like in the preceding cases, 
  the discussion goes the same if we skip a certain relative height $h$ 
  or use an arbitrary number of parts with that relative height.  
  
  Finally, we calculate the weights of the base partitions.  
  Suppose we choose and fix $n_1$, $n_2$, \ldots, $n_\ell$ $\geq 0$, 
  and there are exactly $n_{h+1}$ parts with relative height $h$ 
  for $h = $ 0, 1, \ldots, $(\ell - 1)$.  
  Bringing together the above, 
  the base partition with the greatest weight occurs when $k_0 = 1$, 
  and the other $k_j$'s are zero.  
  This is because the the parts with relative height zero 
  start with 2 in this case.  
  In all other cases, they start with 1.  
  Also, in the aforementioned case, 
  we skip two integers when incrementing the relative height by one.  
  In all other cases, we skip only one integer 
  until the relative height becomes equal to the absolute height, 
  or until the absolute height becomes $(\ell - 1)$.  
  In each case, the difference between two parts 
  with the same relative height $h$ is $2h+2$.  
  
  Thus, for $k_0=1$, the weight of the base partition is as follows.  
  \begin{align*}
    \underbrace{\left[ 2 + 4 + \cdots + 2 n_1 \right] }_{ 
      \textrm{ parts with relative height } 0 }
    + \underbrace{\left[ (2n_1 + 4) + (2n_1 + 8) + \cdots + ( 2n_1 + 4n_2 ) \right] }_{
      \textrm{ parts with relative height } 1 }
  \end{align*}
  \begin{align*}
    + \cdots 
    + \underbrace{\left[ (2n_1 + 4n_2 + \cdots + 2(s-1)n_{s-1} + 2s) 
        + (2n_1 + 4n_2 + \cdots + 2(s-1)n_{s-1} + 4s ) \right.}_{
      \textrm{ parts with relative height } s-1 }
  \end{align*}
  \begin{align*}
    \underbrace{\left. + \cdots 
        + (2n_1 + 4n_2 + \cdots + 2(s-1)n_{s-1} + 2sn_s )\right]}_{
      \textrm{ still parts with relative height } s-1 }
  \end{align*}
  \begin{align*}
    + \cdots 
    + \underbrace{\left[ (2n_1 + 4n_2 + \cdots + 2(\ell-1)n_{\ell-1} + 2\ell) 
        + (2n_1 + 4n_2 + \cdots + 2(\ell-1)n_{\ell-1} + 4\ell ) \right.}_{
      \textrm{ parts with relative height } \ell-1 }
  \end{align*}
  \begin{align*}
    \underbrace{\left. + \cdots 
        + (2n_1 + 4n_2 + \cdots + 2(\ell-1)n_{\ell-1} + 2\ell n_\ell )\right]}_{
      \textrm{ still parts with relative height } \ell-1 }
  \end{align*}
  \begin{align*}
    = \left( n_1^2 + n_1 \right) 
    + \left( 2 n_1 n_2 + 2 n_2^2 + 2 n_2 \right) 
    + \left( 2 n_1 n_3 + 4 n_2 n_3 + 4 n_3^2 + 4 n_3 \right) 
  \end{align*}
  \begin{align*}
    + \cdots 
    + \left( 2 n_1 n_s + 4 n_2 n_s + \cdots 
      + 2 (s-1) n_{s-1} n_s + s n_s^2 + s n_s \right)
  \end{align*}
  \begin{align*}
    + \cdots 
    + \left( 2 n_1 n_\ell + 4 n_2 n_\ell + \cdots 
      + 2 (\ell-1) n_{\ell-1} n_\ell + \ell n_\ell^2 + \ell n_\ell \right)
  \end{align*}
  \begin{align*}
    = \left( n_1 + n_2 + \cdots + n_\ell \right)^2 
    + \left( n_2 + n_3 + \cdots + n_\ell \right)^2 
    + \cdots 
    + \left( n_{\ell-1} + n_\ell \right)^2
    + \left( n_\ell \right)^2 
  \end{align*}
  \begin{align*}
    + \left( n_1 + n_2 + \cdots + n_\ell \right)
    + \left( n_2 + n_3 + \cdots + n_\ell \right) 
    + \cdots 
    + \left( n_{\ell-1} + n_\ell \right)
    + \left( n_\ell \right)
  \end{align*}
  \begin{align}
  \label{eq_weightOfBasePtn_for_k0_1}
    = N_1^2 + N_2^2 + \cdots + N_\ell^2 + N_1 + N_2 + \cdots + N_\ell.  
  \end{align}
  
  When $k_\ell = 1$, the smallest part is 1.  
  One of the edge cases, 
  namely the relative heights becoming equal to $\ell-1$, 
  is reached after relative height zero, 
  so the difference between the largest part with relative height $h$
  and the smallest part with relative height $h+1$ is equal to $2h+2$ 
  for all $h$.  
  This amounts to reducing the weights 
  of all parts in \eqref{eq_weightOfBasePtn_for_k0_1} by 1, 
  i.e. reducing it by $N_1$.  
  Thus, when $k_\ell = 1$, the weight of the base partition is 
  \begin{align*}
    = N_1^2 + N_2^2 + \cdots + N_\ell^2 + N_2 + N_3 + \cdots + N_\ell.  
  \end{align*}
  
  When $k_1 = 1$, 
  on top of the smallest part being 1, 
  the edge case in which the absolute heights 
  becoming equal to the absolute heights 
  is reached after relative height one.  
  With similar considerations to the above paragraph, 
  the weight of the base partition when $k_1 = 1$ is
  \begin{align*}
    = N_1^2 + N_2^2 + \cdots + N_\ell^2 + N_3 + N_4 + \cdots + N_\ell.  
  \end{align*}
  
  This pattern continues and finishes the proof.  
\end{proof}

The pattern observed at the end of the proof of Theorem \ref{thmRussellMain} 
is one of the reasons why Russell chose to line up the initial conditions 
as in~\cite{Russell_starting_point}.  
If we followed~\cite{Russell_starting_point}, 
the tables at the beginning 
of the proof of Theorem \ref{thmRussellMain} would be more complicated.  

\section{Proofs of Results}
\label{secProofs}

\begin{proof}[Proof of Theorem \ref{thmRussellMain}]
  Let $n_1$, $n_2$, \ldots, $n_\ell$ be arbitrary but fixed non-negative integers.  
  
  We first construct a CMPP partition with $k=1$ and $k_i=1$ 
  from a base partition and a vector partition.  
  Construct the unique base partition 
  with exactly $n_s$ parts with relative height $s-1$ 
  for $s = $ 1, 2, \ldots, $\ell$ as in Proposition \ref{propBasePtns}.  
  This base partition has weight \eqref{eqBasePtnWeight}.  
  This is recorded by the exponent of $q$ 
  in the numerator of the general term in  the right hand side of \eqref{eqRussellMain}.  
  
  Take a vector partition
  \begin{align}
  \nonumber
    & \left( 
      (\lambda_{1 \; 1}, \lambda_{1 \; 2}, \ldots, \lambda_{1 \; n_1} ), 
      (\lambda_{2 \; 1}, \lambda_{2 \; 2}, \ldots, \lambda_{2 \; n_2} ), 
      \cdots 
      (\lambda_{s \; 1}, \lambda_{s \; 2}, \ldots, \lambda_{s \; n_s} ), \right. \\ 
  \label{vectorPtn}
      & \qquad \left. \cdots 
      (\lambda_{\ell \; 1}, \lambda_{\ell \; 2}, \ldots, \lambda_{\ell \; n_\ell} ) 
    \right)
  \end{align}
  such that each component has non-decreasing parts that are nonnegative.  
  To be precise, we allow zeros to appear in partitions to keep the lengths fixed.  
  This vector partition is generated by the $q$-factorials 
  in the denominator of the general term in  the right hand side of \eqref{eqRussellMain}.  
  
  For $s=$ $\ell$, $\ell-1$, \ldots, 1, in this order, 
  we move the largest part with relative height $s-1$ $\lambda_{s \; 1}$ times, 
  the second largest part with relative height $s-1$ $\lambda_{s \; 2}$ times, etc. 
  These moves are possible thanks to Proposition \ref{propPossibleMoves}, 
  and uniquely reversible by a remark in the proof of it.  
  
  Now we construct a vector partition and a base partition from a CMPP partition.  
  Given a CMPP partition $\doublestroke{\lambda}$, 
  find $n_s$, the number of parts with relative height $s-1$ 
  for $s =$ 1, 2, \ldots, $\ell$.  
  Then, for $s = $1, 2, \ldots, $\ell$, in this order, 
  move the smallest part with relative height $s-1$ 
  as far back as possible, and record the number of moves as $\lambda_{s \; n_s}$.  
  This will allow at least $\lambda_{s \; n_s}$ moves on the second smallest part 
  with relative height $s-1$.  
  Move that second smallest part with relative height $s-1$ as far back as possible, 
  and record the number of moves as $\lambda_{s \; n_s-1}$, etc.  
  We have observed that $\lambda_{s \; n_s}$ $\geq \lambda_{s \; n_s-1}$ $\geq \cdots$ 
  $\geq \lambda_{s \; 1}$.  
  This must produce the unique base partition with $n_s$ parts with relative height $s-1$ 
  for $s = $1, 2, \ldots, $\ell$.  
  Otherwise, at least one more backward move is possible on a part, 
  contrary to the procedure we just outlined.  
  We also obtain a vector partition in the form of \eqref{vectorPtn}.  
  
  A straightforward induction can be used to show that both maps are injective, 
  and we conclude the proof.  
  
\end{proof}

\subsection{A worked example}
\label{subsecWorkedExample}

We now take a CMPP partition with $\ell = 3$, $k = 1$, 
$k_0 =$ $k_1 =$ $k_3 =$ 0, $k_2 = 1$,
and perform the necessary backward moves 
to find the corresponding base partition and the auxiliary partitions.  
\begin{align*}
  \doublestroke{\lambda} = 
  3_0 + 7_1 + 14_2 + 23_2 + 26_1 + 34_1
\end{align*}
The subscripts indicate the absolute heights.  
We run the algorithm in Subsection \ref{subsecRelHghtDeterm} 
to determine the relative heights.  
\begin{align*}
  \doublestroke{\lambda} = 
  3_{(0)} + 7_{(1)} + 14_{(2)} + 23_{(2)} + 26_{(0)} + 34_{(1)}
\end{align*}
We then draw the upper envelope of the legs of parts and the initial condition.  
Orange legs correspond to relative height zero, 
green legs to one, 
and the blue legs to two.  
The right leg of the initial condition is gray.  
The part $3_0$ has degenerate legs, 
but we know that it necessarily has relative height zero.  
\begin{center}
  \includegraphics[scale=0.13]{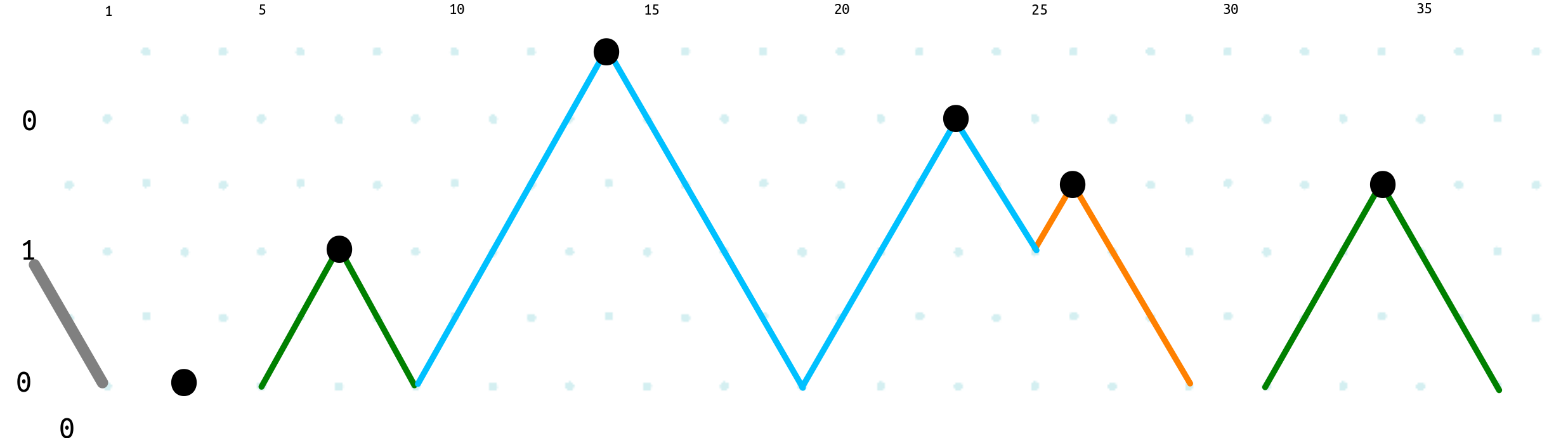}
\end{center}
We have two parts for each relative height 0, 1 and 2, 
so $n_1 = n_2 = n_3 = 2$.  
We work on parts with relative height zero, 
starting with the smallest.  
We move it as far back as possible, and record the number of moves.  
The smallest part with relative height zero moves backward 2 times, 
and the next smallest 17 times, 
thus $(\lambda_{1 \; 1}, \lambda_{1 \; 2}) = $ $(2, 17)$.  
The second smallest part switches relative heights 
with parts having greater relative heights three times.  

Next, we move the parts with relative height one backward, 
the smaller one first.  
The smaller part with relative height one moves backward 3 times.  
We explicitly show the backward moves on the larger part with relative height one.  
The part being moved is shown in red.  
\begin{center}
  \includegraphics[scale=0.13]{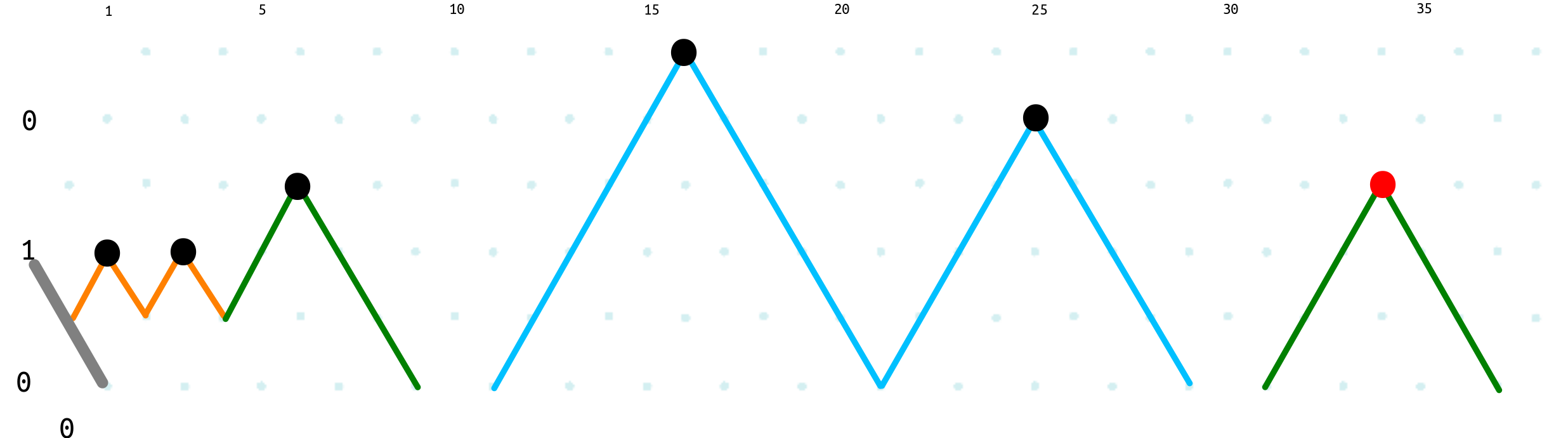}
  \includegraphics[scale=0.13]{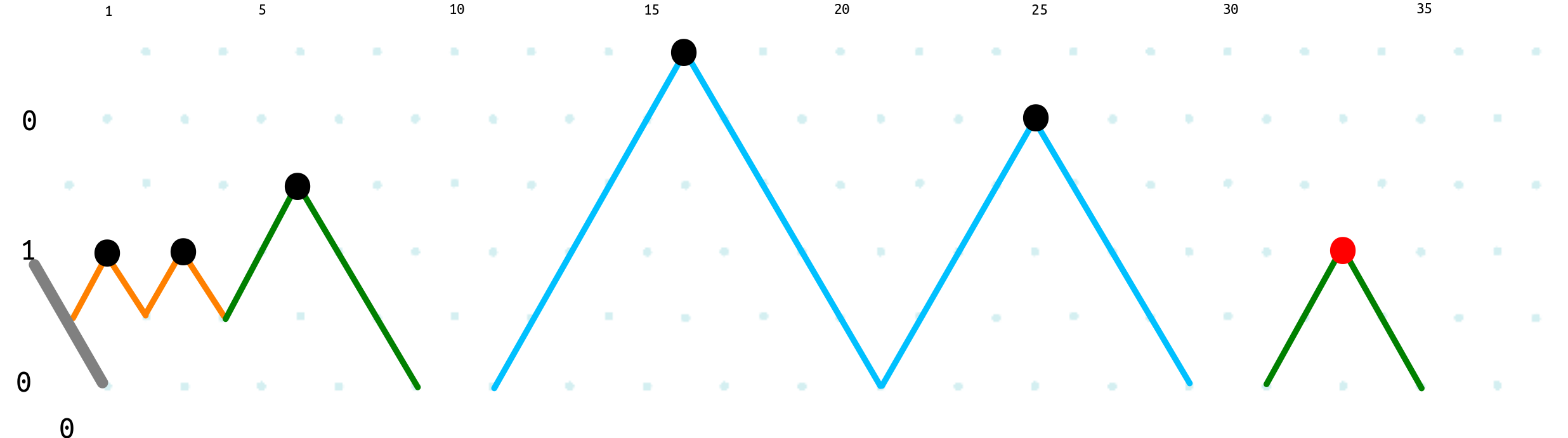}
  \includegraphics[scale=0.13]{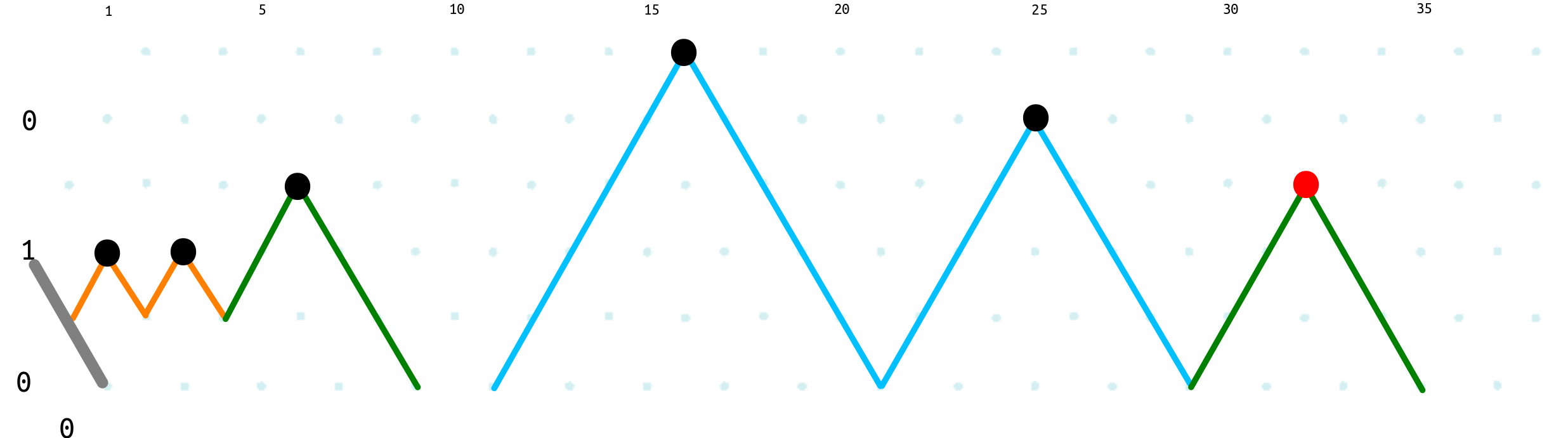}
  \includegraphics[scale=0.13]{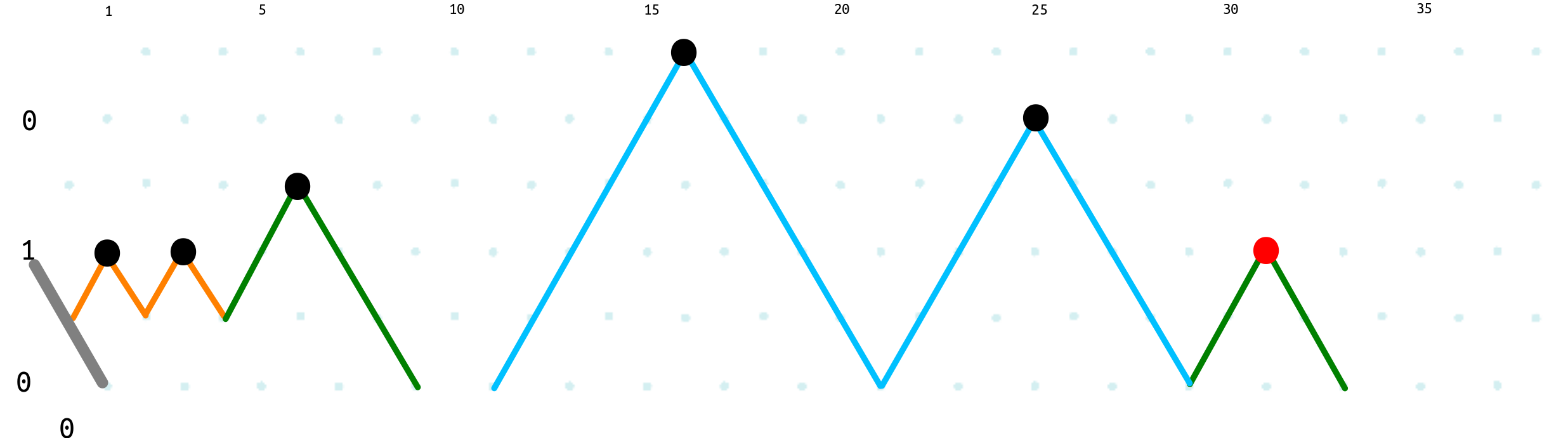}
  \includegraphics[scale=0.13]{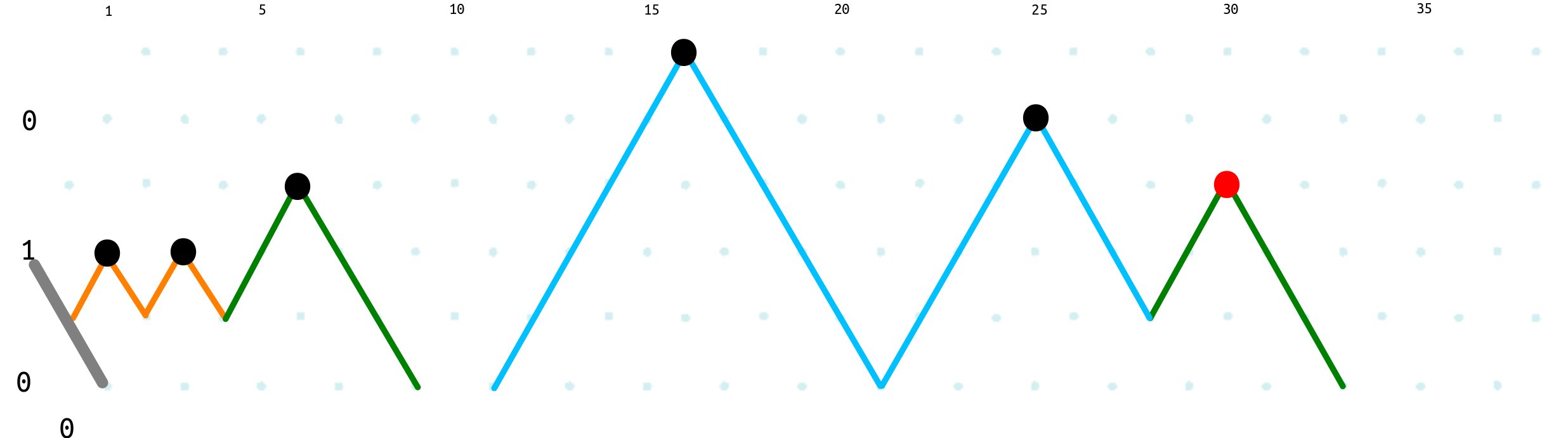}
  \includegraphics[scale=0.13]{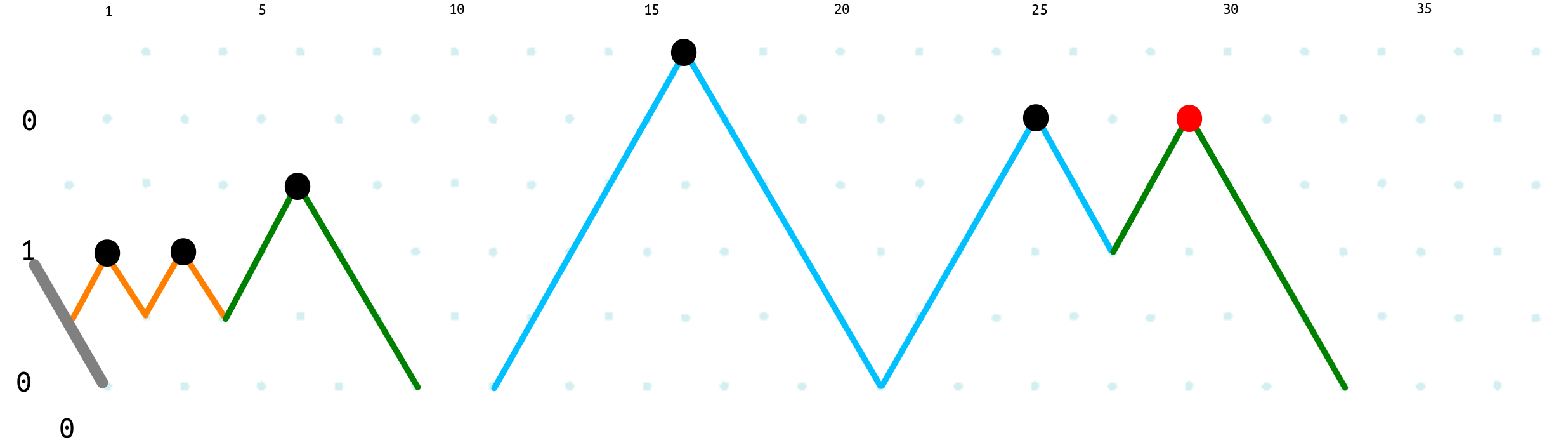}
  \includegraphics[scale=0.13]{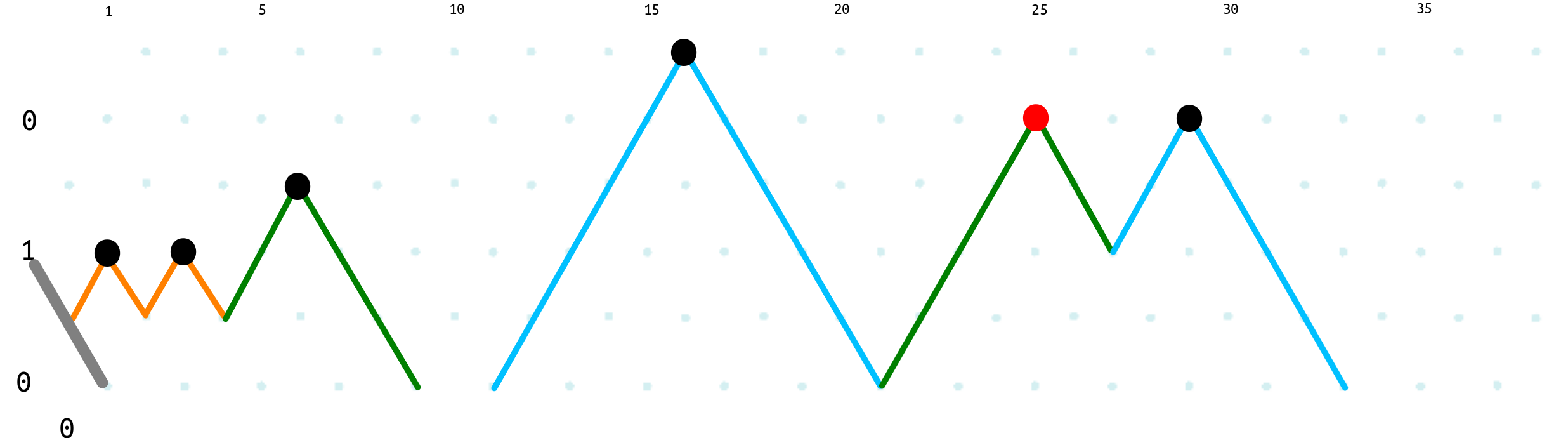}
  \includegraphics[scale=0.13]{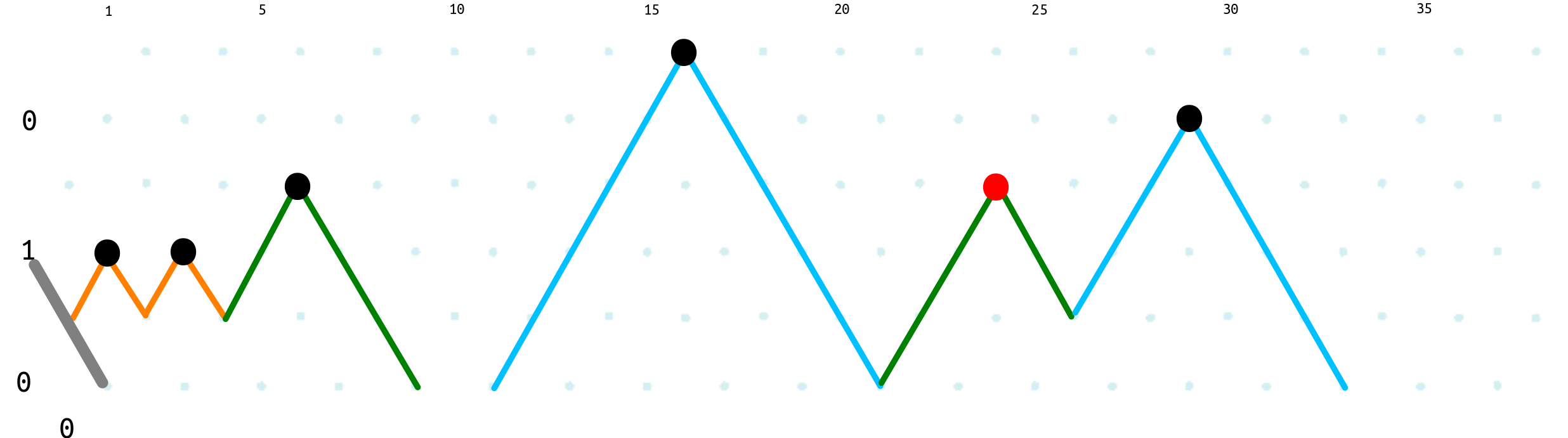}
  \includegraphics[scale=0.13]{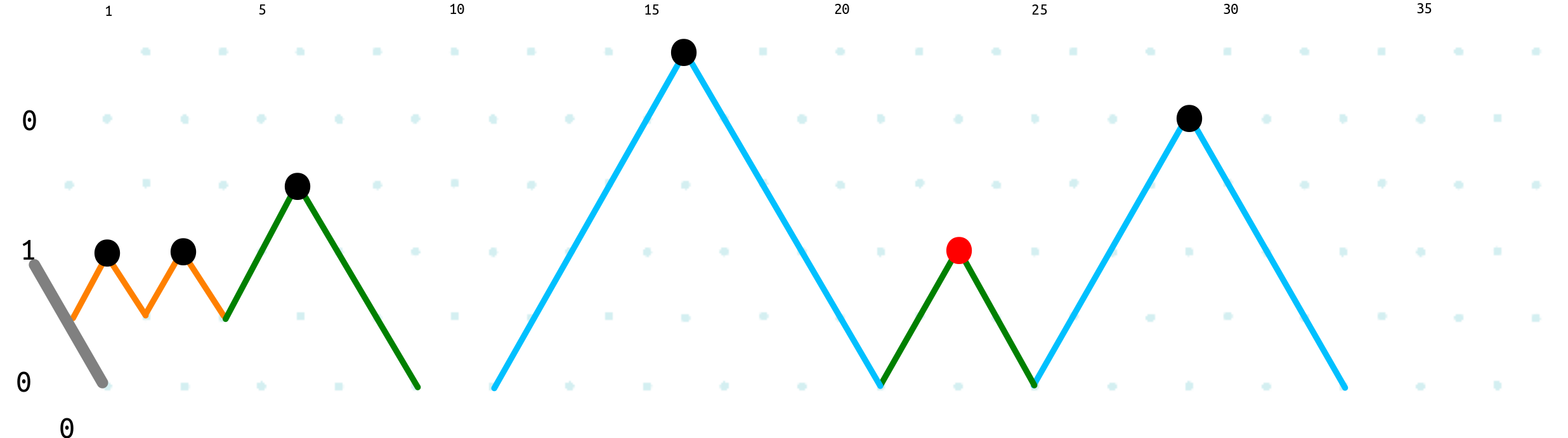}
  \includegraphics[scale=0.13]{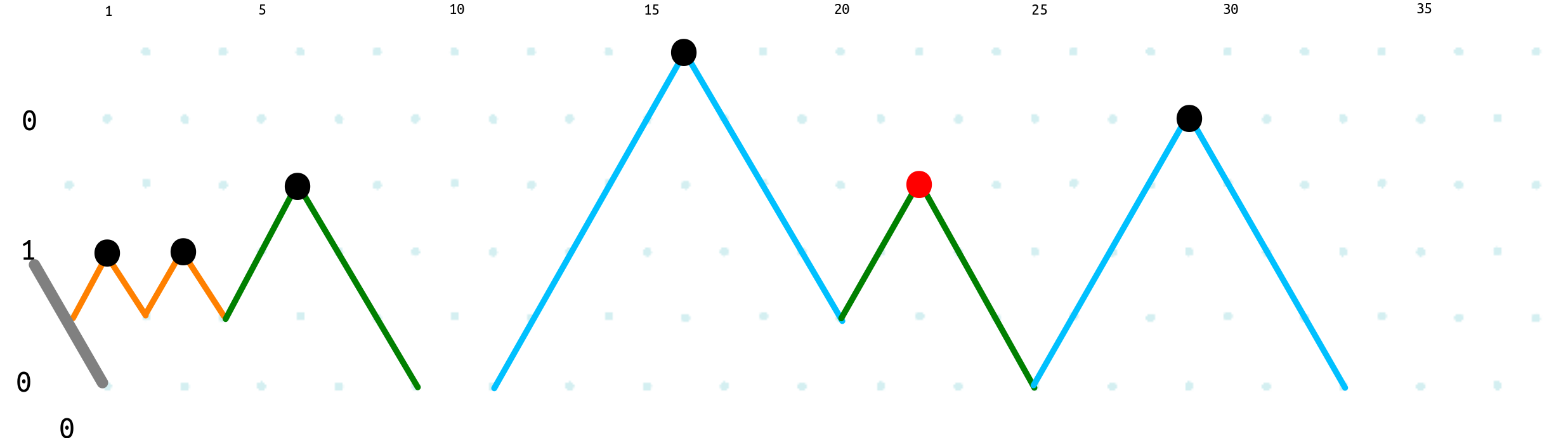}
  \includegraphics[scale=0.13]{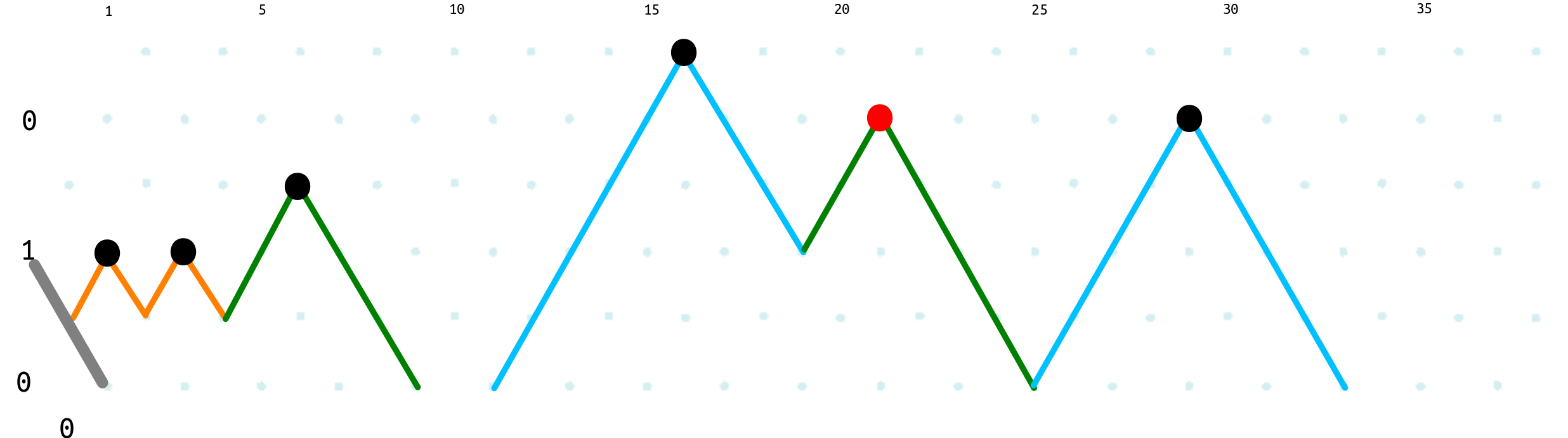}
  \includegraphics[scale=0.13]{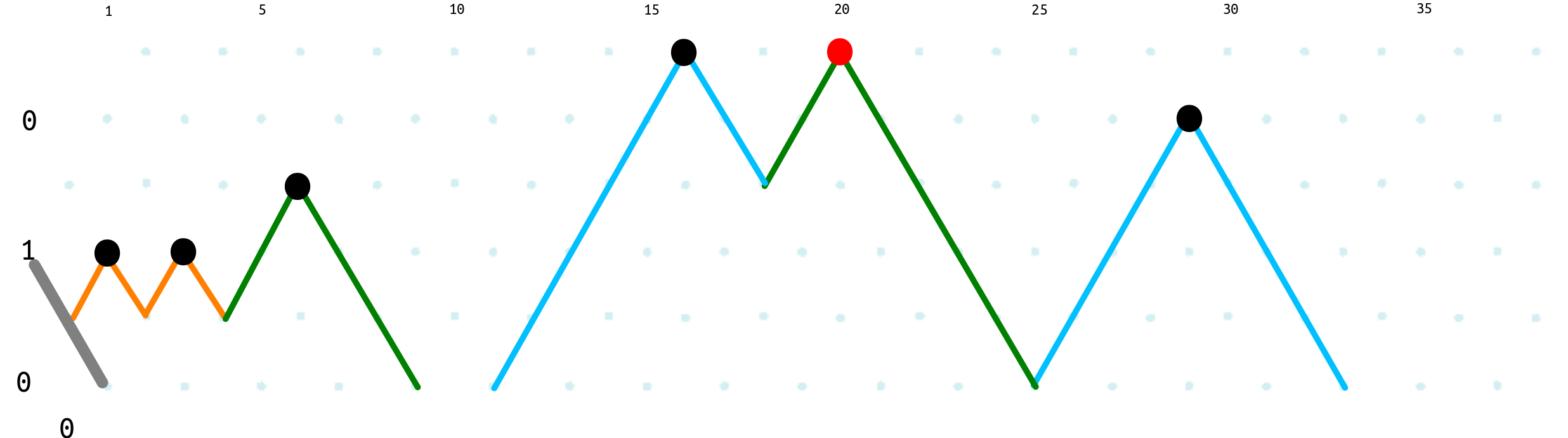}
  \includegraphics[scale=0.13]{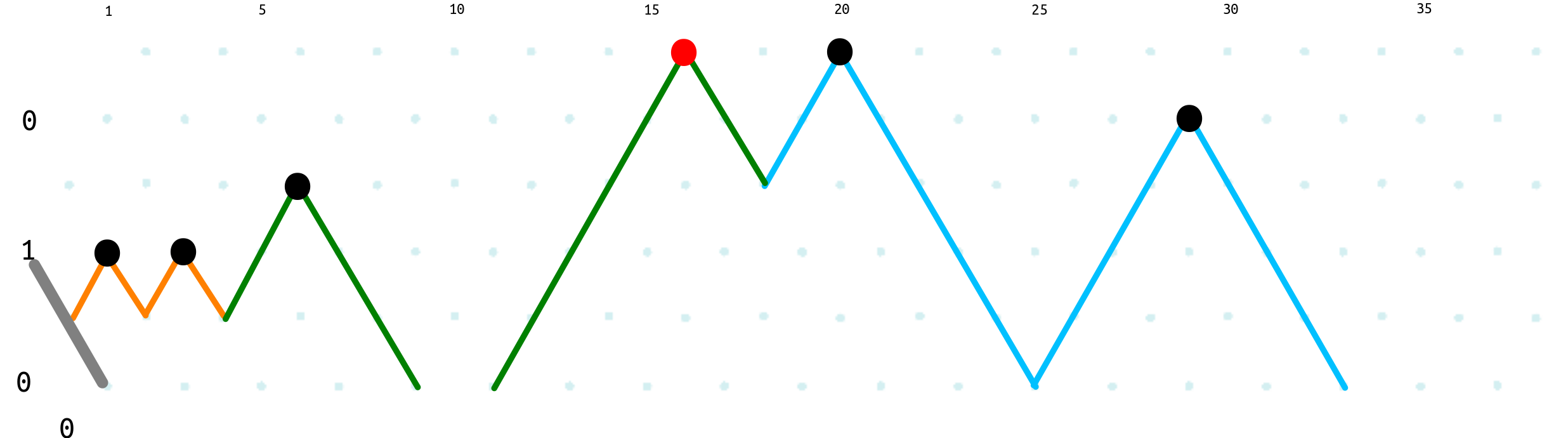}
  \includegraphics[scale=0.13]{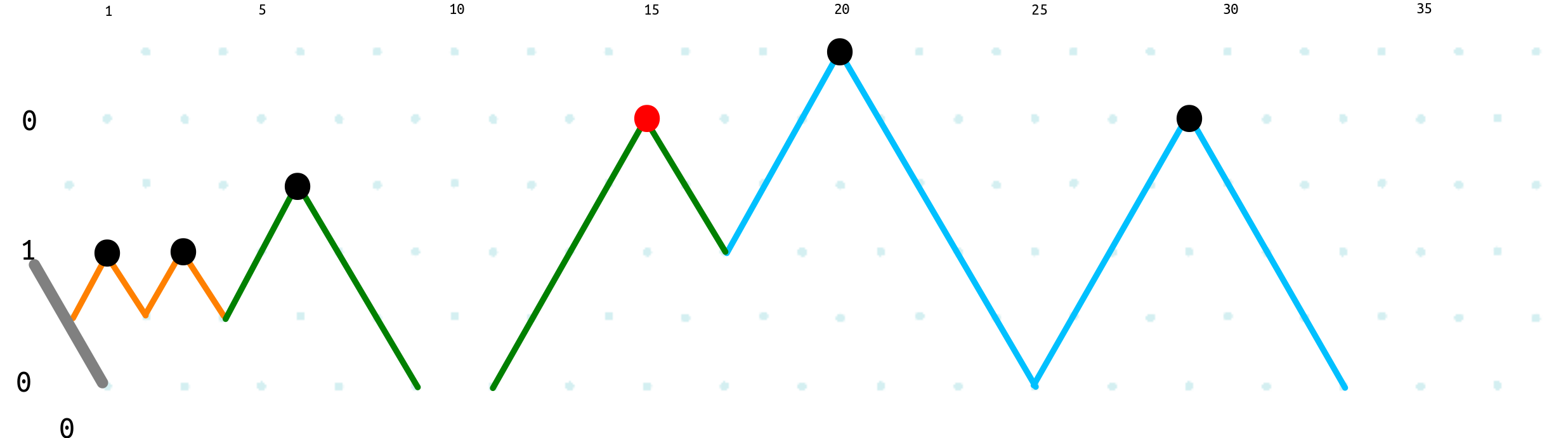}
  \includegraphics[scale=0.13]{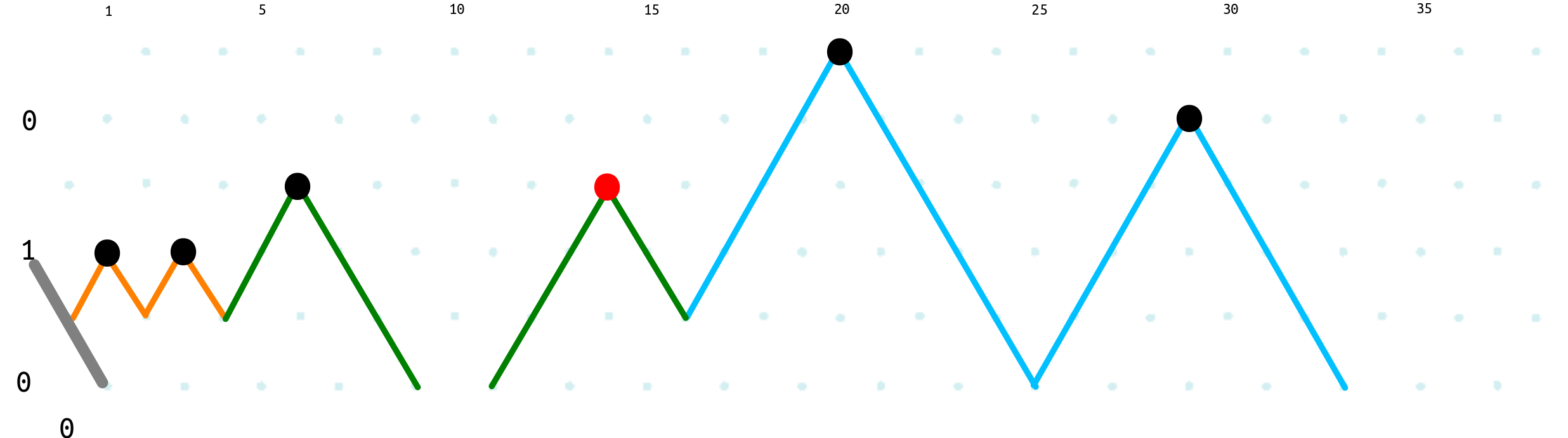}
  \includegraphics[scale=0.13]{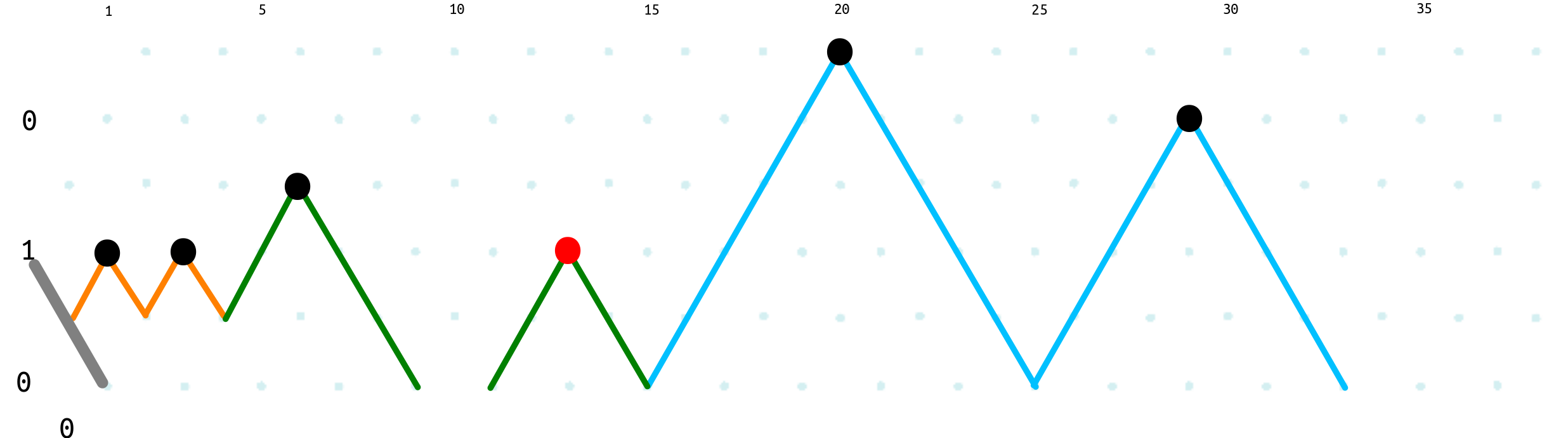}
  \includegraphics[scale=0.13]{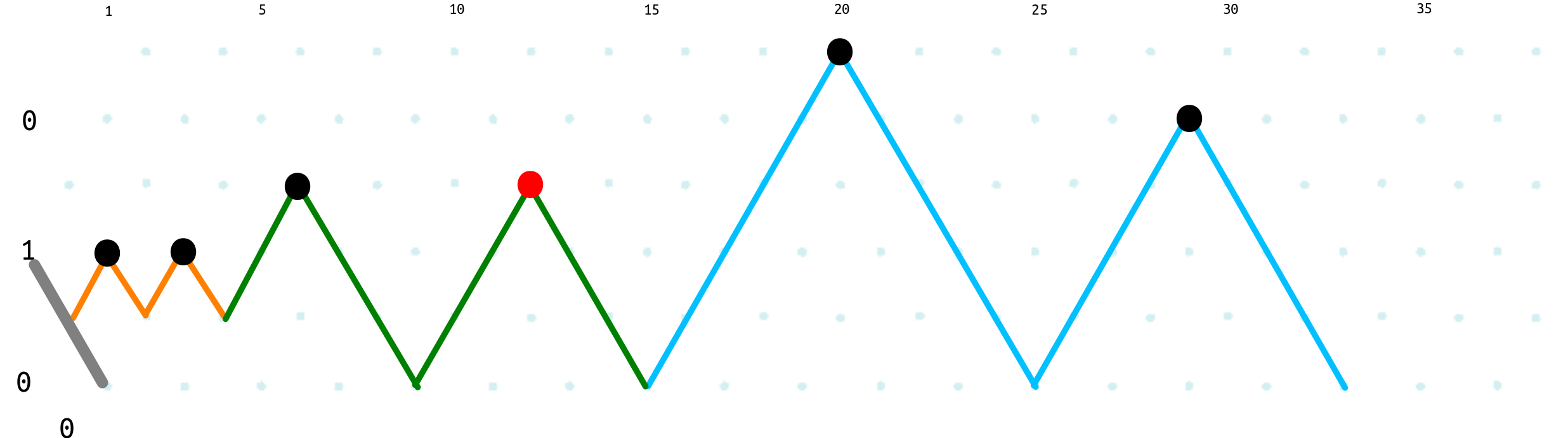}
  \includegraphics[scale=0.13]{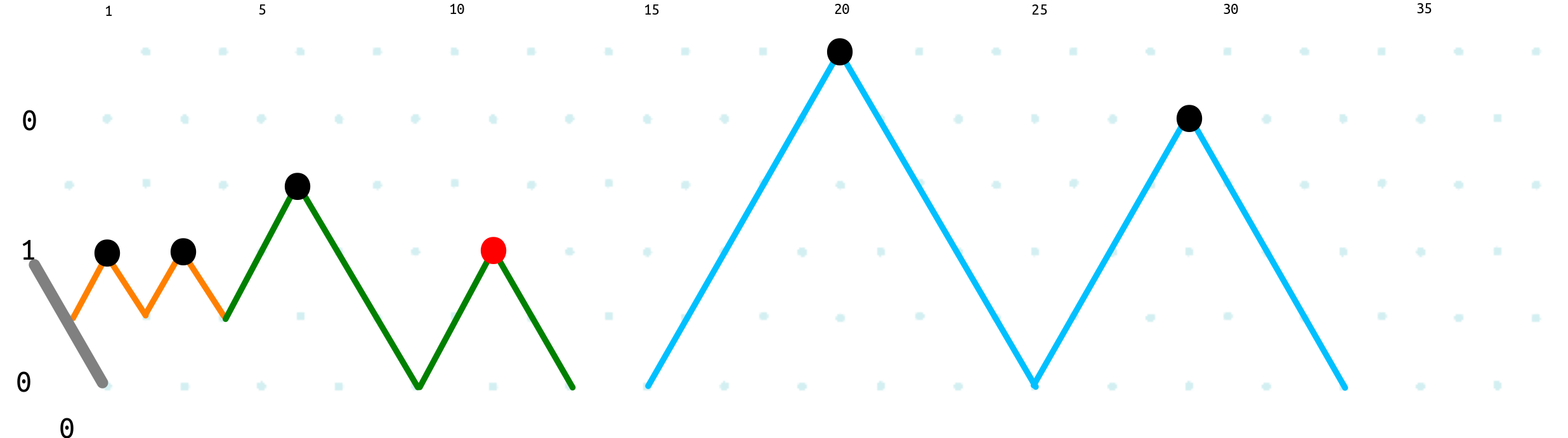}
  \includegraphics[scale=0.13]{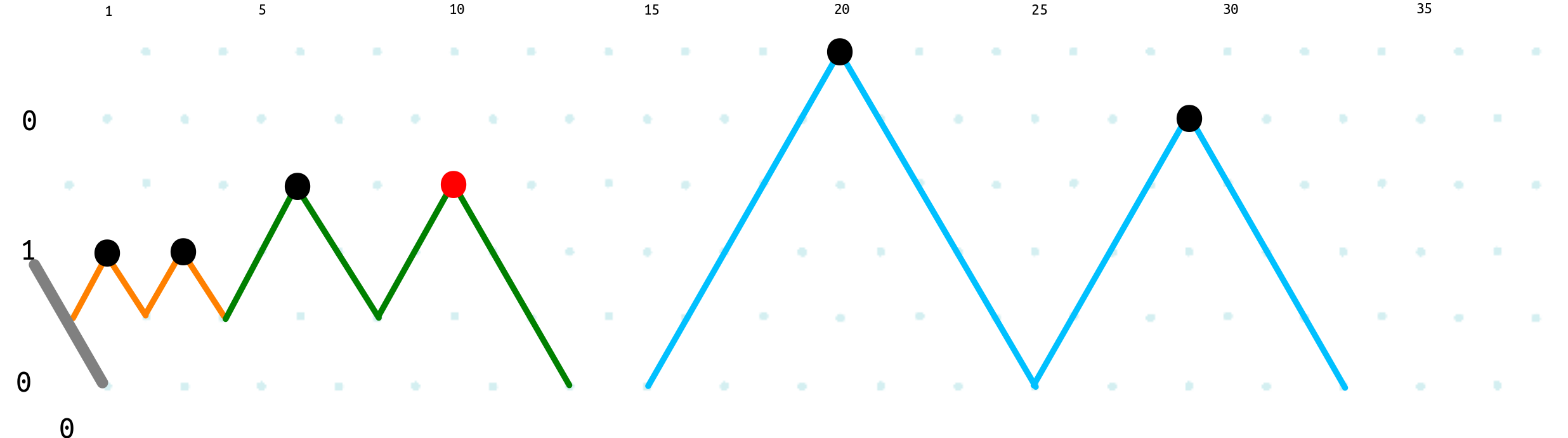}
\end{center}
The larger part with relative height one moved backward 16 times.  
It switched heights with parts with larger relative heights twice.  
We have $(\lambda_{2 \; 1}, \lambda_{2 \; 2}) = $ $(3, 16)$.  

Finally, we deal with the parts with relative height two.  
We only show the last configuration.  
\begin{center}
  \includegraphics[scale=0.13]{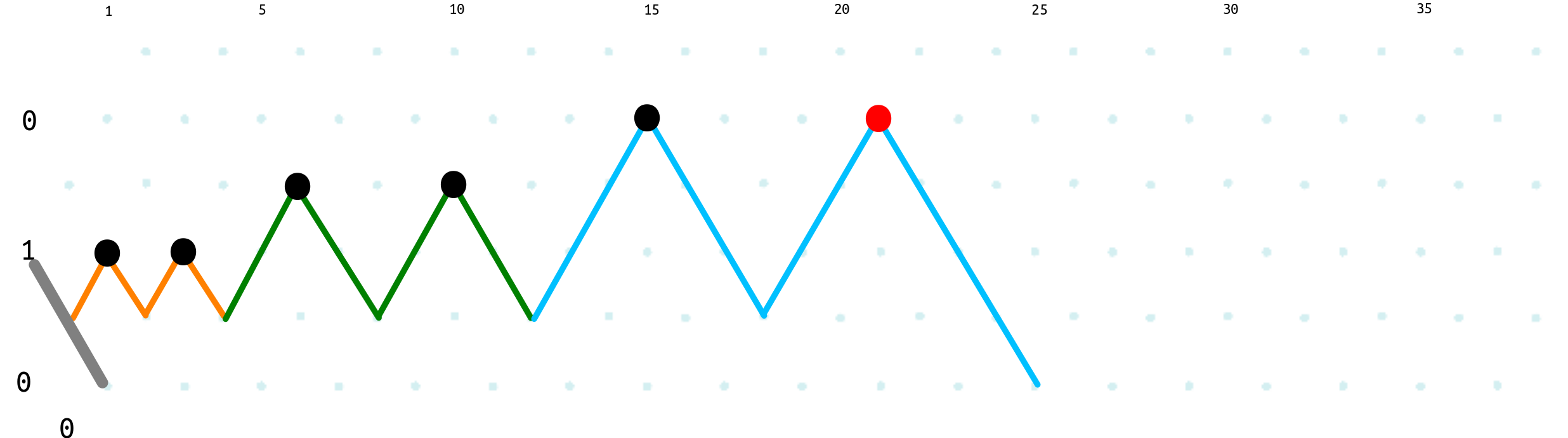}
\end{center}
Similar to the above reckoning, 
We have $(\lambda_{3 \; 1}, \lambda_{3 \; 2}) = $ $(5, 8)$.  

Please observe that it is not possible to move 
any of the parts backward without altering the relative heights, 
or without ending up with an inadmissible CMPP partition.  
Therefore, the last configuration 
\begin{align*}
  \doublestroke{\beta} 
  = 1_{(0)} + 3_{(0)} + 6_{(1)} + 10_{(1)} + 15_{(2)} + 21_{(2)} 
\end{align*}
is a base partition.  
The weight of this base partition $\doublestroke{\beta}$ 
is 
\begin{align*}
  N_1^2 + N_2^2 + N_3^2 
  = (2+2+2)^2 + (2+2)^2 + 2^2 
  = 56
\end{align*}
as asserted in Proposition \ref{propBasePtns}.  

Now we present the proofs of Theorems \ref{thmBSLT1} and \ref{thmBSLT2}.  

\begin{proof}[Proof of Theorem \ref{thmBSLT1}]
  We will explain the deviations from the construction above.  
  The definitions of absolute and relative heights carry over.  
  There are two main differences from the previous construction. 
  \begin{itemize}
  \item The parts with absolute height $\ell-1$ 
    or with relative height $\ell-1$ can only be even.  
  \item The parts with relative height $\ell-1$ can 
    only move forward by $+2$, and backward by $-2$.  
  \end{itemize}
  We first show that forward or backward moves of parts with relative heights 
  less than $\ell-1$ will not be effected.  
  
  Consider a part $n$ with relative height $h < \ell-1$.  
  If the relative height of the part $n$ 
  is determined by the right leg of a predecessor, 
  then the part will become the $n+1$ below the moving $n$ 
  (see the penultimate figure above Proposition \ref{propPossibleMoves}).  
  So, no parts moving forward have to become a deleted odd part 
  in the top row due to their predecessors.  
  
  If the relative height of $n$ is determined by the left leg of a successor, 
  then $n$ may only go as high as level with its immediate successor.  
  Then, it either stops of switches relative heights with its immediate successor.  
  At that point, the relative height of the moving part 
  will be determined by the right leg of a predecessor.  
  By induction, no part will have to become an odd part 
  in the deleted top row due to a successor.  
  
  By similar arguments, the absence of any parts 
  in the deleted top row of odd parts will not prevent 
  any possible forward moves, so Proposition \ref{propPossibleMoves} 
  can be easily adjusted.  
  
  The construction of the base partition must be outlined, as well.  
  One needs some more bookkeeping in the proof of Proposition \ref{propBasePtns}.  
  Observe that in that proof, 
  the absolute heights and the relative heights are non-decreasing.  
  In particular, when the absolute height reaches $\ell-1$, 
  it stays $\ell-1$.  
  Moreover, once the relative heights hit $\ell-1$, 
  then the parity of the placed parts stay the same.  
  Thus, we start by determining the relative height (r.h.) 
  of the smallest part with absolute height (a.h.) $\ell-1$ 
  and its parity for various non-zero initial conditions.  
  \begin{center}
    \begin{tabular}{ccc}
      The non-zero 
      & r.h. of the smallest part
      & no. of parts \\
      initial condition &  
      with a.h. $=\ell-1$ &  on the deleted top row \\\hline 
      $k_0$ & $\ell-1$ & 0 \\ 
      $k_\ell$ & $0$ & $N_1$ \\
      $k_1$ & $\ell-1$ & 0 \\ 
      $k_{\ell-1}$ & $1$ & $N_2$ \\
      & $\vdots$ & 
    \end{tabular}
  \end{center}
  When we need an odd part in the top row in a base partition, 
  we know we cannot move it backward any more 
  without messing up the count of parts with fixed relative heights 
  (i.e. the vector $(n_1, n_2, \ldots, n_\ell)$).  
  In other words, we cannot make those inadmissible odd parts 
  admissible even parts one less each.  
  However, we can make them admissible even parts one \emph{more} each.  
  We shall know that they cannot move backward any further.  
  So, we have the base partition we are after.  
  This amounts to adding 1 to parts with a certain relative height $h$ or more, 
  i.e. adding $N_h$ to the weight of the base partition.  
  
  Let us also recall that parts with relative height $\ell-1$ 
  necessarily move two by two, and the others one by one.  
  This gives us the 
  \begin{align*}
    (q; q)_{n_1} (q; q)_{n_2} \cdots (q; q)_{n_{\ell-1}} (q^2; q^2)_{n_\ell} 
  \end{align*}
  in the denominator.    
\end{proof}

Before we give the proof of Theorem \ref{thmBSLT2}, albeit brief, 
we need to jot down a small result.  
Call the number of $[k_0, k_1, \ldots, k_\ell]-$admissible CMPP partitions 
on diagram \eqref{CMPPdiagAlt2Prep}
of $n$ with $j$ parts $\widetilde{F}(i,j,n)$.  
We set 
\begin{align*}
  \widetilde{P}_i(z) := \widetilde{P}_i(z; q) 
  := \sum_{ j,n \geq 0 } \widetilde{F}(i,j,n) \; q^n \; z^j.  
\end{align*}
Then, 
\begin{align*}
  \widetilde{P}_i(z) = P_i(z), 
\end{align*}
where $P_i(z)$ is given in Theorem \ref{thmRussellMain}.  
One can naturally think of this as a reflection of CMPP partitions 
on a horizontal mirror.  
But, we need to emphasize that the combinatorics is a little different 
since the absolute heights and the relative heights 
are determined by the parts directly below the designated part.  

\begin{proof}[Proof of Theorem \ref{thmBSLT2}]
  This is the proof of Theorem \ref{thmBSLT1} \emph{mutatis mutandis}.  
  The adjusted table is as follows.  
  \begin{center}
    \begin{tabular}{ccc}
      The non-zero 
      & r.h. of the smallest part
      & no. of parts \\
      initial condition &  
      with a.h. $=\ell-1$ &  on the deleted top row \\\hline 
      $k_0$ & $0$ & $N_1$ \\ 
      $k_\ell$ & $\ell-1$ & $0$ \\
      $k_1$ & $1$ & $N_2$ \\ 
      $k_{\ell-1}$ & $\ell-1$ & $0$ \\
      & $\vdots$ & 
    \end{tabular}
  \end{center}
\end{proof}

\section{Commentary and Future Work} 
\label{secFutureWork}

It is routine to bound the part sizes, 
and write polynomials as generating functions.  
Let $P_i(z; M)$ denote the 
generating functions of the CMPP partitions generated by $P_i(z)$ 
in which parts are at most $M$.  Then, for example, 
\begin{align*}
  P_0(z; M) = & \sum_{n_1, n_2, \ldots, n_\ell \geq  0} 
    q^{ N_1^2 + N_2^2 + \cdots + N_\ell^2 + N_1 + N_2 + \cdots + N_\ell } 
    \; z^{N_1} 
    \begin{bmatrix} 
      M - 2n_1 - 2( n_2 + n_3 + \cdots + n_\ell ) 
      \\ n_1 \end{bmatrix} \\ 
    & \quad \times \begin{bmatrix} 
      M - (2n_1 + 4n_2) - 2( n_3 + n_4 + \cdots + n_\ell ) 
      \\ n_2 \end{bmatrix} \\ 
    & \quad \times \cdots \begin{bmatrix} 
      M - (2n_1 + 4n_2 + 6n_3 + \cdots + 2sn_s) - 2( n_{s+1} + n_{s+2} + \cdots + n_\ell ) 
      \\ n_s \end{bmatrix} \\ 
    & \times \cdots \begin{bmatrix} 
      M - (2n_1 + 4n_2 + 6n_3 + \cdots + 2 \ell n_\ell)  
      \\ n_\ell \end{bmatrix}, 
\end{align*}
where the notation is that of Theorem \ref{thmRussellMain}, 
and $\begin{bmatrix} a \\ b \end{bmatrix}$ is the Gaussian polynomial 
$\frac{ (q; q)_a }{ (q; q)_b (q; q)_{a-b} }$~\cite{A_thebluebook}.  
The examplar we happened to choose above 
is a specialization of Warnaar's restricted two-line cylindric partitions 
generating function~\cite{Warnaar23}.  
Similar polynomials due to Warnaar~\cite{Warnaar23} imply Theorem \ref{thmRussellMain} 
upon the connection between two-line cylindric partitions and 
CMPP partitions for $k=1$.  
Please see ~\cite{Russell_starting_point} and two paragraphs below.  

The most obvious follow up problem 
is to find a similar construction for obtaining 
evidently positive generating functions for the cases $k > 1$.  
Evidently positive series for the $[k, 0, \ldots, 0]-$ 
and $[0, \ldots, 0, k]-$admissible CMPP partitions~\cite{CMPP_original} 
are known~\cite{GOW}.  
The construction in this paper does not readily generalize.  
It looks like even the determination of relative heights must be done from scratch.  
There is a conjecture on such a generating series in~\cite{KRTW} for a specific case.  

Russell gives a bijection~\cite{Russell_starting_point} 
between CMPP partitions for $k=1$ 
and two-line cylindric partitions~\cite{GesselKrattenthaler} 
in conjunction with~\cite{JKS}.  
There is also an almost canonical bijection between the said objects per~\cite{KOS}.  
CMPP partitions for $k=1$ can be regarded in terms of two-line cylindric partitions.  
They are two line cylindric partitions in which none of the slices repeat, 
and none of the slices share the same first or the second line.  
In this sense, CMPP partitions for $k = 1$ are reminiscent of 
Rogers-Ramanujan identities~\cite{RR}.  
For $k > 1$, this connection is lost.  
Because \emph{incompatible} slices are forced to appear together.  
The connection to cylindric partitions in the $k>2$, 
if it exists, must be subtle.  
Warnaar [private communication] maintains that this connection does not exist
for representation theoretic reasons.  

The lattice paths in~\cite{FW} are not related to 
the upper envelopes formed by drawing the legs of parts 
in the CMPP partitions.  
In fact, the uppoer envelopes used in this paper 
are not always connected.  

Last, but not least, 
one wonders if it possible to adjust Russell's functional equations 
to include the series in the missing cases, as well.  

\section*{Acknowledgements} 

We thank Shashank Kanade and Matthew Russell for useful discussions 
during the preparation of the manuscript, 
and for mentioning~\cite{FW} and~\cite{KRTW}.  
We also thank S. Ole Warnaar for correcting many inaccuracies 
and pointing out~\cite{GOW} and~\cite{Warnaar25}.  

\bibliographystyle{amsplain}

\begin{thebibliography}{10}

\bibitem{A_RRG}
Andrews, G.E., 1966. 
An analytic proof of the Rogers-Ramanujan-Gordon identities. 
\emph{American Journal of Mathematics}, {\bf 88(4)}, pp.844--846.

\bibitem{AG}
Andrews, G.E., 1974. 
An analytic generalization of the Rogers-Ramanujan identities for odd moduli. 
\emph{Proceedings of the National Academy of Sciences}, {\bf 71(10)}, pp.4082--4085.

\bibitem{Bressoud79}
Bressoud, D.M., 1979. 
A generalization of the Rogers-Ramanujan identities for all moduli. 
\emph{Journal of Combinatorial Theory, Series A}, {\bf 27(1)}, pp.64--68.

\bibitem{Bressoud80}
Bressoud, D.M., 1980. 
An analytic generalization of the Rogers-Ramanujan identities with interpretation. 
\emph{The Quarterly Journal of Mathematics}, {\bf 31(4)}, pp.385--399.

\bibitem{A_thebluebook} 
Andrews, G.E., 1998. 
\emph{The theory of partitions (No. 2)}. 
Cambridge university press.

\bibitem{CMPP_original}
Capparelli, S., Meurman, A., Primc, A. and Primc, M., 2022. 
New partition identities from $C^{(1)} _\ell$-modules. 
\emph{Glasnik matemati\v{c}ki}, {\bf 57(2)}, pp.161--184. 

\bibitem{DK}
Dousse, J. and Konan, I., 2022. 
Characters of level $1$ standard modules of $C_n^{(1)}$ 
as generating functions for generalised partitions. 
arXiv preprint arXiv:2212.12728.

\bibitem{FW}
Foda, O. and Welsh, T.A., 2016. 
Cylindric partitions,$\mathcal{w}_r$ characters and the Andrews–Gordon–Bressoud identities. 
\emph{Journal of Physics A: Mathematical and Theoretical}, {\bf 49(16)}, p.164004.

\bibitem{GR} 
Gasper, G. and Rahman, M., 2011. 
\emph{Basic hypergeometric series (Vol. 96)}. 
Cambridge university press.

\bibitem {GesselKrattenthaler}  
Gessel, I.M. and Krattenthaler, C., 
Cylindric partitions, 
\emph{Trans. Amer. Math. Soc.}, {\bf 349(2)}:429--479, 1997. 

\bibitem{GOW}
Griffin, M.J., Ono, K. and Warnaar, S.O., 2016. 
A framework of Rogers-Ramanujan identities and their arithmetic properties. 
\emph{Duke Mathematical Journal}, {\bf 165(8)}, pp.1475--1527.

\bibitem{RRG}
Gordon, B., 1961. 
A combinatorial generalization of the Rogers-Ramanujan identities. 
\emph{American Journal of Mathematics}, {\bf 83(2)}, pp.393--399.

\bibitem{JKS} 
Jing, N.,Misra K.C., and Savage, C.D.,  2001.  
On multi-color partitions and the generalized Rogers–Ramanujan identities. 
\emph{Communications in Contemporary Mathematics}, {\bf 3(4)}, pp.533--548.

\bibitem{KLRS}
Kanade, S., Lepowsky, J., Russell, M.C. and Sills, A.V., 2017. 
Ghost series and a motivated proof of the Andrews–Bressoud identities. 
\emph{Journal of Combinatorial Theory, Series A}, {\bf 146}, pp.33--62.

\bibitem{KRTW}
Kanade, S., Russell, M.C., Tsuchioka, S. and Warnaar, S.O., 2024. 
Remarks on the conjectures of Capparelli, Meurman, Primc and Primc. 
arXiv preprint arXiv:2404.03851. 

\bibitem{KOS}
Kur\c{s}ung\"{o}z, K. and \"{O}mr\"{u}uzun Seyrek, H., 2025. 
A decomposition of cylindric partitions and cylindric partitions into distinct parts. 
\emph{European Journal of Combinatorics}, {\bf 130}, p.104219.

\bibitem{LW}
Lepowsky, J. and Wilson, R.L., 1985. 
The structure of standard modules: II. The case $A^{(1)}_1$, principal gradation. 
\emph{Inventiones mathematicae}, {\bf 79(3)}, pp.417--442.

\bibitem{PS16}
Primc, M. and \v{S}iki\'{c}, T., 2016. 
Combinatorial bases of basic modules for affine Lie algebras $C^{(1)}_n$. 
\emph{Journal of mathematical physics}, {\bf 57(9)}.

\bibitem{PS19}
Primc, M. and \v{S}iki\'{c}, T., 2019. 
Leading terms of relations for standard modules of the affine Lie algebras $C^{(1)}_n$. 
\emph{The Ramanujan Journal}, {\bf 48(3)}, pp.509--543. 

\bibitem{PT25}
Primc, M. and Trupčević, G., 2025. 
Linear independence for $C^{(1)}_\ell$ by using $C^{(1)}_{2 \ell}$. 
\emph{Journal of algebra}, {\bf 661}, pp.341--356.

\bibitem{RR} 
Ramanujan, S. and Rogers, L.J., 1919. 
Proof of certain identities in combinatory analysis. 
In \emph{Proc. Cambridge Philos. Soc} {\bf 19(214--216)}, p. 3.

\bibitem{Russell_starting_point}
Russell, M.C., 2026. 
Companions to the Andrews-Gordon and Andrews-Bressoud 
Identities and Recent Conjectures of Capparelli, Meurman, Primc, and Primc. 
\emph{SIGMA} {\bf 22(046)}, 38 pages.  
(Contribution to the \emph{Special Issue on 
Recent Advances in Vertex Operator Algebras in honor of James Lepowsky})


\bibitem{T18}
Trup\v{c}evi\'{c}, G., 2018. 
Bases of standard modules for affine Lie algebras of type. 
\emph{Communications in algebra}, {\bf 46(8)}, pp.3663--3673.

\bibitem{Warnaar23}
Warnaar, S.O., 2023. 
The $A_2$ Andrews–Gordon identities and cylindric partitions. 
Transactions of the American Mathematical Society, Series B, 10(22), pp.715-765.

\bibitem{Warnaar25}
Warnaar, S.O. Affine Jacobi-Trudi formulas 
and $q, t$-Rogers-Ramanujan identities. 
arXiv preprint arXiv:2511.17034. 2025 Nov 21.

\end{thebibliography}

\end{document}